\documentclass[10pt]{amsart}
\usepackage{amsthm, amsfonts, amssymb, amsmath, graphicx, enumitem, bbold, tkz-berge, caption}
\usetikzlibrary{decorations.pathreplacing, graphs}
\usepackage{fouriernc}

\newcommand{\into}{\hookrightarrow}

\newcommand{\C}{\mathbb{C}}
\newcommand{\R}{\mathbb{R}}

\newcommand{\Z}{\mathbb{Z}}

\newcommand{\F}{\mathbb{F}}
\newcommand{\K}{\mathbb{K}}
\newcommand{\GL}{\mathrm{GL}}
\newcommand{\bd}{\partial}

\newcommand{\e}{\epsilon}

\newcommand{\la}{\langle}
\newcommand{\ra}{\rangle}
\newcommand{\edge}{\sim}
\mathchardef\mhyphen="2D
\newcommand{\adspec}{\alpha\mhyphen\mathrm{spec}}
\newcommand{\lapspec}{\lambda\mhyphen\mathrm{spec}}

\newcommand{\Nm}{\mathrm{N}}
\newcommand{\Tr}{\mathrm{Tr}}
\newcommand{\bl}{\bullet}
\newcommand{\wh}{\circ}
\newcommand{\ct}{\mathbb{1}}

\theoremstyle{definition}
\newtheorem{thm}{Theorem}[section]
\newtheorem{lem}[thm]{Lemma}
\newtheorem{cor}[thm]{Corollary}
\newtheorem{prop}[thm]{Proposition}

\newtheorem{ex}[thm]{Example}
\newtheorem{Rule}[thm]{Rule}
\newtheorem{exer}[thm]{Exercise}

\newtheorem*{notes}{Notes}

\begin{document}
\title{A brief introduction to Spectral Graph Theory}
\author[]{Bogdan Nica}
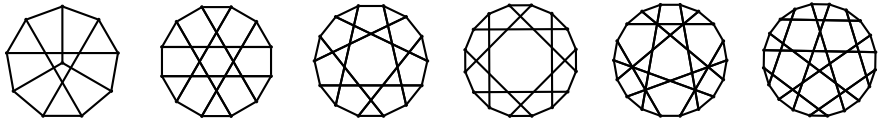
\date{\today}
\maketitle

\begin{figure}[!ht]
\centering
\GraphInit[vstyle=Simple]
\tikzset{VertexStyle/.style = {
shape = circle,
fill = black,
inner sep = 0pt,
draw}}
\begin{minipage}[b]{0.15\linewidth}
\centering
\begin{tikzpicture}[rotate=10, scale=.15]
\grPetersen[form=3,RA=5]
\end{tikzpicture}
\end{minipage}
\begin{minipage}[b]{.16\linewidth}
\centering
\begin{tikzpicture}[rotate=15,scale=.15]
\grFranklin[Math,RA=5]
\end{tikzpicture}
\end{minipage}
\begin{minipage}[b]{.15\linewidth}
\centering
\begin{tikzpicture}[scale=.15]
\grLCF[RA=5]{5,9}{7}
\end{tikzpicture}
\end{minipage}
\begin{minipage}[b]{.15\linewidth}
\centering
\begin{tikzpicture}[rotate=-11, scale=.15]
\grMobiusKantor[RA=5]
\end{tikzpicture}
\end{minipage}
\begin{minipage}[b]{.15\linewidth}
\centering
\begin{tikzpicture}[scale=.15, rotate=40]
\grPappus[RA=5]
\end{tikzpicture}
\end{minipage}
\begin{minipage}[b]{.15\linewidth}
\centering
\begin{tikzpicture}[scale=.15, rotate=188]
\grDesargues[Math,RA=5]
\end{tikzpicture}
\end{minipage}
\end{figure}

\bigskip
\section*{-- Introduction --}

Spectral graph theory starts by associating matrices to graphs, notably, the adjacency matrix and the laplacian matrix. The general theme is then, firstly, to compute or estimate the eigenvalues of such matrices, and secondly, to relate the eigenvalues to structural properties of graphs. As it turns out, the spectral perspective is a powerful tool. Some of its loveliest applications concern facts that are, in principle, purely graph-theoretic or combinatorial. To give just one example, spectral ideas are a key ingredient in the proof of the so-called Friendship Theorem: if, in a group of people, any two persons have exactly one common friend, then there is a person who is everybody's friend.

This text is an introduction to spectral graph theory, but it could also be seen as an invitation to algebraic graph theory. On the one hand, there is, of course, the linear algebra that underlies the spectral ideas in graph theory. On the other hand, most of our examples are graphs of algebraic origin. The two recurring sources are Cayley graphs of groups, and graphs built out of finite fields. In the study of such graphs, some further algebraic ingredients (e.g., characters) naturally come up.

The table of contents gives, as it should, a good glimpse of where is this text going. Very broadly, the first half is devoted to graphs, finite fields, and how they come together. This part is meant as an appealing and meaningful motivation. It provides a context that frames and fuels much of the second, spectral, half. 

Most sections have one or two exercises. Their position within the text is a hint. The exercises are optional, in the sense that virtually nothing in the main body depends on them. But the exercises are often of the non-trivial variety, and they should enhance the text in an interesting way. The hope is that the reader will enjoy them.

We assume a basic familiarity with linear algebra, finite fields, and groups, but not necessarily with graph theory. This, again, betrays our algebraic perspective. 

This text is based on a course I taught in G\"ottingen, in the Fall of 2015. I would like to thank Jerome Baum for his help with some of the drawings. The present version is preliminary, and comments are welcome (email: bogdan.nica@gmail.com).

\newpage
 \tableofcontents

 \newpage
\section{Graphs}

\bigskip
\subsection*{Notions}\label{sec: some definitions}\
\medskip

A \emph{graph} consists of vertices, or nodes, and edges connecting pairs of vertices. 

Throughout this text, graphs are \emph{finite} (there are finitely many vertices), \emph{undirected} (edges can be traversed in both directions), and \emph{simple} (there are no loops or multiple edges). Unless otherwise mentioned, graphs are also assumed to be \emph{connected} (there is a sequence of edges linking any vertex to any other vertex) and \emph{non-trivial} (there are at least three vertices). On few occasions the singleton graph, $\circ$, the `marriage graph', $\circ\!\!-\!\!\circ$, or disconnected graphs will come up. By and large, however, when we say `a graph' we mean that the above conditions are fulfilled.

A \emph{graph isomorphism} is a bijection between the vertex sets of two graphs such that two vertices in one graph are adjacent if and only if the corresponding two vertices in the other graph are adjacent. Isomorphic graphs are viewed as being the same graph.

Within a graph, two vertices that are joined by an edge are said to be \emph{adjacent}, or \emph{neighbours}. The \emph{degree} of a vertex counts the number of its neighbours. A graph is \emph{regular} if all its vertices have the same degree. 
 
Basic examples of regular graphs include the \emph{complete graph} $K_n$, the \emph{cycle graph} $C_n$, the \emph{cube graph} $Q_n$. See Figure~\ref{fig: KCQ}. The complete graph $K_n$ has $n$ vertices, with edges connecting any pair of distinct vertices. Hence $K_n$ is regular of degree $n-1$. The cycle graph $C_n$ has $n$ vertices, and it is regular of degree $2$. The cube graph $Q_n$ has the binary strings of length $n$ as vertices, with edges connecting two strings that differ in exactly one slot. Thus $Q_n$ has $2^n$ vertices, and it is regular of degree $n$. 

\begin{figure}[!ht]\bigskip
\GraphInit[vstyle=Simple]
\tikzset{VertexStyle/.style = {
shape = circle,
inner sep = 0pt,
minimum size = .8ex,
draw}}
\begin{minipage}[b]{0.32\linewidth}
\centering
\begin{tikzpicture}[rotate=90, scale=.6]
  \grComplete[RA=2]{5}
\end{tikzpicture}
\end{minipage}
\begin{minipage}[b]{0.32\linewidth}
\centering
\begin{tikzpicture}[rotate=90, scale=.6]
  \grCycle[RA=2]{5}
\end{tikzpicture}
\end{minipage}
\begin{minipage}[b]{0.32\linewidth}
\centering
\begin{tikzpicture}[rotate=45, scale=.5]\grPrism[RA=3,RB=1.5]{4}%
\end{tikzpicture}
\end{minipage}
\caption{Complete graph $K_5$. Cycle graph $C_5$. Cube graph $Q_3$.}\label{fig: KCQ}
\end{figure}
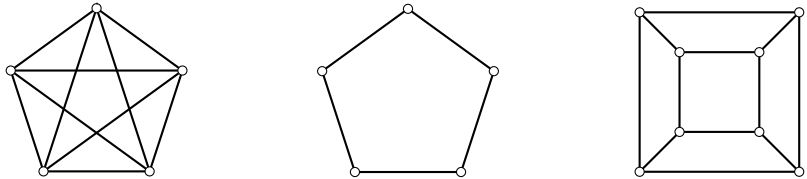

Another example of a regular graph is the Petersen graph, a $3$-regular graph on $10$ vertices pictured in Figure~\ref{fig: petersen}. 

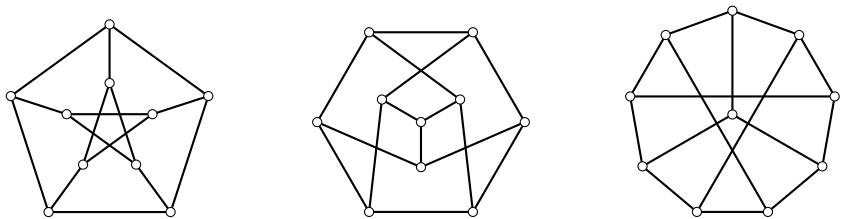
\begin{figure}[!ht]\bigskip
\GraphInit[vstyle=Simple]
\tikzset{VertexStyle/.style = {
shape = circle,
inner sep = 0pt,
minimum size = .8ex,
draw}}
\begin{minipage}[b]{0.32\linewidth}
\centering
\begin{tikzpicture}[rotate=90,scale=.6]
  \grPetersen[form=1, RA=2.3,RB=1]
 \end{tikzpicture}
\end{minipage}
\begin{minipage}[b]{0.32\linewidth}
\centering
\begin{tikzpicture}[scale=.6]
  \grPetersen[form=2,RA=2.3,RB=1]
 \end{tikzpicture}
\end{minipage}
\begin{minipage}[b]{0.32\linewidth}
\centering
\begin{tikzpicture}[rotate=10, scale=.6]
\grPetersen[form=3,RA=2.3]
\end{tikzpicture}
\end{minipage}
\caption{Three views of the Petersen graph.}\label{fig: petersen}
\end{figure}

A formal description of the Petersen graph runs as follows: the vertices are the $2$-element subsets of a $5$-element set, and edges represent the relation of being disjoint. The Petersen graph is, in many ways, the smallest interesting graph.

Notable families of non-regular graphs are the path graphs, the star graphs, the wheel graphs, the windmill graphs - all illustrated in Figure~\ref{fig: irregular}.

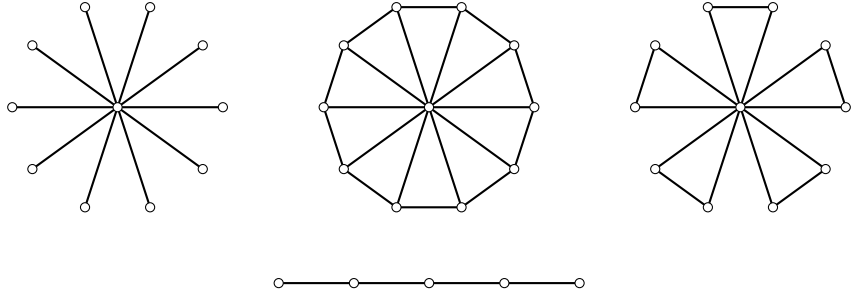
\begin{figure}[!ht]\bigskip
\centering
\GraphInit[vstyle=Simple]
\tikzset{VertexStyle/.style = {
shape = circle,
inner sep = 0pt,
minimum size = .8ex,
draw}}

\begin{minipage}[b]{0.32\linewidth}
\centering
\begin{tikzpicture}
  \grStar[RA=1.4]{11}
\end{tikzpicture}
\end{minipage}
\begin{minipage}[b]{0.32\linewidth}
\centering
\begin{tikzpicture}
\grWheel[RA=1.4]{11}
\end{tikzpicture}
\end{minipage}
\begin{minipage}[b]{0.32\linewidth}
\centering
\begin{tikzpicture}
  \grStar[RA=1.4]{11}
  \Edges(a0, a1)
  \Edges(a2, a3)
  \Edges(a4, a5)
  \Edges(a6, a7)
  \Edges(a8, a9)
\end{tikzpicture}
\end{minipage}

\bigskip
\bigskip
\begin{tikzpicture}
\grPath[RA=1,RS=1]{5}
\end{tikzpicture}
\caption{Star graph. Wheel graph. Windmill graph. Path graph.}\label{fig: irregular}
\end{figure}

We remark that the marriage graph $\circ\!\!-\!\!\circ$ is a degenerate case for several of the families mentioned above. From now on we denote it by $K_2$.

We usually denote a graph by $X$. We write $V$ and $E$ for the vertex set, respectively the edge set. The two fundamental parameters of a graph are
\begin{itemize}[leftmargin=50pt]
\item[$n$ :] number of vertices,
\item[$d$ :] maximal vertex degree.
\end{itemize}
Our non-triviality assumption means that $n\geq 3$ and $d\geq 2$. If $X$ is regular, then $d$ is simply the degree of every vertex. In general, the degree is a map on the vertices, $v\mapsto \mathrm{deg}(v)$. 

\begin{prop} Let $X$ be a graph. Then the following holds:
\begin{align*}
\sum_{v\in V} \mathrm{deg}(v)=2|E|
\end{align*}
In particular, $dn\geq 2|E|$ with equality if and only if $X$ is regular.
\end{prop}

A \emph{path} in a graph connects some vertex $v$ to some other vertex $w$ by a sequence of adjacent vertices $v=v_0 \edge v_1 \edge \ldots \edge v_k=w$. A path is allowed to contain backtracks, i.e., steps of the form $u\edge u'\edge u$. The concept of a path was, in fact, implicit when we used the term `connected', since the defining property of a connected graph is that any two vertices can be joined by a path. A path from a vertex $v$ to itself is said to be a closed path based at $v$.  A \emph{cycle} is a non-trivial closed path with distinct intermediate vertices. Here, non-triviality means that we rule out empty paths, and back-and-forths across edges.  

The \emph{length} of a path is given by the number of edges, counted with multiplicity. An integral-valued \emph{distance} on the vertices of a graph can be defined as follows: the distance $\mathrm{dist}(v,w)$ between two vertices $v$ and $w$ is the minimal length of a path from $v$ to $w$. So metric concepts such as diameter, and spheres or balls around a vertex, make sense. 

\begin{exer}\label{exer: mantel} A graph on $n$ vertices that contains no $3$-cycles has at most $\lfloor n^2/4\rfloor$ edges.
\end{exer}

A \emph{tree} is a graph that has no cycles. For instance, star graphs and path graphs are trees. Two important examples are the trees $T_{d,R}$ and $\tilde{T}_{d,R}$, described as follows. There is a root vertex of degree $d-1$ in $T_{d,R}$, respectively of degree $d$ in $\tilde{T}_{d,R}$; the pendant vertices lie on a sphere of radius $R$ about the root; the remaining intermediate vertices all have degree $d$. See Figure~\ref{fig: comparing trees} for an illustration.

\begin{figure}[!ht]\bigskip
\centering
\GraphInit[vstyle=Classic]
\tikzset{VertexStyle/.style = {
shape = circle,
inner sep = 0pt,
minimum size = .8ex,
draw}}
\begin{minipage}[b]{0.48\linewidth}
\centering
\begin{tikzpicture}[scale=.7]
\SetVertexNoLabel
\Vertex[x=1.5,y=1]{A}
\Vertex[x=2,y=1]{B}
\Vertex[x=2.5,y=1]{C}
\Vertex[x=3,y=1]{D}
\Vertex[x=7,y=1]{a}
\Vertex[x=7.5,y=1]{b}
\Vertex[x=8,y=1]{c}
\Vertex[x=8.5,y=1]{d}
\Vertex[x=2,y=1.5]{E}
\Vertex[x=3,y=1.5]{F}
\Vertex[x=3,y=2.5]{G}
\Vertex[x=7,y=1.5]{e}
\Vertex[x=8,y=1.5]{f}
\Vertex[x=7,y=2.5]{g}
\Vertex[x=5,y=4.5]{R}
\Edges(R, G, E, A)
\Edges(G, F, D)
\Edges(F, C)
\Edges(E, B)
\Edges(R, g, e, a)
\Edges(g, f, d)
\Edges(f, c)
\Edges(e, b)
\end{tikzpicture}
\end{minipage}
\begin{minipage}[b]{0.48\linewidth}
\centering
\begin{tikzpicture}[scale=.7]
\SetVertexNoLabel
\Vertex[x=1.5,y=1]{A}
\Vertex[x=2,y=1]{B}
\Vertex[x=2.5,y=1]{C}
\Vertex[x=3,y=1]{D}
\Vertex[x=7,y=1]{a}
\Vertex[x=7.5,y=1]{b}
\Vertex[x=8,y=1]{c}
\Vertex[x=8.5,y=1]{d}
\Vertex[x=2,y=1.5]{E}
\Vertex[x=3,y=1.5]{F}
\Vertex[x=3,y=2.5]{G}
\Vertex[x=7,y=1.5]{e}
\Vertex[x=8,y=1.5]{f}
\Vertex[x=7,y=2.5]{g}
\Vertex[x=5,y=4.5]{R}
\Vertex[x=5, y=2.5]{u}
\Vertex[x=4.5, y=1.5]{v}
\Vertex[x=5.5, y=1.5]{w}
\Vertex[x=4.5, y=1]{x}
\Vertex[x=4, y=1]{y}
\Vertex[x=5.5, y=1]{z}
\Vertex[x=6, y=1]{t}
\Edges(R, G, E, A)
\Edges(G, F, D)
\Edges(F, C)
\Edges(E, B)
\Edges(R, g, e, a)
\Edges(g, f, d)
\Edges(f, c)
\Edges(e, b)
\Edges(R, u)
\Edges(v, u)
\Edges(w, u)
\Edges(v, x)
\Edges(w, z)
\Edges(v, y)
\Edges(w, t)
\end{tikzpicture}
\end{minipage}
\caption{$T_{3,3}$ and $\tilde{T}_{3,3}$.}\label{fig: comparing trees}
\end{figure}
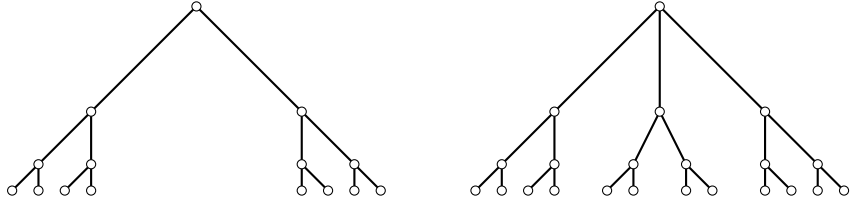

Let $X_1$ and $X_2$ be two graphs with vertex sets $V_1$, respectively $V_2$. The \emph{product graph} $X_1\times X_2$ has vertex set $V_1\times V_2$, and edges defined as follows: $(u_1,u_2)\edge (v_1,v_2)$ if either $u_1\sim v_1$ and $u_2=v_2$, or $u_1=v_1$ and $u_2\sim v_2$. Let us point out that there are several reasonable ways of defining a product of two graphs. The type of product that we have defined is, however, the only one we will use.

For example, the $n$-fold product $K_2\times\dots\times K_2$ is the cube graph $Q_n$. Another product we will encounter is $K_n\times K_n$. One can think of $K_n\times K_n$ as describing the possible moves of a rook on an $n\times n$ chessboard.

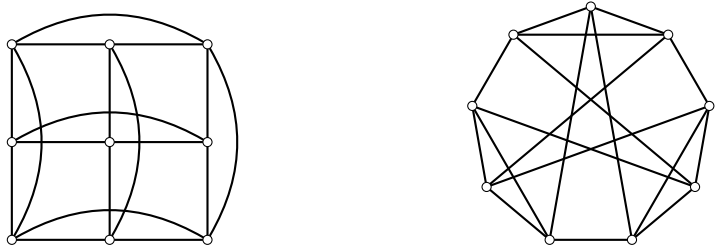
\begin{figure}[!ht]\bigskip
\centering
\GraphInit[vstyle=Simple]
\tikzset{VertexStyle/.style = {
shape = circle,
inner sep = 0pt,
minimum size = .8ex,
draw}}
\begin{minipage}[b]{0.48\linewidth}
\centering
\begin{tikzpicture}[scale=1.3]
\SetVertexNoLabel
\Vertex{A}
\EA(A){B} \WE(A){C} \NO(A){D} \SO(A){E}
\NOEA(A){F} \NOWE(A){G} \SOEA(A){H} \SOWE(A){I}
\Edges(G,D,F)
\Edges(C,A,B)
\Edges(I,E,H)
\Edges(G,C,I)
\Edges(D,A, E)
\Edges(F,B,H)
\Edge[style={bend left}](I)(H)
\Edge[style={bend left}](C)(B)
\Edge[style={bend left}](G)(F)
\Edge[style={bend left}](G)(I)
\Edge[style={bend left}](F)(H)
\Edge[style={bend left}](D)(E)
\end{tikzpicture}
\end{minipage}
\begin{minipage}[b]{0.48\linewidth}
\centering
\begin{tikzpicture}[rotate=50, scale=.8]
\SetVertexNoLabel
  \grEmptyCycle[RA=2]{9}
  \Edges(a0, a1, a2, a0)
  \Edges(a3, a4, a5, a3)
  \Edges(a6, a7, a8, a6)
  \Edges(a0, a4, a8, a0)
  \Edges(a1, a5, a6, a1)
  \Edges(a2, a3, a7, a2)
\end{tikzpicture}
\end{minipage}
\caption{Two views of $K_3\times K_3$.}
\end{figure}

How to tell two graphs apart? Obstructions to graph isomorphism are provided, globally, by size, and locally by vertex degrees. Such counting criteria are, however, quite naive. For example, they are useless in distinguishing the $3$-regular graphs on $8$ vertices that are pictured in Figure~\ref{fig: linking}. 

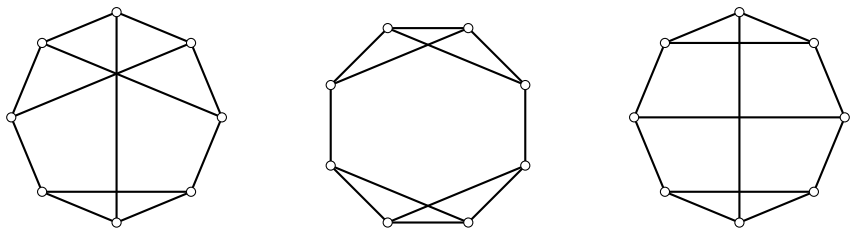
\begin{figure}[!ht]\bigskip
\centering
\GraphInit[vstyle=Simple]
\tikzset{VertexStyle/.style = {
shape = circle,
inner sep = 0pt,
minimum size = .8ex,
draw}}
\begin{minipage}[b]{0.32\linewidth}
\centering
\begin{tikzpicture}
  \grCycle[RA=1.4]{8}
  \Edges(a0, a3)
  \Edges(a1, a4)
  \Edges(a2, a6)
  \Edges(a5, a7)
\end{tikzpicture}
\end{minipage}
\begin{minipage}[b]{0.32\linewidth}
\centering
\begin{tikzpicture}[rotate=22.5]
  \grCycle[RA=1.4]{8}
  \Edges(a0, a2)
  \Edges(a4, a6)
  \Edges(a7, a5)
 \Edges(a3, a1)
\end{tikzpicture}
\end{minipage}
\begin{minipage}[b]{0.32\linewidth}
\centering
\begin{tikzpicture}
  \grCycle[RA=1.4]{8}
  \Edges(a2, a6)
  \Edges(a7, a5)
  \Edges(a3, a1)
  \Edges(a4, a0)
\end{tikzpicture}
\end{minipage}

\caption{Three $3$-regular graphs on $8$ vertices.}\label{fig: linking}
\end{figure}

So let us enhance the local information offered by vertex degrees. Instead of just counting the neighbours of a vertex, we consider their structure. Given a graph, the \emph{link graph} (or the \emph{neighbourhood graph}) of a vertex has all its neighbours as vertices, and edges inherited from the ambient graph. A graph isomorphism is link-preserving, in the sense that it induces a graph isomorphism between the corresponding link graphs.

As an illustration of this idea, let us consider again the graphs in Figure~\ref{fig: linking}, and let us record the pattern of link graphs. The first two graphs have, both, two types of link graphs, $\circ\!\!-\!\!\circ\; \circ$ and $\circ\;\circ\; \circ$. The number of vertices having each type is, however, different. The third graph has link graphs $\circ\!\!-\!\!\circ\; \circ$ and $\circ\!\!-\!\!\circ\!\!-\!\!\circ$. In conclusion, the graphs in Figure~\ref{fig: linking} are mutually non-isomorphic.

\begin{exer}\label{exer: frucht} Show that the regular graph in Figure~\ref{fig: frucht} has no non-trivial automorphisms.
\end{exer}

\begin{figure}[!ht]\medskip
\centering
\GraphInit[vstyle=Simple]
\tikzset{VertexStyle/.style = {
shape = circle,
inner sep = 0pt,
minimum size = .8ex,
draw}}
\begin{tikzpicture}
  \grCycle[RA=1.5]{12}
  \Edges(a0, a2)
  \Edges(a4, a6)
  \Edges(a7, a9)
  \Edges(a3, a11)
  \Edges(a5, a10)
 \Edges(a1, a8)
\end{tikzpicture}
\caption{The Frucht graph.}\label{fig: frucht}
\end{figure}
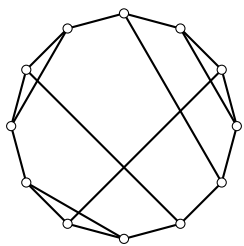

\begin{notes} The word `graph' was coined by Sylvester in 1878, as an abbreviation of `graphical notation'. This was a well-established term among chemists at the time, and it referred to depictions of molecules in which nodes represented atoms, and edges indicated bonds between them.
\end{notes}

\bigskip
\subsection*{Bipartite graphs}\label{sec: bipartite}\
\medskip

A graph is said to be \emph{bipartite} if its vertex set can be partitioned into two subsets, such that no two points belonging to the same subset are connected by an edge. We note that, if it exists, such a bipartition is unique. We often think of the two subsets as consisting of `white', respectively `black' vertices. 

For example, the cube graph $Q_n$ is bipartite. Indeed, we may partition the binary strings by the weight parity, where the \emph{weight} of a binary string is the number of entries equal to $1$ - in other words, the distance to the all $0$'s string. Also trees are bipartite: a choice of a base vertex defines a partition according to the parity of the distance to the chosen vertex.  

The \emph{complete bipartite graph} $K_{m,n}$ is the bipartite counterpart of a complete graph $K_n$. The vertices of $K_{m,n}$ are two disjoint sets of size $m$, respectively $n$, with edges connecting each element of one set to every element of the other set.

\begin{figure}[!ht]\medskip
\centering
\GraphInit[vstyle=Simple]
\tikzset{VertexStyle/.style = {
shape = circle,
inner sep = 0pt,
minimum size = .9ex,
draw}}
\begin{tikzpicture}[scale=.4]
\grCompleteBipartite[RA=4,RB=4,RS=5]{4}{4}
 \end{tikzpicture}
 \caption{Complete bipartite graph $K_{4,4}$.}
 \end{figure}
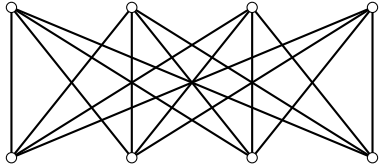
 
The complete graph $K_n$ is not bipartite, as it contains $3$-cycles. A cycle graph $C_n$ is bipartite if and only if $n$ is even. In fact, bipartiteness can be characterized as follows.
 
\begin{thm}
 A graph is bipartite if and only if it contains no cycles of odd length.
\end{thm}
 
 \begin{proof}
If a graph is bipartite, then a path of odd length has its endpoints in different sides of the bi-partition - in particular it cannot be a cycle. For the other direction, pick a vertex $u$. If all cycles have even length then all closed paths have even length, which in turn implies that lengths of paths from a fixed vertex $u$ to any other vertex $v$ have a well-defined parity. This induces a partition of the vertex set into two types, the `odd' and the `even' vertices with respect to $u$, and there are no edges between vertices of the same type.
 \end{proof}
 
 \begin{exer}\label{exer: mantel2} A graph on $n$ vertices that contains no $3$-cycles has $\lfloor n^2/4\rfloor$ edges. What is the graph?
\end{exer}

There is a canonical way of passing from non-bipartite graphs to bipartite ones. The \emph{bipartite double} of a graph $X$ is the graph whose vertex set consists of two disjoint copies, $V_\bl$ and $V_\wh$, of the vertex set of $X$, and two vertices $u_\bl$ and $v_\wh$ are adjacent if and only if $u$ and $v$ are adjacent in $X$. Thus, we get two edges, $u_\bl\sim v_\wh$ and $u_\wh\sim v_\bl$, for each edge $u\sim v$ in $X$. 

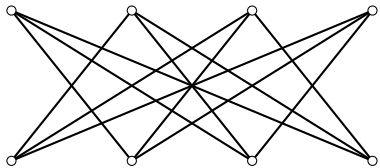
\begin{figure}[!ht]\medskip
\centering\GraphInit[vstyle=Simple]
\tikzset{VertexStyle/.style = {
shape = circle,
inner sep = 0pt,
minimum size = .8ex,
draw}}
\begin{tikzpicture}[scale=.4]
\grCrown[RA=4,RS=5]{4}
\end{tikzpicture}
\caption{Bipartite double of $K_4$.}\label{fig: bidoubleK4}
\end{figure}

For example, the bipartite double of $K_4$ is depicted in Figure~\ref{fig: bidoubleK4}. It resembles $K_{4,4}$, except that $4$ edges given by a black-white pairing have been removed. After unwrapping the tangle, it turns out to be the cube graph $Q_3$. 

The bipartite double of a cycle $C_n$ is the doubled cycle $C_{2n}$ if $n$ is odd, respectively twice the cycle $C_{n}$ if $n$ is even. This is an instance of a general phenomenon.
 
\begin{prop}
The bipartite double of a graph $X$ is connected if and only if $X$ is non-bipartite.
\end{prop}

\begin{proof}
A vertex $u_\bl$ is connected to $v_\bl$, respectively to $v_\wh$ in the bipartite double of $X$, if and only if $u$ is connected to $v$ in $X$ by a path of even, respectively odd length. Thus, if the bipartite double is connected, then a path from $u_\bl$ to $v_\bl$ and one from $u_\bl$ to $v_\wh$ yield a close path of odd length in $X$, so $X$ is non-bipartite. Conversely, assume that $X$ is non-bipartite, so $X$ contains some cycle of odd length. Pick a vertex $u$ on the cycle. We claim that, in $X$, we can join $u$ to any other vertex $v$ by a path of even length, as well as by one of odd length. Indeed, take a path $p$ from $u$ to $v$. Then $p$ and the path $p'$, obtained by first going once around the cycle and then following $p$, are two paths from $u$ to $v$ of different parities. In the bipartite double, this means that $u_\bl$ is connected to $v_\bl$ and to $v_\wh$, no matter what $v$ is. Thus, the bipartite double is connected.
\end{proof}

The upshot is that we will only take bipartite doubles of non-bipartite graphs.

But we can also ask the reverse question: given a bipartite graph, is it the bipartite double of some other graph? A first exercise in this direction, left to the reader, is that a tree is not a bipartite double.

\bigskip
\section{Invariants}
A graph invariant is some property of a graph, most commonly a number, that is preserved by isomorphisms. There is a great variety of numerical invariants that one can associate to a graph, besides the fundamental parameters $n$ (the size) and $d$ (the maximal degree). Among the most important, and the ones we are interested herein, are the chromatic number, the independence number, the diameter, the girth, and the isoperimetric constant. Roughly speaking, the first two measure the freeness, the next two the largeness, and the last one the connectivity of a graph.

\bigskip
\subsection*{Chromatic number and independence number}\label{sec: chr-ind}\
\medskip

The \emph{chromatic number}, denoted $\chi$, is the minimal number of colours needed to paint the vertices in such a way that adjacent vertices have different colours. The \emph{independence number}, denoted $\iota$, is the maximal number of vertices that can be chosen so that no two of them are adjacent. A set of vertices is said to be \emph{independent} if no two vertices are adjacent.

The chromatic number satisfies $2\leq \chi \leq n$. At the extremes, $\chi=2$ characterizes bipartite graphs whereas $\chi=n$ characterizes the complete graph $K_n$. The independence number satisfies $1\leq \iota \leq n-1$. The lowest value $\iota=1$ characterizes the complete graph $K_n$, while the highest value $\iota=n-1$ characterizes the star graph on $n$ vertices.

\begin{align*}
\begin{array}{rccccc}
 & \quad K_n \quad &  \quad C_n \quad  & \quad Q_n \quad & \quad K_{n,n} \quad & \textrm{Petersen}   \\[.4em]\cline{2-6}\\
  \textrm{chromatic}\quad\chi \quad & n & 2 \textrm{ or } 3  & 2 & 2 & 3\\[.4em]
 \textrm{independence}\quad \iota \quad & 1 & \lfloor n/2\rfloor  & 2^{n-1} & n & 4
\end{array}
\end{align*}
\smallskip

The chromatic number and the independence number are related by the following fact.

\begin{prop}
We have $\chi\cdot \iota\geq n$.
\end{prop}

\begin{proof}
Consider a colouring with $\chi$ colours. Then the vertex set is partitioned into $\chi$ monochromatic sets. Furthermore, each monochromatic set of vertices is independent so its size is at most $\iota$. 
\end{proof}

The obvious bound $\chi \leq n$ can be improved as follows.

\begin{prop}\label{prop: chromatic bound}
We have $\chi\leq d+1$.
\end{prop}

\begin{proof} We argue by induction on $n$. Consider the (possibly disconnected) graph obtained after deleting some vertex. It can be coloured by using $d+1$ colours. The deleted vertex has at most $d$ neighbours, so at least one more colour is available to complete the colouring.
\end{proof}

The previous proposition is sharp: complete graphs and odd cycles satisfy $\chi=d+1$. Nevertheless, it can be strengthened as follows.

\begin{thm}[Brooks]\label{thm: brooks}
We have $\chi\leq d$ except for complete graphs and odd cycles.
\end{thm}

\begin{exer}\label{exer: irregular chr} Show that $\chi\leq d$ for irregular graphs.
\end{exer}

\begin{notes}
The chromatic number is the oldest graph invariant. Much of the early work in graph theory was motivated by a cartographic observation that we now call the Four-Colour Theorem.

Theorem~\ref{thm: brooks} is due to Brooks (\emph{On colouring the nodes of a network}, Proc. Cambridge Philos. Soc. 1941). The original proof, but also the modern approaches, are somewhat involved. Exercise~\ref{exer: irregular chr} covers the easy case when the graph is not regular.
\end{notes}

\bigskip
\subsection*{Diameter and girth}\label{sec: diam-gir}\
\medskip

The \emph{diameter}, denoted $\delta$, is the maximal distance between vertices. The \emph{girth}, denoted $\gamma$, is the minimal length of a cycle.

The diameter satisfies $1\leq \delta\leq n-1$. At the extremes, $\delta=1$ characterizes the complete graph $K_n$, and $\delta=n-1$ characterizes the path on $n$ vertices. The girth satisfies $3\leq \gamma\leq n$, unless we are dealing with a tree in which case we put $\gamma=\infty$. The value $\gamma=n$ characterizes the cycle $C_n$. 

\begin{align*}
\begin{array}{rccccc}
 & \quad K_n \quad &  \quad C_n \quad  & \quad Q_n \quad & \quad K_{n,n} \quad & \textrm{Petersen}  \\[.4em]\cline{2-6}\\
 \textrm{diameter}\quad\delta \quad & 1 & \lfloor n/2\rfloor & n & 2 & 2\\[.4em]
 \textrm{girth}\quad\gamma \quad  & 3 & n & 4 & 4 & 5
\end{array}
\end{align*}
\smallskip

For graphs which are not trees, the diameter and the girth are related as follows.

\begin{prop}
We have $\gamma\leq 2\delta+1$. If $\gamma$ is even, then actually $\gamma\leq 2\delta$.
\end{prop}

\begin{proof}
We have to show that $\delta\geq \lfloor \gamma/2\rfloor$. Assume that, on the contrary, $\delta< \lfloor \gamma/2\rfloor$. Consider a cycle of length $\gamma$ in the graph, and pick two vertices $u$ and $v$ at distance $\lfloor \gamma/2\rfloor$ as measured along the cycle. On the other hand, $u$ and $v$ can be connected by a path of length at most $\delta$. So we have two paths from $u$ to $v$ whose lengths are equal to $\lfloor \gamma/2\rfloor$, respectively less than $\lfloor \gamma/2\rfloor$. In particular, the two paths are distinct. It follows that there is a cycle whose length is less than $2\lfloor \gamma/2\rfloor$. As $2\lfloor \gamma/2\rfloor\leq \gamma$, this is a contradiction. 
\end{proof}

The next result gives bounds for diameter and the girth in terms of the fundamental parameters. 

\begin{thm}\label{thm: logarithmic} 
Assume $d\geq 3$. Then we have
\begin{align*}
\delta> \frac{\log n}{\log(d-1)}-2.
\end{align*} 
For a regular graph, we also have
\begin{align*}
\gamma<2\:\frac{\log n}{\log(d-1)}+2.
\end{align*}
\end{thm}

\begin{proof}
Let $u$ be a fixed vertex. Then balls around $u$ grow at most exponentially with respect to the radius:
\begin{align*}
|B_r(u)|\leq 1+d+d(d-1)+\ldots+d(d-1)^{r-1}<d (d-1)^r
\end{align*}
If $r=\delta$ then the ball $B_r(u)$ is the entire vertex set, so $n<d(d-1)^\delta$. Therefore $\delta > \log_{d-1} n-\log_{d-1} d$. As $\log_{d-1} d<2$, we get the claimed lower bound on $\delta$.

Let us turn to the girth bound for regular graphs. Note, incidentally, that the girth is well-defined: a regular graph cannot be a tree. We exploit the idea that, for a certain short range, balls around $u$ achieve exponential growth in the radius. Namely, as long as $2r<\gamma$ we have
\begin{align*}
|B_r(u)|= 1+d+d(d-1)+\ldots+d(d-1)^{r-1}> (d-1)^r.
\end{align*}
Take $r$ largest possible subject to $2r<\gamma$. Then $\gamma\leq 2r+2$ while $n\geq |B_r(u)|>(d-1)^r$, in other words $r<\log_{d-1} n$. The claimed upper bound for the girth follows. 
\end{proof}

The bounds of the previous theorem become particularly appealing when viewed in an asymptotic light. A \emph{$d$-regular family} is an infinite sequence of $d$-regular graphs $\{X_k\}$ with $|X_k|\to \infty$. For a $d$-regular family with $d\geq 3$, Theorem~\ref{thm: logarithmic} says that the diameter growth is at least logarithmic in the size, $\delta_k\gg \log |X_k|$, while the girth growth is at most logarithmic in the size, $\gamma_k\ll \log |X_k|$. Here, we are using the following growth relation on sequences: $a_k\ll b_k$ if there is a constant $C$ such that $a_k\leq C b_k$ for all $k$.

At this point, it is natural to ask about extremal families. A $d$-regular family $\{X_k\}$ is said to have \emph{small diameter} if  $\delta_k\asymp \log |X_k|$, respectively \emph{large girth} if $\gamma_k\asymp \log |X_k|$. The growth equivalence of sequences is defined as follows: $a_k\asymp b_k$ if and only if $a_k\ll b_k$ and $a_k\gg b_k$.

\begin{ex}\label{ex: sd}
Consider the rooted tree $\tilde{T}_{3,k}$, for $k\geq 3$. Most vertices have degree $3$, and we can actually turn it into a $3$-regular graph by adding edges between pendant vertices. Indeed, the pendant vertices have a natural partition into $4$-tuples, induced by the vertices that lie two levels up. Let $X_k$ be the graph obtained by turning every $4$-tuple into a $4$-cycle. It is easy to check that the $3$-regular graph $X_k$ has diameter $2k$, and $3\cdot 2^k -2$ vertices. Thus $\{X_k\}$ is a regular family of small diameter.
\end{ex}

\begin{figure}[!ht]\medskip
\GraphInit[vstyle=Classic]
\tikzset{VertexStyle/.style = {
shape = circle,
inner sep = 0pt,
minimum size = .8ex,
draw}}
\begin{tikzpicture}[scale=.9]
\SetVertexNoLabel
\Vertex[x=1,y=.5]{A}
\Vertex[x=2,y=1]{B}
\Vertex[x=2.5,y=1]{C}
\Vertex[x=3,y=.5]{D}
\Vertex[x=7,y=.5]{a}
\Vertex[x=7.5,y=1]{b}
\Vertex[x=8,y=1]{c}
\Vertex[x=9,y=.5]{d}
\Vertex[x=2,y=1.5]{E}
\Vertex[x=3,y=1.5]{F}
\Vertex[x=3,y=2.5]{G}
\Vertex[x=7,y=1.5]{e}
\Vertex[x=8,y=1.5]{f}
\Vertex[x=7,y=2.5]{g}
\Vertex[x=5,y=4.5]{R}
\Vertex[x=5, y=2.5]{u}
\Vertex[x=4.5, y=1.5]{v}
\Vertex[x=5.5, y=1.5]{w}
\Vertex[x=4.65, y=1]{x}
\Vertex[x=4, y=.5]{y}
\Vertex[x=5.35, y=1]{z}
\Vertex[x=6, y=.5]{t}
\Edges(R, G, E, A)
\Edges(G, F, D)
\Edges(F, C)
\Edges(E, B)
\Edges(R, g, e, a)
\Edges(g, f, d)
\Edges(f, c)
\Edges(e, b)
\Edges(R, u)
\Edges(v, u)
\Edges(w, u)
\Edges(v, x)
\Edges(w, z)
\Edges(v, y)
\Edges(w, t)
\Edges(A, B, C, D, A)
\Edges(x, y, t, z, x)
\Edges(a, b, c, d, a)
\end{tikzpicture}
\caption{The graph $X_3$.}
\end{figure}
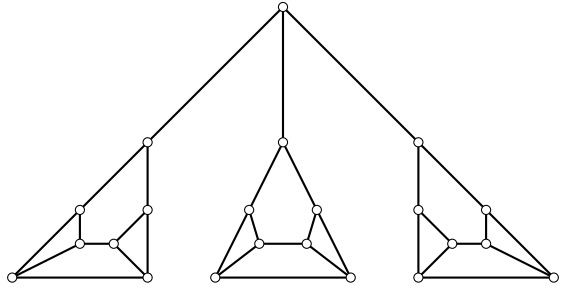

\begin{notes}
Large-girth families exist, but constructions are more involved. A beautiful group-theoretic construction was given by Margulis (\emph{Explicit constructions of graphs without short cycles and low density codes}, Combinatorica 1982).  
\end{notes}

\bigskip
\subsection*{Isoperimetric number}\label{sec: isop}\
\medskip

For a non-empty vertex subset $S$ in a graph, the \emph{boundary} $\bd S$ is the set of edges connecting $S$ to its complement $S^c$. Note that $S$ and $S^c$ have the same boundary. 

The \emph{isoperimetric constant}, denoted $\beta$, is given by
\begin{align*}
\beta=\min_{S}\frac{|\bd S|}{|S|}
\end{align*}
where the minimum is taken over all vertex subsets $S$ that contain no more than half the vertices. A set $S$ on which $\beta$ is attained is said to be an \emph{isoperimetric set}.

Finding $\beta$ can be quite difficult, even on small or familiar graphs. Unlike all the previous graph invariants, which are integral, the isoperimetric constant $\beta$ is rational. It is therefore harder to limit possibilities in a trial-and-error approach. In concrete examples, it can be quite easy to give an upper bound for $\beta$, by means of a well-chosen vertex subset. However, proving a matching lower bound for $\beta$ is usually much harder.

\begin{ex} Consider the cube graph $Q_n$. The binary strings starting in $0$ define a subset $S$ containing half the number of vertices of $Q_n$, and $|\bd S|/|S|=1$ since every vertex in $S$ has one external edge. Thus $\beta\leq 1$. We claim that, in fact, $\beta=1$. 

We think of $Q_n$ as the product $K_2\times\dots\times K_2$. We wish to argue that $\beta(X\times K_2)=1$ whenever $\beta(X)=1$. As $\beta(K_2)=1$, the bound $\beta(X\times K_2)\leq 1$ is an instance of the general principle that
\begin{align*}
\beta(X\times Y)\leq \min\{\beta (X),\beta(Y)\}.
\end{align*} 
Indeed, let $S$ be an isoperimetric subset in $X$. In $X\times Y$, the boundary of the subset $S\times Y$ is in bijection with $(\bd S) \times Y$. It follows that $\beta(X\times Y)\leq \beta(X)$, and similarly for $\beta(Y)$.

 The trickier bit is to show that $\beta(X\times K_2)\geq 1$. We visualize $X\times K_2$ as two `horizontal' copies of $X$, $X_\bl$ and $X_\wh$, with additional `vertical' edges joining $v_\bl$ to $v_\wh$ for each vertex $v$ in $X$. A vertex subset $S$ in $X\times K_2$ is of the form $A_\bl\cup B_\wh$, where $A$ and $B$ are vertex subsets in $X$. The boundary of $S$ in $X\times K_2$ has horizontal edges, corresponding to the boundary of $A$, respectively $B$, in $X$, and vertical edges, corresponding to the symmetric difference $A\triangle B$. Thus, $|\bd S|=|\bd A|+|\bd B|+|A\triangle B|$. We have to show that
\begin{align*}
|\bd A|+|\bd B|+|A\triangle B|\geq |A|+|B|
\end{align*}
whenever $|A|+|B|\leq |X|$. If $|A|,|B|\leq |X|/2$, then we already have $|\bd A|\geq |A|$ and $|\bd B|\geq |B|$. If $|A|\geq |X|/2\geq |B|$, then $|\bd B|\geq |B|$, $|\bd A|\geq |A^c|=|X|-|A|\geq |B|$, and $|A\triangle B|\geq |A|-|B|$.  Adding up these inequalities completes the argument. 
\end{ex}

\begin{exer}\label{exer: iso petersen} 
Show that $\beta(K_n)=\lceil n/2\rceil$; $\beta(C_n)=2/\lfloor n/2\rfloor$; $\beta(K_{n,n})=\lceil n^2/2\rceil/n$; $\beta=1$ for the Petersen graph.
\end{exer}

\begin{align*}
\begin{array}{rccccc}
 & \quad K_n \quad &  \quad C_n \quad  & \quad Q_n \quad & \quad K_{n,n} \quad & \textrm{Petersen}  \\[.4em]\cline{2-6}\\
\textrm{isoperimetric}\quad\beta \quad  & \sim n/2 & \sim 4/n & 1 & \sim n/2 & 1
\end{array}
\end{align*}
\smallskip

This table records our findings in asymptotic form: we write $a_n\sim b_n$ if $a_n/b_n\to 1$ as $n\to \infty$.

The isoperimetric constant clearly satisfies $0<\beta\leq d$. Taking a subset $S$ consisting of two adjacent vertices, the upper bound can be improved to $d-1$. But a more careful argument, explained below, yields an upper bound of roughly $d/2$. Note, however, that $\beta$ may slightly exceed $d/2$. This is the case for $K_n$ and $K_{n,n}$, for odd $n$.

\begin{thm}\label{better bound beta} We have:
\begin{align*}
\beta\leq \frac{d}{2}\cdot \frac{n+1}{n-1}
\end{align*}
\end{thm}

\begin{proof}
For $1\leq s\leq \lfloor n/2\rfloor$, we let $\beta_s$ denote the minimum of $|\bd S|/|S|$ as $S$ runs over vertex subsets of  size $s$. An upper bound for $\beta_s$ is the average 
\begin{align*}
\frac{1}{\# \{S: |S|=s\}}\sum _{|S|=s} \frac{|\bd S|}{|S|}
\end{align*}
which we can actually compute. Any given edge is in the boundary of $2\big(\begin{smallmatrix} n-2 \\ s-1 \end{smallmatrix}\big)$ subsets of size $s$, so the numerator is $\sum _{|S|=s} |\bd S|=2|E|\big(\begin{smallmatrix} n-2 \\ s-1 \end{smallmatrix}\big)$. The denominator is $s\big(\begin{smallmatrix} n \\ s \end{smallmatrix}\big)$. Simplifying, we deduce that
\begin{align*}
\beta_s\leq \frac{2|E|(n-s)}{n(n-1)}.
\end{align*}
As $2|E|\leq dn$, we get
\begin{align*}
\beta=\min_s \beta_s\leq \min_s \frac{d(n-s)}{n-1}= \frac{d(n-\lfloor n/2\rfloor)}{n-1}
\end{align*}
which is actually slightly better than the claim.
\end{proof}

\begin{exer}\label{exer: upper logarithmic bound} Let $B_r(u)$ denote the ball of radius $r$ around a vertex $u$. Prove the following short-range exponential growth behaviour: $|B_{r}(u)|\geq (1+\beta/d)^r$ as long as $|B_{r-1}(u)|\leq n/2$. Conclude that
\begin{align*}
\delta \leq \frac{2\log (n/2)}{\log(1+\beta/d)}+2.
\end{align*}
\end{exer}

Let us explain an important qualitative consequence of this exercise. A $d$-regular family $\{X_k\}$ is said to be an \emph{expander} if the sequence of isoperimetric constants $\{\beta_k\}$ is bounded away from $0$. This is a very interesting condition: roughly speaking, an expander family is an infinite sequence of sparse graphs which nevertheless maintain a sizable connectivity. Exercise~\ref{exer: upper logarithmic bound} implies the following fact: an expander family has small diameter. Note that the converse does not hold. The small-diameter family $\{X_k\}$ of Example~\ref{ex: sd} is not an expander. Indeed, if $S$ is the set of vertices lying on one of the three main branches stemming from the root, then $|\bd S|=1$ and $|S|=2^k-1$. Thus $\beta_k\to 0$ exponentially fast.

\begin{notes}
Among the five invariants under consideration, the isoperimetric constant is the youngest. To the best of our knowledge, it was introduced by Buser (\emph{Cubic graphs and the first eigenvalue of a Riemann surface}, Math. Z. 1978). This partly explains our choice of notation, $\beta$. Elsewhere, the isoperimetric constant is denoted by $h$ or $i$, and it is also called the Cheeger constant because it is the graph-theoretic analogue of an isoperimetric constant for manifolds, denoted $h$, due to Cheeger (\emph{A lower bound for the smallest eigenvalue of the Laplacian}, in `Problems in analysis', Princeton University Press 1970). Other sources call it conductance, and denote it by $\phi$ or $\Phi$.

Theorem~\ref{better bound beta} is due to Mohar (\emph{Isoperimetric numbers of graphs}, J. Combin. Theory Ser. B 1989).
\end{notes}

\bigskip
\section{Regular graphs I}
In this section, we focus on constructions of regular graphs from groups. The most important notion here is that of a Cayley graph. We also discuss a bipartite variation, that of a bi-Cayley graph.

\bigskip
\subsection*{Cayley graphs}\label{sec: sameness}\
\medskip

Let $G$ be a finite group. Let $S\subseteq G$ be a subset that does not contain the identity of $G$, and which is \emph{symmetric} (i.e., closed under taking inverses) and \emph{generating} (i.e., every element of $G$ is a product of elements of $S$). The \emph{Cayley graph} of $G$ with respect to $S$ has the elements of $G$ as vertices, and an edge between every two vertices $g,h\in G$ satisfying $g^{-1}h\in S$. In other words, the neighbours of a vertex $g$ are the vertices of the form $gs$, where $s\in S$. 

The Cayley graph of $G$ with respect to $S$ an $|S|$-regular graph of size $|G|$. The assumptions on $G$ and $S$ reflect our standing convention that graphs should be finite, connected, and simple.

\begin{ex} The Cayley graph of a group $G$ with respect to the set of non-identity elements of $G$ is the complete graph on $n=|G|$ vertices. 
\end{ex}

\begin{ex} The Cayley graph of $\Z_n$ with respect to $\{\pm 1\}$ is the cycle graph $C_n$. 
\end{ex}

\begin{ex} The Cayley graph of $(\Z_2)^n$ with respect to $\{e_i=(0,\dots, 0,1,0,\dots, 0): i=1,\dots,n\}$ is the cube graph $Q_n$.
\end{ex}

\begin{ex} The \emph{halved cube graph} $\tfrac{1}{2}Q_n$ is defined as follows: its vertices are those binary strings of length $n$ that have an even weight, and edges connect two such strings when they differ in precisely two slots. This is the Cayley graph of the subgroup $\big\{v\in (\Z_2)^n: \sum v_i=0\big\}$ with respect to $\{e_i+e_j: 1\leq i< j\leq n\}$.
\end{ex}

\begin{ex} (The twins) Here we consider two Cayley graphs coming from the same group, $\Z_4\times \Z_4$. The two symmetric generating sets are as follows: 
\begin{align*}
S_1&=\{\pm (1,0), \pm (0,1),  (2,0), (0,2)\}\\
S_2&=\{\pm (1,0), \pm (0,1), \pm (1,1)\}
\end{align*}
The corresponding Cayley graphs are $6$-regular graphs on $16$ vertices. It is not hard to recognize the Cayley graph with respect to $S_1$ as the product $K_4\times K_4$. The second Cayley graph, with respect to $S_2$, is known as the \emph{Shrikhande graph}. Both are depicted in Figure~\ref{fig: twins}. An amusing puzzle, left to the reader, is to check the pictures by finding an appropriate $\Z_4\times \Z_4$ labeling. The more interesting one is, of course, the Shrikhande graph.

\begin{figure}[!ht]
\medskip
\GraphInit[vstyle=Simple]
\tikzset{VertexStyle/.style = {
shape = circle,
inner sep = 0pt,
minimum size = .8ex,
draw}}
\begin{minipage}[b]{0.48\linewidth}
\centering
\begin{tikzpicture}[scale=.6]
  \grEmptyCycle[prefix=a, RA=4pt, rotation=11.25]{16}
  \EdgeDoubleMod{a}{16}{0}{1}{a}{16}{4}{1}{16}
  \EdgeDoubleMod{a}{16}{0}{1}{a}{16}{8}{1}{8}
  \EdgeDoubleMod{a}{16}{0}{4}{a}{16}{1}{4}{4}
  \EdgeDoubleMod{a}{16}{0}{4}{a}{16}{2}{4}{4}
  \EdgeDoubleMod{a}{16}{0}{4}{a}{16}{3}{4}{4}
  \EdgeDoubleMod{a}{16}{1}{4}{a}{16}{2}{4}{4}
  \EdgeDoubleMod{a}{16}{1}{4}{a}{16}{3}{4}{4}
  \EdgeDoubleMod{a}{16}{2}{4}{a}{16}{3}{4}{4}
\end{tikzpicture}
\end{minipage}
\begin{minipage}[b]{0.48\linewidth}
\centering
\begin{tikzpicture}[scale=.6]
  \grEmptyCycle[prefix=a, RA=4pt, rotation=22.5]{8}
  \grEmptyCycle[prefix=b, RA=2pt, rotation=22.5]{8}
  \EdgeDoubleMod{a}{8}{0}{1}{a}{8}{1}{1}{8}
  \EdgeDoubleMod{a}{8}{0}{1}{a}{8}{2}{1}{8}
  \EdgeDoubleMod{a}{8}{0}{1}{b}{8}{1}{1}{8}
  \EdgeDoubleMod{a}{8}{1}{1}{b}{8}{0}{1}{8}
  \EdgeDoubleMod{b}{8}{0}{1}{b}{8}{2}{1}{8}
  \EdgeDoubleMod{b}{8}{0}{1}{b}{8}{3}{1}{8}
\end{tikzpicture}
\end{minipage}
\caption{$K_4\times K_4$ and the Shrikhande graph.}\label{fig: twins}
\end{figure}
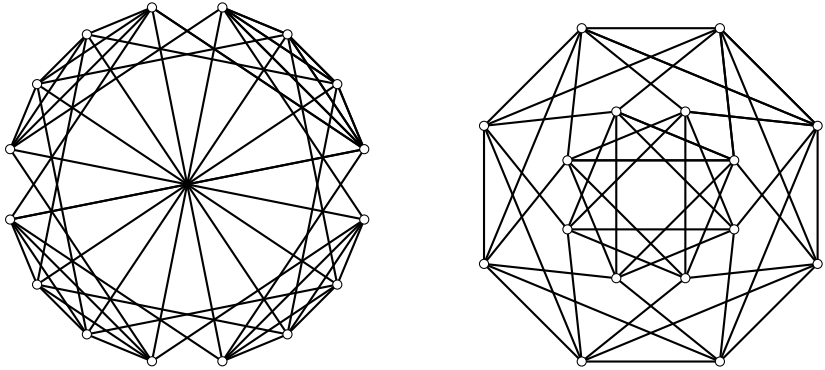
It seems quite appropriate to think of these two graphs as being twins. Not only are they Cayley graphs of the same group, but they also turn out to be indistinguishable by any of the five invariants under consideration! Both graphs have diameter $2$ and girth $3$. They also have the same chromatic number, $\chi=4$. Indeed, we get a nice $4$-colouring for both graphs by noting that the addition homomorphism $\Z_4\times \Z_4\to \Z_4$, given by $(a,b)\mapsto a+b$, is non-trivial on both sets of generators. On the other hand, three colours are not sufficient for $K_4\times K_4$, as it contains $K_4$ as a subgraph; in the Shrikhande graph, it is the outer octagon that cannot be $3$-coloured. Next, both graphs have independence number $\iota=4$. Firstly, we know that $\iota\geq 4$ by using the inequality $\chi\times \iota\geq 16$. Secondly, we easily find independent sets of size $4$ in each one of the two graphs. Finally, both graphs turn out to have isoperimetric constant $2$. A conceptual explanation for this fact will be given later on, using spectral ideas. 
\end{ex}

\begin{ex} Let $n\geq 3$. The Cayley graph of $\Z_{3n-1}$ with respect to the symmetric generating set $S=\{1,4,\dots,3n-2\}$ is known as the \emph{Andr\'asfai graph} $A_n$.
\end{ex}

\begin{exer}\label{exer: andrasfai} 
Show that the Andr\'asfai graph $A_n$ has diameter $2$, girth $4$, chromatic number $3$, and independence number $n$.
\end{exer}

\begin{ex} Let us view the five Platonic solids as graphs. The tetrahedron and the cube yield two familiar graphs, $K_4$ and $Q_3$. The next three are the octahedral graph (Figure~\ref{fig: octa}), the icosahedral graph (Figure~\ref{fig: icos}), and the dodecahedral graph (Figure~\ref{fig: dode}).

\begin{figure}[!ht]\bigskip
\centering
\GraphInit[vstyle=Simple]
\tikzset{VertexStyle/.style = {
shape = circle,
inner sep = 0pt,
minimum size = .8ex,
draw}}
\begin{minipage}[b]{0.48\linewidth}
\centering
\begin{tikzpicture}[scale=.35]
\grSQCycle[RA=5]{6}
\end{tikzpicture}
\end{minipage}
\begin{minipage}[b]{0.48\linewidth}
\centering
\begin{tikzpicture}[scale=.35, rotate=-30]
\grOctahedral[RA=5.5,RB=1]
\end{tikzpicture}
\end{minipage}
\caption{Two views of the octahedral graph.}\label{fig: octa}
\bigskip
\bigskip
\begin{minipage}[b]{0.48\linewidth}
\centering
\begin{tikzpicture}[scale=.35]
\grIcosahedral[RA=5,RB=3]
\end{tikzpicture}
\end{minipage}
\begin{minipage}[b]{0.48\linewidth}
\centering
\begin{tikzpicture}[scale=.35, rotate=-30]
\grIcosahedral[form=2,RA=5.5,RB=1.8,RC=.8]
\end{tikzpicture}
\end{minipage}
\caption{Two views of the icosahedral graph.}\label{fig: icos}
\bigskip
\bigskip
\begin{minipage}[b]{0.48\linewidth}
\centering
\begin{tikzpicture}[scale=0.35]
\grDodecahedral[RA=5,RB=3,RC=2,RD=1]
\end{tikzpicture}
\end{minipage}
\begin{minipage}[b]{0.48\linewidth}
\centering
\begin{tikzpicture}[scale=.35]
  \grCycle[prefix=a, RA=1pt, rotation=18+36]{5}
  \grCycle[prefix=b, RA=3pt, rotation=18]{10}
  \grCycle[prefix=c, RA=5.4pt, rotation=18]{5}
  \EdgeDoubleMod{a}{5}{0}{1}{b}{10}{1}{2}{5}
  \EdgeDoubleMod{b}{10}{0}{2}{c}{5}{0}{1}{5}
\end{tikzpicture}
\end{minipage}
\caption{Two views of the dodecahedral graph.}\label{fig: dode}
\end{figure}

The octahedral graph can be realized as the Cayley graph of $\Z_6$ with respect to $\{\pm 1,\pm 2\}$. Another realization is on the dihedral group with $6$ elements. We think of this group as the semidirect product $\Z_3\rtimes \{\pm 1\}$: the underlying set is $\Z_3\times \{\pm 1\}$, with operation $(a,\e)(b,\tau)=(a+\e b, \e\tau)$. The elements in $\Z_3\times \{-1\}$ are involutions, and the remaining two non-trivial elements, $(\pm 1,1)$, are inverse to each other. A symmetric subset $S$ of size $4$ must contain these two elements, and two more involutions picked from $\Z_3\times \{-1\}$. The Cayley graph with respect to any one of the three possible choices for $S$ is the octahedral graph. There is a good reason why this is the case: there is just one $4$-regular graph on $6$ vertices.

The icosahedral graph is a Cayley graph, as well. Namely, it is the Cayley graph of the alternating group $\mathrm{Alt}(4)$ with respect to the set $\{(123), (132),(234), (243),(12)(34)\}$.
\end{ex}

\begin{ex} As we know, a cube graph is bipartite. Is a cube graph a bipartite double? 
  
We start from $Q_{n+1}$, and we look for an $(n+1)$-regular graph on $2^n$ vertices. Adding one edge per vertex to $Q_n$ would be such a graph. So let us consider the Cayley graph of $(\Z_2)^n$ with respect to $\{e_1,\dots,e_n,a\}$, where the new generator $a\in (\Z_2)^n$ has weight greater than $1$. This Cayley graph is referred to as the \emph{decked cube graph} $DQ_n(a)$ in what follows.

A decked cube graph only depends on the weight of the new generator. Indeed, let $a$ and $a'$ have equal weight. Then we can actually set up a \emph{Cayley graph isomorphism} between $DQ_n(a)$ and $DQ_n(a')$, that is, a group isomorphism between the underlying groups which maps one generating set onto the other. For there is a permutation of the basis $\{e_1,\dots,e_n\}$ which, once extended to a group automorphism of $(\Z_2)^n$, maps $a$ to $a'$.

A decked cube graph is bipartite if and only if the new generator has odd weight. To see this, consider the bipartition of $Q_n$ according to weight parity. Then the new edges join vertices of different colour precisely when the weight of the new generator is odd.

So we focus on a decked cube graph $DQ_n(a)$ where $a\in (\Z_2)^n$ has even weight, and we claim that its bipartite double is $Q_{n+1}$. Note the following general fact: the bipartite double of the Cayley graph of a group $G$ with respect to $S$ is the Cayley graph of $G\times \Z_2$ with respect to $S\times \{1\}$. So the bipartite double of $DQ_n(a)$ is the Cayley graph of $(\Z_2)^{n+1}$ with respect to $\{(e_1,1),\dots,(e_n,1),(a,1)\}$. To show that the latter graph is $Q_{n+1}$, we use again a Cayley graph isomorphism. We adopt a linear algebra perspective, viewing $\Z_2$ as field. Let $A$ be the $(n+1)$-by-$(n+1)$ matrix over $\Z_2$ with rows $(e_1,1),\dots,(e_n,1),(a,1)$. Then $v\mapsto Av$ defines a group homomorphism from $(\Z_2)^{n+1}$ to itself, mapping the standard basis $e'_1,\dots,e'_{n+1}$ to $(e_1,1),\dots,(e_n,1),(a,1)$. We are left with checking that the matrix $A$ is invertible, equivalently, that the binary strings $(e_1,1),\dots,(e_n,1),(a,1)$ are linearly independent in $(\Z_2)^{n+1}$. Let
\begin{align*}
c_0(a,1)+\sum_{i=1}^n c_i(e_i,1)=0
\end{align*}
for some coefficients $c_i\in \Z_2$. Then $c_0=c_1+\ldots+c_n$ and $c_0a=(c_1,\dots,c_n)$. If $c_0=1$ then $a$ has odd weight, contradicting our assumption on $a$. Hence $c_0=0$, and $c_1=\ldots=c_n=0$ follows as well.

Returning to our starting question, we infer that $Q_{n+1}$ is a bipartite double as soon as there is an element of $(\Z_2)^n$ having even, but non-zero weight. This holds for $n\geq 2$, so all cube graphs starting from $Q_3$ are bipartite doubles. As for $Q_2$, it is not a bipartite double.
\end{ex}

\begin{notes}
Cayley graphs were introduced by Cayley in 1878.
\end{notes}


\bigskip
\subsection*{Vertex-transitive graphs}\label{sec: vertex-tran}\
\medskip

A graph is said to be \emph{vertex-transitive} if any vertex can be taken to any other vertex by a graph automorphism. 

\begin{prop}
Cayley graphs are vertex-transitive.
\end{prop}

\begin{proof} Consider a Cayley graph of a group $G$. We look for a graph automorphism taking the identity vertex to an arbitrary vertex $g\in G$. Left multiplication by $g$, namely the mapping $h\mapsto gh$, does the job.
\end{proof}

\begin{exer}\label{exer: only left} Consider the Cayley graph of a group $G$ with respect to $S$. Show that the following are equivalent: (i) for each $g\in G$, right multiplication by $g$ is a graph automorphism; (ii) inversion on $G$, $g\mapsto g^{-1}$, is a graph automorphism; (iii) $S$ is closed under conjugation.
\end{exer}

\begin{ex} The Petersen graph is vertex-transitive. Let us give two arguments, according to how we choose to think of the Petersen graph.

Symbolically, the Petersen graph records the relation of disjointness between $2$-element subsets of a $5$-element set $S$. A permutation of $S$ maps $2$-element subsets to $2$-element subsets, and it does so in a disjointness-preserving fashion. So any permutation of $S$ defines an automorphism of the Petersen graph. To conclude, observe that any $2$-element subset of $S$ can be taken to any other one by a permutation.

Graphically, we think of the Petersen graph in its pentagram-within-pentagon drawing. Consider a labeling of the vertices as in the right-most panel of the triptych in Figure~\ref{fig: knotpet}. Note that there is an obvious $5$-fold rotational symmetry. So the vertices of the outer pentagon, $1$ through $5$, are in the same orbit of the automorphism group, and the same is true for the vertices of the outer pentagon, $a$ through $e$. We need an automorphism that mixes up the two orbits. This is explained in the next two panels of Figure~\ref{fig: knotpet}: switch $b$ and $3$, respectively $e$ and $4$, and then switch $c$ and $d$. A mixing automorphism can be read off by comparing the first and the last panels.

 \begin{figure}[!ht]\bigskip
\centering
\GraphInit[vstyle=Classic]
\tikzset{VertexStyle/.style = {
shape = circle,
inner sep = 0pt,
minimum size = .8ex,
draw}}
\begin{minipage}[b]{0.32\linewidth}
\centering
\begin{tikzpicture}[scale=.6]
\Vertex[x=5,y=4.5]{1}
\Vertex[x=2.5,y=2.5]{2}
\Vertex[x=7.5,y=2.5]{5}
\Vertex[x=3.5,y=0]{3}
\Vertex[x=6.5,y=0]{4}
\Vertex[x=5,y=3.2]{a}
\Vertex[x=3.6,y=2.3]{b}
\Vertex[x=6.4,y=2.3]{e}
\Vertex[x=4.2,y=0.9]{c}
\Vertex[x=5.8,y=0.9]{d}
\Edges(1, 2, 3, 4, 5, 1)
\Edges(a,d,b,e,c,a)
\Edges(1,a)
\Edges(2,b)
\Edges(3,c)
\Edges(4,d)
\Edges(5,e)
\end{tikzpicture}
\end{minipage}
\begin{minipage}[b]{0.32\linewidth}
\centering
\begin{tikzpicture}[scale=.6]
\Vertex[x=5,y=4.5]{1}
\Vertex[x=2.5,y=2.5]{2}
\Vertex[x=7.5,y=2.5]{5}
\Vertex[x=3.5,y=0]{b}
\Vertex[x=6.5,y=0]{e}
\Vertex[x=5,y=3.2]{a}
\Vertex[x=3.6,y=2.3]{3}
\Vertex[x=6.4,y=2.3]{4}
\Vertex[x=4.2,y=0.9]{c}
\Vertex[x=5.8,y=0.9]{d}
\Edges(1, 2, 3, 4, 5, 1)
\Edges(a,d,b,e,c,a)
\Edges(1,a)
\Edges(2,b)
\Edges(3,c)
\Edges(4,d)
\Edges(5,e)
\end{tikzpicture}
\end{minipage}
\begin{minipage}[b]{0.32\linewidth}
\centering
\begin{tikzpicture}[scale=.6]
\Vertex[x=5,y=4.5]{1}
\Vertex[x=2.5,y=2.5]{2}
\Vertex[x=7.5,y=2.5]{5}
\Vertex[x=3.5,y=0]{b}
\Vertex[x=6.5,y=0]{e}
\Vertex[x=5,y=3.2]{a}
\Vertex[x=3.6,y=2.3]{3}
\Vertex[x=6.4,y=2.3]{4}
\Vertex[x=4.2,y=0.9]{d}
\Vertex[x=5.8,y=0.9]{c}
\Edges(1, 2, 3, 4, 5, 1)
\Edges(a,d,b,e,c,a)
\Edges(1,a)
\Edges(2,b)
\Edges(3,c)
\Edges(4,d)
\Edges(5,e)
\end{tikzpicture}
\end{minipage}
\caption{}\label{fig: knotpet}
\end{figure}
\end{ex}

\begin{ex} The five Platonic graphs are vertex-transitive. Reflections take any face to any other face. Within each face, we can get from any vertex to any other vertex by using rotations. Recall that all but one of them, the dodecahedral graph, are in fact Cayley graphs.
\end{ex}

\begin{ex} Are there vertex-transitive graphs which are not Cayley graphs? The answer is positive, as one might guess, but providing examples is not such an easy matter. Let us show that the Petersen graph and the dodecahedral graph, both of which we already know to be vertex-transitive, are not Cayley graphs.

The Petersen graph and the dodecahedral graph have two noticeable features in common: both are $3$-regular, and have girth $5$. Consider a Cayley graph having these two properties. Then there are two cases, according to whether the symmetric generating set $S$ contains one or three involutions.  

Let $S=\{r,s,t\}$, where $r^2=s^2=t^2=1$. The edges of a $5$-cycle are labeled by $r$, $s$, or $t$. Two consecutive edges have different labels, so one of labels, say $t$, appears only once. Up to swapping $r$ and $s$, we may assume that we have the labeling of Figure~\ref{fig: 3 involutions}. Then $t=(rs)^2=(sr)^2$. It follows that $r$ commutes with $t$: $rt=r(sr)^2=(rs)^2r=tr$. This means that there is a $4$-cycle in the graph, a contradiction.

\begin{figure}[ht!]\medskip
\begin{minipage}[b]{0.32\linewidth}
\centering
\GraphInit[vstyle=Simple]
\tikzset{VertexStyle/.style = {
shape = circle,
inner sep = 0pt,
minimum size = .2ex,
draw}}
\begin{tikzpicture}[rotate=90]
\SetGraphUnit{1.5}
\SetVertexNoLabel
\Vertices{circle}{A,B,C,D,E}
\Edge[label=$s$](B)(A)
\Edge[label=$r$](A)(E)
\Edge[label=$s$](E)(D)
\Edge[label=$r$](C)(B)
\Edge[label=$t$](D)(C)
\end{tikzpicture}
\caption{}
\label{fig: 3 involutions}
\end{minipage}
\begin{minipage}[b]{0.32\linewidth}
\centering
\GraphInit[vstyle=Simple]
\tikzset{VertexStyle/.style = {
shape = circle,
inner sep = 0pt,
minimum size = .2ex,
draw}}
\begin{tikzpicture}[rotate=90]
\SetGraphUnit{1.5}
\SetVertexNoLabel
\Vertices{circle}{A,B,C,D,E}
\Edge[style={->},label=$a$](B)(A)
\Edge[style={->}, label=$a$](A)(E)
\Edge[label=$r$](E)(D)
\Edge[label=$r$](C)(B)
\Edge[label=$a^{\pm}$](D)(C)
\end{tikzpicture}
\caption{}
\label{fig: two r}
\end{minipage}
\begin{minipage}[b]{0.32\linewidth}
\centering
\GraphInit[vstyle=Simple]
\tikzset{VertexStyle/.style = {
shape = circle,
inner sep = 0pt,
minimum size = .2ex,
draw}}
\begin{tikzpicture}[rotate=90]
\SetGraphUnit{1.5}
\SetVertexNoLabel
\Vertices{circle}{A,B,C,D,E}
\Edge[style={->},label=$a$](B)(A)
\Edge[style={->}, label=$a$](A)(E)
\Edge[style={->}, label=$a$](E)(D)
\Edge[style={->}, label=$a$](C)(B)
\Edge[label=$r$](D)(C)
\end{tikzpicture}
\caption{}
\label{fig: one r}
\end{minipage}
\end{figure}

Let $S=\{r,a,a^{-1}\}$, where $r^2=1$. Label again the edges of a $5$-cycle by the generators. There can be at most two occurrences of $r$. If there are two occurrences, then we get the relation $rar=a^{\pm 2}$ (Figure~\ref{fig: two r}). But then $a=ra^{\pm 2}r=(rar)^{\pm 2}=a^4$, meaning that $a^3=1$. This yields a $3$-cycle in the graph, hence a contradiction. If there is only one occurrence of $r$, then we get $r=a^4$ (Figure~\ref{fig: one r}). The commutation relation $ar=ra$ yields a $4$-cycle in the graph, a contradiction. The remaining case is no occurrence of $r$ in the chosen $5$-cycle. However, both the Petersen graph and the dodecahedral graph contain two $5$-cycles sharing an edge, and it cannot be that both $5$-cycles have no occurrences of $r$. 
\end{ex}

A vertex-transitive graph is, in particular, regular. In fact, a vertex-transitive graph has the same link graph at each vertex. This idea serves as an obstruction to vertex-transitivity. An illustration is provided by the regular graphs in Figure~\ref{fig: linking}, whose link patterns reveal that they are not vertex-transitive. Also, we may speak of the link graph of a vertex-transitive graph, and we may read it off at any vertex. For instance, the link graph of $K_4\times K_4$ is a disjoint union of two $3$-cycles, and the link graph of the Shrikhande graph is a $6$-cycle. So the twin graphs are, just as we suspected, not isomorphic.

\bigskip
\subsection*{Bi-Cayley graphs}\label{sec: bi-cayley}\
\medskip

Let $G$ be a finite group, and let $S\subseteq G$ be a subset which need not be symmetric, or which might contain the identity. A Cayley graph of $G$ with respect to $S$ can still be defined, but it will have directed edges or loops. Subsequently passing to an simple undirected graph, by forgetting directions and erasing loops, does not reflect the choice of $S$. An appropriate undirected substitute for the Cayley graph of $G$ with respect to $S$ can be constructed as follows. This illustrates a general procedure of turning a graph with directed edges and loops into a simple graph, by a bipartite construction. 

The \emph{bi-Cayley graph} of $G$ with respect to $S$ is the bipartite graph on two copies of $G$, $G_\bl$ and $G_\wh$, in which $g_\bl$ is adjacent to $h_\wh$ whenever $g^{-1}h\in S$. This is a regular graph, of degree $|S|$. It is connected if and only if the symmetric subset $S\cdot S^{-1}=\{st^{-1}:s,t\in S\}$ generates $G$. This is stronger than requiring $S$ to generate $G$, though equivalent if $S$ happens to contain the identity. Connectivity of a bi-Cayley graph may be cumbersome as a direct algebraic verification but in many cases this task is facilitated by alternate descriptions of the graph.

\begin{ex} The bi-Cayley graph of $\Z_7$ with respect to $S=\{1,2,4\}$ is the \emph{Heawood graph}, Figure~\ref{fig: heawood} below. An interesting observation is that, when computing the difference set $S-S$, every non-zero element of $\Z_7$ appears exactly once. 

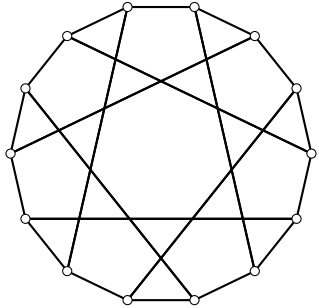
\begin{figure}[!ht]\medskip
\centering
\GraphInit[vstyle=Simple]
\tikzset{VertexStyle/.style = {
shape = circle,
inner sep = 0pt,
minimum size = .8ex,
draw}}
\begin{tikzpicture}[scale=.4]
\grLCF[RA=5]{5,9}{7}
\end{tikzpicture}
\caption{The Heawood graph.}\label{fig: heawood}
\end{figure}
\end{ex}

There are several observations to be made. The bi-Cayley graph of $G$ with respect to $S$ can also be defined as the bipartite graph on two copies of $G$, in which $g_\bl$ is adjacent to $h_\wh$ whenever $gh\in S$. This adjacency law is more symmetric, and tends to appear more often in practice. The correspondence $g_\bl\leftrightarrow (g^{-1})_\bl$, $h_\wh\leftrightarrow h_\wh$ shows the equivalence of the two descriptions.

If $S$ is such that the Cayley graph of $G$ with respect to $S$ already makes sense, then its bipartite double is the bi-Cayley graph of $G$ with respect to $S$. We interpret this as saying that the bi-Cayley graph construction is essentially a generalization of the Cayley graph construction.

The bi-Cayley graph of $G$ with respect to $S$ is a Cayley graph in many cases of interest. Namely, assume $G$ is an abelian group, and consider the semidirect product $G\rtimes \{\pm 1\}$ given by inversion. This means that we endow the set $G\times \{\pm 1\}$ with the non-abelian operation $(g,\e)(h,\tau)=(gh^{\e}, \e\tau)$ for $g,h\in G$ and $\e,\tau\in \{\pm 1\}$. If $G$ happens to be a cyclic group, then $G\rtimes \{\pm 1\}$ is a dihedral group. This familiar instance prompts us to deem $G\rtimes \{\pm 1\}$ a \emph{generalized dihedral group}. The advantage of the semidirect product over the direct product $G\times\{\pm 1\}$ is that now $S\times \{-1\}$ is symmetric in $G\rtimes \{\pm 1\}$, for it consists of involutions. The Cayley graph of $G\rtimes \{\pm 1\}$ with respect to $S\times \{-1\}$ is a bipartite graph on $G\times \{+1\}$ and $G\times \{-1\}$, with $(g,+1)$ connected to $(h,-1)$ whenever $g^{-1}h\in S$. This is precisely the bi-Cayley graph of $G$ with respect to $S$.

\bigskip
\section{Regular graphs II}

This section introduces some types of regular graphs. Their common feature is that the number of neighbours shared by pairs of vertices is severely restricted.

\bigskip
\subsection*{Strongly regular graphs}\label{sec: strongly regular}\
\medskip

A regular graph is said to be \emph{strongly regular} if there are two non-negative integers $a$ and $c$ such that any two adjacent vertices have $a$ common neighbours, and any two distinct non-adjacent vertices have $c$ common neighbours. 

\begin{ex}
The Petersen graph is strongly regular, with $a=0$ and $c=1$. 
\end{ex}

\begin{ex} 
The cycle $C_5$ is strongly regular, with $a=0$ and $c=1$. The only other cycle which is strongly regular is $C_4$, which has $a=0$ and $c=2$.
\end{ex}

\begin{ex}  The complete bipartite graph $K_{n,n}$ is strongly regular, with $a=0$ and $c=n$. No other bipartite graph is strongly regular. 
\end{ex}

\begin{ex}  The rook's graph $K_n\times K_n$ is strongly regular, with $a=n-2$ and $c=2$. No other non-trivial product is strongly regular.
\end{ex}

\begin{ex}  The twin graphs, $K_4\times K_4$ and the Shrikhande graph, are strongly regular with $a=c=2$.
\end{ex}

The complete graph $K_n$ has $a=n-2$ but undefined $c$. Treating it as a strongly regular graph is a matter of convention, and ours is to exclude it. 

The parameters $(n, d,a,c)$ of a strongly regular graph are constrained by a number of relations. The simplest one is the following:

\begin{prop}\label{prop: double counting}
The parameters $(n, d,a,c)$ of a strongly regular graph satisfy
\begin{align*}
d(d-a-1)=(n-d-1)c.
\end{align*} 
\end{prop}

\begin{proof}
Fix a base vertex; its link graph has $d$ vertices, and there are $n-d-1$ farther vertices. The desired relation comes from counting in two ways the edges between link vertices and farther vertices.

Each link vertex has $a$ common neighbours with the base vertex. In other words, the link graph is $a$-regular. So each link vertex is adjacent to $d-a-1$ farther vertices.

Each farther vertex has $c$ common neighbours with the base vertex. That is, each farther vertex is adjacent to $c$ link vertices. 
\end{proof} 

As we have seen, the link graph of any vertex is an $a$-regular graph on $d$ vertices.  It cannot be complete on $d$ vertices - if it were, the whole graph would be complete on $d+1$ vertices. So $a<d-1$, and $c>0$. The latter fact means the following.  

\begin{prop} A strongly regular graph has diameter $2$.
\end{prop}

There are non-isomorphic strongly regular graphs with the same parameters $(n, d,a,c)$. For example, the twin graphs have parameters $(16,6,2,2)$. The following construction provides many more examples of this kind.

\begin{thm}\label{thm: machine} Fix a positive integer $n$. Let $G$ be a finite group with $n$ elements, and let $X_G$ be the Cayley graph of $G\times G$ with respect to the symmetric generating subset $\{(s,1), (1,s), (s,s): s\in G, s\neq 1\}$. Then:
\begin{itemize}
\item[(i)] $X_G$ is strongly regular with parameters $(n^2,3n-3,n,6)$, 
\item[(ii)] the number of elements of order $2$ in $G$ is a graph invariant of $X_G$. 
\end{itemize}
\end{thm}

\begin{proof}
Part (i) is straightforward counting. Pick two vertices in $X_G$, one of which may be assumed to be the identity $(1,1)$. Its neighbors are of the form $(s,s)$, or $(s,1)$, or $(1,s)$ with $s\neq 1$. There are $n$ common neighbours in each case: $(s,1)$, $(1,s)$, and $(t,t)$ with $t\neq 1,s$; $(s,s)$, $(1,s)$, and $(t,1)$ with $t\neq 1,s$; $(s,s)$, $(s,1)$, and $(1,t)$ with $t\neq 1,s$. Now take a vertex $(x,y)$ not adjacent to $(1,1)$, meaning that $x\neq 1$, $y\neq 1$, and $x\neq y$. Then $(x,y)$ has $6$ common neighbours with $(1,1)$, namely $(x,x)$ and $(y,y)$, $(x,1)$ and $(xy^{-1},1)$, $(1,y)$ and $(1,yx^{-1})$.

(ii) Consider the link graph of $X_G$. It has $3n-3$ vertices, and it is $n$-regular. The vertices of the link at $(1,1)$ are naturally partitioned into three `islands': $S_1=\{(s,1): s\neq 1\}$, $S_2=\{(1,s): s\neq 1\}$, and $S_3=\{(s,s): s\neq 1\}$. Each $S_i$ is a complete subgraph on $n-1$ vertices. A vertex has $n-2$ incident edges within its own island, and two more `bridge' edges connecting it to the remaining two islands. It follows that $3$-cycles in the link graph are either contained in an island, or they consist of three vertices from each of the islands, connected by bridge edges. The number of $3$-cycles contained in an island is a function of $n$. The number of $3$-cycles of bridge type equals the number of elements of order $2$ in $G$. Indeed, the paths of length $3$ starting and ending in $S_2$, say, along bridge edges have the form $(s,1)\edge (s,s)\edge (1,s) \edge (s^{-1},1)$. This is a $3$-cycle precisely when $s$ has order $2$. In conclusion, the number of elements of order $2$ in $G$ is a graph invariant of the link graph of $X_G$, hence of $X_G$ itself, when the order of $G$ is fixed.
\end{proof}

\begin{cor}\label{cor: machine} For every positive integer $N$, there exist $N$ strongly regular graphs that are pairwise non-isomorphic, but have the same parameters.
\end{cor}

\begin{proof}
For $k=1,\dots, N$, consider the abelian group $G(k)=(\Z_2)^{k-1}\times \Z_{2^{N-k}}$. Then $G(k)$ has size $2^{N-1}$, and $2^{k}-1$ elements of order $2$. The previous theorem implies that the graphs $X_{G(k)}$, $k=1,\dots, N$, are strongly regular with the same parameters, but they  are mutually non-isomorphic.
\end{proof}

\begin{notes}
Strongly regular graphs were introduced by Bose (\emph{Strongly regular graphs, partial geometries and partially balanced designs}, Pacific J. Math. 1963). The Shrikhande graph was originally defined by Shrikhande (\emph{The Uniqueness of the $L_2$ Association Scheme}, Ann. Math. Stat. 1959) in a combinatorial way that made strong regularity immediate. 
\end{notes}

\bigskip
\subsection*{Design graphs}\label{sec: design}\
\medskip

Next, we consider bipartite analogues of strongly regular graphs.

A \emph{design graph} is a regular bipartite graph with the property that any two distinct vertices of the same colour have the same number of common neighbours. The complete bipartite graph $K_{n,n}$ fits the definition, but we exclude it by convention.

For a design graph, we let $m$ denote the half-size, $d$ the degree, and $c$ the number of neighbours shared by any monochromatic pair of vertices. Note that the parameter $c$ is, in fact, determined by the following relation: 
\begin{align*}
c(m-1)=d(d-1)
\end{align*} 
This is obtained by counting in two ways the paths of length $2$ joining a fixed vertex with the remaining $m-1$ vertices of the same colour.

As an immediate consequence of the definition, we have following fact.

\begin{prop} 
A design graph has diameter $3$, and girth $4$ or $6$ according to whether $c>1$ or $c=1$. Conversely, a regular bipartite graph of diameter $3$ and girth $6$ is a design graph with parameter $c=1$.
\end{prop}

Design graphs with parameter $c=1$ are particularly interesting. The previous proposition highlights them as being extremal for the girth among regular bipartite graph of diameter $3$. The next exercise is concerned with another extremal property. 

\begin{exer}\label{exer: extremal incidence}
Design graphs with parameter $c=1$ and degree $d$ are the ones that minimize the number of vertices among all $d$-regular graphs of girth $6$. 
\end{exer}

It seems appropriate to refer to design graphs with parameter $c=1$ as \emph{extremal design graphs}. There are some examples of extremal design graphs among the graphs that we already know. One is the cycle $C_6$. The other, more interesting, is the Heawood graph. In the next section, we will construct more design graphs, some of them extremal, by using finite fields.

A \emph{partial design graph} is a regular bipartite graph with the property that there are only two possible values for the number of neighbours shared by any two distinct vertices of the same colour. 

The parameters of a partial design graph are denoted $m$, $d$, respectively $c_1$ and $c_2$. Note that $c_1\neq c_2$, and that the roles of $c_1$ and $c_2$ are interchangeable.

\begin{ex} The bipartite double of a (non-bipartite) strongly regular graph is a design graph or a partial design graph. 
\end{ex}

\begin{ex} The cube graph $Q_n$ is a partial design graph. Indeed, consider the bipartition given by weight parity. Fix two distinct strings with the same weight parity. If they differ in two slots, then they have two common neighbours; otherwise, they have no common neighbour. Thus, $Q_n$ has parameters of  are $c_1=0$, $c_2=2$.
\end{ex}

\begin{ex}\label{ex: T-C} Consider the complete graph $K_6$. It has $15$ edges. A \emph{matching} is a choice of three edges with distinct endpoints, i.e., a partition of the six vertices into two-element subsets. There are $15$ matchings, as well. Define a bipartite graph by using the edges and the matchings as vertices, and connecting matchings to the edges they contain. This is the \emph{Tutte - Coxeter graph}, drawn in Figure~\ref{fig: tutte} below.

\begin{figure}[!ht]\medskip
\GraphInit[vstyle=Simple]
\tikzset{VertexStyle/.style = {
shape = circle,
inner sep = 0pt,
minimum size = .8ex,
draw}}
\centering
\begin{tikzpicture}[scale=.4, rotate=25]
\grLCF[RA=7]{-13,-9,7,-7,9,13}{5}
\end{tikzpicture}
\caption{The Tutte - Coxeter graph.}\label{fig: tutte}
\end{figure}
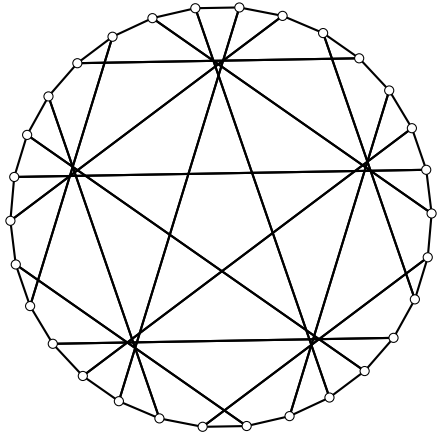
The Tutte - Coxeter graph is a partial design graph. Its half-size is $m=15$. The degree is $d=3$. Indeed, a matching is adjacent to three edges, and an edge is adjacent to three matchings. The joint parameters are $c_1=0$, $c_2=1$. Two distinct edges are adjacent to a unique matching if they have disjoint endpoints, and to no matching otherwise. Two distinct matchings have at most one edge in common.

A further property of the Tutte - Coxeter graph is that it has diameter $4$ and girth $8$. The verification is left to the reader.
\end{ex}

\bigskip
\section{Finite fields}
In this section, and in the next one, the focus is on finite fields. Among other things, we construct some interesting families of graphs, and we acquire some of the tools that later on will prove useful for studying these graphs. 

\bigskip
\subsection*{Notions}\label{sec: finite fields kit}\
\medskip

Let us start by recalling some fundamental facts about finite fields. 

A finite field has $q=p^d$ elements for some prime $p$, the characteristic of the field, and some positive integer $d$, the dimension of the field over its prime subfield. 

For each prime power $q$ there exists a field with $q$ elements, which is furthermore unique up to isomorphism. We think of $\Z_p=\Z/p\Z$ as `the' field with $p$ elements. In general, `the` field with $q=p^d$ elements can be realized as a quotient $\Z_p[X]/(f(X))$, where $f(X)\in \Z_p[X]$ is an irreducible polynomial of degree $d$. Such a polynomial $f$ exists for each given $d$; however, no general recipe is known for producing one. 

The multiplicative group of a finite field is cyclic. This is, again, a non-constructive existence: given a field, there is no known recipe for producing a multiplicative generator.

We now turn to extensions of finite fields. The following result plays a key role in this respect.

\begin{thm}
Let $\K$ be a field with $q^n$ elements. Then the map $\phi: \K\to \K$, given by $\phi(a)=a^q$, has the following properties:
\begin{itemize}
\item[(i)] $\phi$ is an automorphism of $\K$ of order $n$;
\item[(ii)] $\F=\{a\in \K: \phi(a)=a\}$ is a field with $q$ elements, and $\phi$ is an $\F$-linear isomorphism when $\K$ is viewed as a linear space over $\F$.
\end{itemize}
\end{thm}

\begin{proof}
(i) Clearly, $\phi$ is injective, multiplicative, and $\phi(1)=1$. To see that $\phi$ is additive, we iterate the basic identity $(a+b)^p=a^p+b^p$, where $p$ is the characteristic of $\K$, up to $(a+b)^q=a^q+b^q$. Thus $\phi$ is automorphism of $\K$. Each $a\in \K^*$ satisfies the relation $a^{q^n-1}=1$, so $a^{q^n}=a$ for all $a\in \K$. Thus $\phi^n$ is the identity map on $\K$. Assuming that $\phi^t$ is the identity map on $\K$ for some $0<t<n$, we would get that $X^{q^t}=X$ has $q^n$ solutions in $\K$, a contradiction.

(ii) $\F$ is a subfield since $\phi$ is an automorphism. As $\K^*$ is cyclic of order $q^n-1$, and $q-1$ divides $q^n-1$, there are precisely $q-1$ elements $a\in \K^*$ satisfying $a^{q-1}=1$. Therefore $\F$ has $q$ elements. Finally, note that $\phi(ab)=\phi(a)\phi(b)=a\phi(b)$ whenever $a\in \F$ and $b\in \K$.
\end{proof}

Turning the above theorem on its head, we can start with a field $\F$ with $q$ elements, and then consider a field $\K$ with $q^n$ elements as an overfield of $\F$. Then $\K$ is an $n$-dimensional linear space over $\F$, and we say that $\K$ is an \emph{extension} of $\F$ of \emph{degree} $n$. The map $\phi$ is called the \emph{Frobenius automorphism} of $\K$ over $\F$.

\begin{exer} \label{exer: max order} Let $\F$ be a field with $q$ elements. Show that the maximal order of an element in the general linear group $\GL_n(\F)$ is $q^n-1$.
\end{exer}

Let $\F$ be a field with $q$ elements, and let $\K$ be an extension of $\F$ of degree $n$. The \emph{trace} and the \emph{norm} of an element $a\in \K$ are defined as follows:
\begin{align*}
\mathrm{Tr}(a)=\sum_{k=0}^{n-1} \phi^k(a)=a+a^q+\ldots+a^{q^{n-1}}, \qquad \mathrm{N}(a)=\prod_{k=0}^{n-1} \phi^k(a)=a\cdot a^q\cdot \ldots\cdot a^{q^{n-1}}
\end{align*}

One might think of the trace and the norm as the additive, respectively the multiplicative content of a Frobenius orbit. The importance of the trace and the norm stems from the following properties.

\begin{thm}
The trace is additive, in fact $\F$-linear, while the norm is multiplicative. The trace and the norm map $\K$ onto $\F$.
\end{thm}

\begin{proof}
The first statement is obvious, so let us turn to the second statement. The trace and the norm are Frobenius-invariant, $\phi(\mathrm{Tr}(a))=\mathrm{Tr}(a)$ and $\phi(\mathrm{N}(a))=\mathrm{N}(a)$, so they are $\F$-valued. Consider $\mathrm{Tr}: \K\to \F$ as an additive homomorphism. The size of its kernel is at most $q^{n-1}$, so the size of its image is at least $q^n/q^{n-1}=q$. Therefore $\mathrm{Tr}$ is onto. Similarly, consider $\mathrm{N}: \K^*\to \F^*$ as a multiplicative homomorphism. Its kernel has size at most $(q^n-1)/(q-1)$, so the size of its image is at least $q-1$. A fortiori, $\mathrm{N}$ is onto.
\end{proof}

\begin{exer} \label{exer: hilbert} Show that the kernels of the trace and the norm maps can be described as follows: 
\begin{align*}
\{a\in \K: \mathrm{Tr}(a)=0\}=\{b^q-b: b\in \K\}, \qquad \{a\in \K: \mathrm{N}(a)=1\}=\{b^q/b: b\in \K^*\}
\end{align*}
\end{exer}

\begin{ex}\label{ex: quadext} Let us discuss in some detail the simple case of degree $2$, or \emph{quadratic}, extensions. We give a fairly concrete picture for such an extension, its Frobenius automorphism, and the associated trace and norm maps. We do this in odd characteristic, and we invite the reader to provide a similar discussion in even characteristic.

Let $\F$ be a finite field of odd characteristic. The squaring map $\F\to \F$, given by $x\mapsto x^2$, is not injective hence not surjective as well. Let $d\in \F$ be a non-square, that is, an element not in the range of the squaring map. Then the quadratic polynomial $X^2-d\in \F[X]$ is irreducible, and
\begin{align*}
\K=\F[X]/(X^2-d)=\{x+y\sqrt{d}: x,y\in \F\}
\end{align*}
is a quadratic extension of $\F$. We use $\sqrt{d}$ to denote the class of $X$ in $\F[X]/(X^2-d)$. 

The Frobenius automorphism of $\K$ over $\F$ is determined by what it does on $\sqrt{d}$. As $\phi(\sqrt{d})^2=\phi(d)=d=(\sqrt{d})^2$, we infer that $\phi(\sqrt{d})=\pm \sqrt{d}$. Now $\sqrt{d}$ cannot be fixed by $\phi$, since $\sqrt{d}\notin\F$, so $\phi(\sqrt{d})=- \sqrt{d}$. Therefore $\phi$ is the `conjugation':
\begin{align*}
\phi(x+y\sqrt{d})= x-y\sqrt{d}
\end{align*}
The trace and the norm are then given as follows:
\begin{align*}
\mathrm{Tr}(x+y\sqrt{d})&=(x+y\sqrt{d})+(x-y\sqrt{d})=2x\\
\mathrm{N}(x+y\sqrt{d})&=(x+y\sqrt{d})(x-y\sqrt{d})=x^2-dy^2
\end{align*}
\end{ex}


\bigskip
\subsection*{Projective combinatorics} \label{sec: projective}\
\medskip

Let $\F$ be a field with $q$ elements, and consider a linear  space $V$ of dimension $n$ over $\F$. We think of $V$ as an ambient space, and we investigate the geometry and combinatorics of its subspaces. A $k$-dimensional subspace of $V$ is called a $k$-space in what follows.

\begin{prop}\label{prop: subs}
The number of $k$-spaces is given by the $q$-binomial coefficients:
\begin{align*}
\binom{n}{k}_{\! q}:=\frac{(q^n-1)\dots(q^{n}-q^{k-1})}{(q^k-1)\dots(q^k-q^{k-1})}=\frac{(q^n-1)\dots(q^{n-k+1}-1)}{(q^k-1)\dots(q-1)}
\end{align*}
\end{prop}

\begin{proof}
There are $(q^n-1)\dots(q^{n}-q^{k-1})$ ordered ways of choosing $k$ linearly independent vectors. Some of these choices span one and the same $k$-space. Namely, a $k$-space has $(q^k-1)\dots(q^k-q^{k-1})$ ordered bases. This corresponds to the case $n=k$ of the previous count. 
\end{proof}

Applying the previous proposition to a quotient space of $V$, we get the following useful consequence.

\begin{cor}\label{cor: chainy}
The number of $m$-spaces containing a given $k$-space is $\binom{n-k}{m-k}_{\! q}$.
\end{cor}

For $q\to \infty$, the $q$-binomial coefficients satisfy $\binom{n}{k}_{\! q}\sim q^{k(n-k)}$. If, on the other hand, we formally let $q\to 1$ then the $q$-binomial coefficients turn into the usual binomial coefficients. The familiar symmetry of the latter, $\binom{n}{k}=\binom{n}{n-k}$, has a $q$-analogue: 
\begin{align*}
\binom{n}{k}_{\! q}=\binom{n}{n-k}_{\! q}
\end{align*} 
This can be checked directly, using the formula displayed above. The more conceptual explanation, however, would be a duality between subspaces of complementary dimensions. This can be set up with the help of a scalar product.

A \emph{scalar product} on $V$ is a bilinear form $\la\cdot,\cdot\ra: V\times V\to \F$ which is symmetric, $\la v,w\ra=\la w,v\ra$ for all $v,w\in V$, and non-degenerate, meaning that $\la v,w\ra=0$ for all $w\in V$ implies $v=0$. The obvious example is the standard scalar product on $\F^n$, given by $\la x,y\ra=x_1y_1+\ldots+x_ny_n$. For a more interesting example, let $\K$ be a degree $n$ extension of $\F$. Then $\K$ can be endowed with the tracial scalar product $\la a,b\ra_\tau=\mathrm{Tr}(ab)$. In general, any ambient space $V$ can be equipped with a scalar product.

A scalar product brings in the notion of orthogonality. In particular, for a subspace $W$ of $V$ we can define 
\begin{align*}
W^\perp=\{v\in V: \la v,w\ra=0 \textrm{ for all } w\in W\}
\end{align*}
which is again a subspace of $V$. Over the real or complex numbers, we would call $W^\perp$ the orthogonal complement of $W$. Over a finite field, however, there is a catch: $W^\perp$ may fail to be a direct summand of $W$, because a non-zero vector can be orthogonal to itself. This is particularly visible in the case of the standard scalar product on $\F^n$. Nevertheless, the following holds.

\begin{prop}\label{prop: ortho}
Let $V$ be endowed with a scalar product. Then the assignment $W\mapsto W^\perp$ is an involutive correspondence between the $k$-spaces and the $(n-k)$-spaces of $V$.
\end{prop}

\begin{proof}
The main claim is that $\dim W+\dim W^\perp=n$. Then $W\mapsto W^\perp$ sends $k$-spaces to $(n-k)$-spaces. Also, $\dim W^{\perp\perp}=\dim W$. As $W\subseteq W^{\perp\perp}$, we conclude that, in fact, $W=W^{\perp\perp}$. That is, $W\mapsto W^\perp$ is an involution.

Consider the linear map $\psi:V\to V^*$, where $V^*$ is the dual of $V$, given by $v\mapsto \la v,\cdot\ra$. Non-degeneracy of the scalar product means that $\psi$ is injective, hence an isomorphism for dimension reasons. Let $\psi_W:V\to W^*$ be the linear map sending $v\in V$ to the restriction $\psi(v)|_W$. On the one hand, $W^\perp=\mathrm{ker}(\psi_W)$. On the other hand $\mathrm{im}(\psi_W)=W^*$, as $\psi_W$ is the composition of the isomorphism $\psi$ with the onto map $V^*\to W^*$ given by restriction. Now the identity $\dim V=\dim\mathrm{ker}(\psi_W)+\dim\mathrm{im}(\psi_W)$ yields our claim.\end{proof}

The map $W\mapsto W^\perp$ is also order-reversing with respect to inclusion: $W_1\subseteq W_2$ implies $W_2^\perp\subseteq W_1^\perp$. This fact can be used for an alternate proof of Corollary~\ref{cor: chainy}. The number of $m$-spaces containing a given $k$-space equals the number of $(n-m)$-spaces contained in a given $(n-k)$-space, and this is $\binom{n-k}{n-m}_{\! q}=\binom{n-k}{m-k}_{\! q}$.

\bigskip
\subsection*{Incidence graphs}\label{sec: incidence graphs}\
\medskip

Let $V$ be an ambient linear space of dimension $n\geq 3$ over a finite field $\F$ with $q$ elements. The \emph{incidence graph} $I_n(q)$ is the bipartite graph whose vertices are the $1$-spaces respectively the $(n-1)$-spaces, and whose edges connect $1$-spaces to $(n-1)$-spaces containing them. Note that the construction only depends on the dimension of $V$, and not on $V$ itself; hence the notation.

\begin{thm}
The incidence graph $I_n(q)$ is a design graph, with parameters
\begin{align*}
m=\binom{n}{1}_{\! q}=\frac{q^n-1}{q-1}, \qquad d=\binom{n-1}{1}_{\! q}=\frac{q^{n-1}-1}{q-1}, \qquad c=\binom{n-2}{1}_{\! q}=\frac{q^{n-2}-1}{q-1}.
\end{align*}
\end{thm}

\begin{proof} 
By definition, the incidence graph $I_n(q)$ is bipartite. The half-size $m$, as well as the degree $d$ for each vertex, are given by Proposition~\ref{prop: subs}, Corollary~\ref{cor: chainy} and the symmetry property of the $q$-binomial coefficients. 

To check the design property, we use the dimensional formula 
\begin{align*}
\dim(W+W')+\dim(W\cap W')=\dim W+\dim W'
\end{align*} 
where $W$ and $W'$ are subspaces of $V$. 

Let $W$ and $W'$ be distinct $(n-1)$-spaces. The common neighbours of $W$ and $W'$ are the $1$-spaces contained in $W\cap W'$. Now $W+W'$ has dimension $n$, so $W\cap W'$ is an $(n-2)$-space. The number of $1$-spaces contained in an $(n-2)$-space is $\binom{n-2}{1}_{\! q}$.
 
Similiarly, let $W$ and $W'$ be distinct $1$-spaces. The common neighbours of $W$ and $W'$ are the $(n-1)$-spaces containing $W+W'$.  This is a $2$-space, since $W\cap W'$ is $0$-dimensional. The number of $(n-1)$-spaces containing a given $2$-space is $\binom{n-2}{1}_{\! q}$. 
\end{proof}

\begin{ex}\label{ex: I3}
The incidence graph $I_3(q)$ is an extremal design graph. Its half-size is $m=q^2+q+1$, and it is regular of degree $d=q+1$. For $q=2$, this is the Heawood graph.
\end{ex}

Assume now that $V$ is endowed with a scalar product. Proposition~\ref{prop: ortho} provides a bijection between the $(n-1)$-spaces and the $1$-spaces. Using this bijection, we can give an alternate, and somewhat simpler, description of $I_n(q)$. Take two copies of the set of $1$-spaces, that is, lines through the origin, and join a `black' line to a `white' line whenever they are orthogonal. Note that the orthogonal picture is independent of the choice of scalar product, for it agrees with the original incidence picture.

The orthogonal picture immediately reveals that $I_n(q)$ is an induced subgraph of $I_{n+1}(q)$. Indeed,  view $\F^n$ as the subspace of $\F^{n+1}$ having vanishing last coordinate. Then the lines of $\F^n$ are also lines in $\F^{n+1}$, and they are orthogonal in $\F^n$ if and only if they are orthogonal in $\F^{n+1}$.

The incidence graph $I_n(q)$ enjoys a regularity property that is not obvious at first sight. Here, the orthogonal picture turns out to be very useful.

\begin{thm}
The incidence graph $I_n(q)$ is a Cayley graph.
\end{thm}

\begin{proof}
Let $\mathbb{K}$ be an extension of $\F$ of degree $n$, endowed with the tracial scalar product. The lines through the origin can be identified with the quotient group $\K^*/\F^*$. So we may think of $I_n(q)$ as the bipartite graph on two copies of $\K^*/\F^*$, with an edge joining $[a]_\bl$ to $[b]_\wh$ whenever $\mathrm{Tr}(ab)=0$. But this is the bi-Cayley graph of $\K^*/\F^*$ with respect to the subset $S=\{[s]\in G: \Tr(s)=0\}$. Since $\K^*/\F^*$ is cyclic, this bi-Cayley graph can be viewed as a Cayley graph of a dihedral group. \end{proof}

\begin{exer}\label{exer: singer} Let $q$ be, as usual, a power of a prime. Show that there are $q+1$ numbers among $\{1,2,\dots,q^2+q+1\}$ with the following property: if someone picks two different numbers from the list, and tells you their difference, then you know the numbers. 
\end{exer}

\begin{notes} 
Example~\ref{ex: I3} addresses the problem of constructing extremal design graphs. Namely, it says that there are `standard' examples in degree $q+1$, for each prime power $q$. Are there any other examples? It turns out that there are non-standard examples in degree $q+1$, for each $q\geq 9$ which is a prime power but not a prime. The constructions are, again, of algebraic nature. However, there are no known examples, and there are probably none, in degree $n+1$, when $n$ is not a prime power! This is, up to a slight reformulation, one of the most famous problems in combinatorics. The common formulation concerns combinatorial objects known as \emph{finite projective planes}. Extremal design graphs are the incidence graphs of finite projective planes. A beautiful theorem of Bruck and Ryser says the following: if there are extremal design graphs of degree $n+1$, and $n\equiv 1,2$ mod $4$, then $n$ is a sum of two square. For example, the Bruck - Ryser criterion implies that there are no examples in degree $7$, but it says nothing about the case of degree $11$. In 1989, after massive computer calculations, it was eventually shown that there are no examples in degree $11$. Or was it? See Lam (\emph{The Search for a Finite Projective Plane of Order 10}, Amer. Math. Monthly 1991).

Exercise~\ref{exer: singer} is based on a theorem of Singer (\emph{A theorem in finite projective geometry and some applications to number theory}, Trans. Amer. Math. Soc. 1938).
\end{notes}

\bigskip
\section{Squares in finite fields}
The previous section was concerned, for the most part, with linear algebra \emph{over} a finite field. In this section, we mostly work \emph{within} a finite field $\F$ with $q$ elements. Throughout, we assume that $q$ is odd.

\bigskip
\subsection*{Around squares} \label{sec: around squares}\
\medskip

The squaring homomorphism $\F^*\to \F^*$, given by $x\mapsto x^2$, is two-to-one: $x^2=y^2$ if and only if $x=\pm y$. So half the elements of $\F^*$ are squares, and half are not. The \emph{quadratic signature} $\sigma: \F^*\to \{\pm 1\}$ is given by
\begin{align*}
\sigma(a)=\begin{cases} 1 & \textrm{ if } a \textrm{ is a square in }\F^*,\\-1 & \textrm{ if } a \textrm{ is not a square in }\F^*.\end{cases}
\end{align*}

\begin{thm}
The quadratic signature $\sigma$ is multiplicative on $\F^*$, and it is explicitly given by the `Euler formula' 
\begin{align*}
\sigma(a)= a^{(q-1)/2}.
\end{align*}
\end{thm}

\begin{proof} 
Let $\tau(a)=a^{(q-1)/2}$ for $a\in \F^*$. Note that $\tau(a)=\pm 1$, as $\tau(a)^2=a^{q-1}=1$. Note also that $\tau$ does take the value $-1$, for otherwise the equation $X^{(q-1)/2}=1$ would have too many solutions in $\F$. Thus $\tau:\F^*\to \{\pm 1\}$ is an onto homomorphism. Its kernel has size $\frac{1}{2}(q-1)$, and it contains the non-zero squares, whose number is also $\frac{1}{2}(q-1)$. Therefore the kernel of $\tau$ consists precisely of the non-zero squares. In other words, $\tau(a)=1$ if $a$ is a square in $\F^*$, and $\tau(a)=-1$ if $a$ is not a square in $\F^*$, that is, $\tau=\sigma$.
\end{proof}

The question whether $-1$ is a square or not is an important one. An application of Euler's formula yields the following.

\begin{cor}\label{prop: -1 as a square}
$-1$ is a square in $\F$ if and only if $q\equiv 1$ mod $4$.
\end{cor}

If $\K$ is an extension of $\F$, how is being a square in $\F$ related to being a square in $\K$? This question, too, can be settled by applying Euler's formula. Let $n$ denote the degree of the extension $\K/\F$. Then for each $a\in \F^*$ we have that
\begin{align*}
\sigma_\K(a)=a^{(q^n-1)/2}=\sigma_\F(a)^{1+q+\ldots+q^{n-1}}=\sigma_\F(a)^n.
\end{align*}
Thus, the following holds:

\begin{cor}
If $\K$ is an extension of even degree, then all elements of $\F$ are squares in $\K$. If $\K$ is an extension of odd degree, then an element of $\F$ is a square in $\K$ if and only if it is already a square in $\F$.
\end{cor}

There are as many squares as non-squares in $\F^*$, so 
\begin{align*}
\sum_{a\in \F^*} \sigma(a)=0.
\end{align*}
The quadratic signature $\sigma$ is extended to $\F$ by setting $\sigma(0)=0$. With this convention, $\sigma$ continues to be multiplicative and to satisfy the Euler formula. The above summation formula holds as well over $\F$. 

The next lemma is a computation involving the quadratic signature. We will use it as an ingredient in proving several interesting results.

\begin{lem}\label{lem: convolution} Put
\begin{align*}
J_c=\sum_{a+b=c} \sigma(a)\:\sigma(b).
\end{align*}
Then $J_c=-\sigma(-1)$ if $c\neq 0$, and $J_0=\sigma(-1) \:(q-1)$.
\end{lem}

\begin{proof}
As $ \sigma(a)\:\sigma(-a)=\sigma(-1)\:\sigma(a^2)=\sigma(-1)$ for each $a\in \F^*$, we get
\begin{align*}
J_0=\sum_{a} \sigma(a)\:\sigma(-a)= \sigma(-1)\: (q-1).
\end{align*} 
For $c\neq 0$, a change of variables yields
\begin{align*}
J_c=\sum_{a+b=1} \sigma(ac)\:\sigma(bc)= J_1.
\end{align*}
The value of $J_1$ can be computed in two ways. Directly, we may write $\sigma(a)\:\sigma(1-a)=\sigma(a-a^2)=\sigma(a^2)\:\sigma(a^{-1}-1)=\sigma(a^{-1}-1)$ for each $a\in \F^*$ to get
\begin{align*}
J_1=\sum_{a} \sigma(a)\:\sigma(1-a)=\sum_{a\neq 0}  \sigma(a^{-1}-1)=-\sigma(-1).
\end{align*} 
Indirectly, we observe that 
\begin{align*}
\sum_c J_c=\Big(\sum_a \sigma(a)\Big)\Big(\sum_b \sigma(b)\Big)=0
\end{align*}
so $J_0+(q-1)\: J_1=0$. It follows that $J_1=-\sigma(-1)$, by using the value of $J_0$.
\end{proof}

The first application of Lemma~\ref{lem: convolution} concerns the equation $aX^2+bY^2=1$, for $a,b\in \F^*$. A simple counting argument shows that the equation is solvable in $\F$: the two subsets $\{ax^2: x\in \F\}$ and $\{1-by^2: y\in \F\}$ have size $\tfrac{1}{2}(q+1)$, so they must intersect. 

\begin{prop}\label{prop: conic}
The equation $aX^2+bY^2=1$ has $q-\sigma(-ab)=q\pm 1$ solutions. In particular, the equation $X^2+Y^2=c$ has $q-\sigma(-1)$ solutions, for every $c\in \F^*$.
\end{prop}

\begin{proof}[First proof] Let $N(f)$ denote the number of solutions in $\F$ of a polynomial equation $f=0$. Note that $N(X^2-c)=1+\sigma(c)$; more generally, $N(aX^2-c)=1+\sigma(ac)$ for $a\neq 0$. Then:
\begin{align*}
N(aX^2+bY^2-1)&=\sum_{r+s=1} N(aX^2-r) \: N(bY^2-s)\\
&=\sum_{r+s=1} \big(1+\sigma(ar)\big)\: \big(1+\sigma(bs)\big)
\end{align*}
We have 
\begin{align*}
\sum_{r+s=1} 1=q, \qquad \sum_{r+s=1} \sigma(r)=\sum_{r+s=1} \sigma(s)=0,\end{align*} 
so we get
\begin{align*}
N(aX^2+bY^2-1)&=q+\sigma(ab)\:\sum_{r+s=1}\sigma(r)\:\sigma(s)\\
&=q-\sigma(-ab).
\end{align*}
The second part follows by taking $a=b=1/c$.
\end{proof} 

\begin{proof}[Second proof] 
Dividing the equation $aX^2+bY^2=1$ by $a$, and setting $d=-ba^{-1}$, $c=a^{-1}$, we reach the equation $X^2-dY^2=c$. We wish to show that, for each $d\in \F^*$, the latter equation has $q-\sigma(d)$ solutions independently of $c\in \F^*$.

Consider the case when $d$ is a square. After writing $d=d_0^2$ for some $d_0\in \F^*$, the equation $X^2-dY^2=c$ becomes $(X+d_0Y)(X-d_0Y)=c$. This, in turn, is equivalent to solving the system
\begin{align*}
X+d_0Y=s, \quad X-d_0Y=cs^{-1}
\end{align*}
for each $s\in \F^*$. This system has a unique solution, namely $x=(s+cs^{-1})/2$ and $y=d_0^{-1}(s-cs^{-1})/2$. Note here that division by $2$ is allowed, since the characteristic is odd. We conclude that, in this case, $X^2-dY^2=c$ has $q-1$ solutions.

Consider now the case when $d$ is not a square. Let $\K=\F[X]/(X^2-d)$ be the quadratic extension of $\F$ discussed in Example~\ref{ex: quadext}. Then the solutions of $X^2-dY^2=c$ correspond, via $(x,y)\mapsto x+y\sqrt{d}$, to the elements of $\K^*$ having norm $c$. Since the norm map is an onto homomorphism from $\K^*$ to $\F^*$, there are $|\K^*|/|\F^*|=(q^2-1)/(q-1)=q+1$ such elements for each $c$. To conclude, $X^2-dY^2=c$ has $q+1$ solutions in this case.
\end{proof}

\begin{exer}\label{exer: spheres}
Find the number of solutions to the equation $X_1^2+\ldots+X_n^2=c$, where $c\in \F^*$. 
\end{exer}

\bigskip
\subsection*{Two arithmetic applications}\
\medskip

\begin{lem}[Gauss]\label{Gauss lemma}
Let $p$ be an odd prime, and let $\zeta\in \C$ be a $p$-th root of unity. Then the cyclotomic integer
\begin{align*}
G=\sum_{a\in \Z_p} \sigma(a)\: \zeta^a \in \Z[\zeta]
\end{align*}
satisfies $G^2=\sigma(-1)\: p$.
\end{lem}

\begin{proof}
We have $\zeta^p=1$, so $1+\zeta+\ldots+\zeta^{p-1}=0$. Using Lemma~\ref{lem: convolution} along the way, we compute
\begin{align*}
G^2&=\sum_{a,b} \sigma(a)\:\sigma(b)\:  \zeta^{a+b}=\sum_c \Big(\sum_{a+b=c} \sigma(a)\:\sigma(b)\Big)\: \zeta^c\\
&= \sigma(-1)\: (p-1)-\sigma(-1)\: \Big(\sum_{c\neq 0} \zeta^c\Big)= \sigma(-1)\: p,
\end{align*}
as desired.
\end{proof}

The above result was a lemma, so we might guess that something bigger lurks around the corner. Indeed, it is the law of quadratic reciprocity.

\begin{thm}[Euler, Legendre, Gauss]\label{thm: QR}
Let $p$ and $\ell$ be distinct odd primes. If $p\equiv 1$ mod $4$ or $\ell\equiv 1$ mod $4$, then $p$ is a square mod $\ell$ if and only if $\ell$ is a square mod $p$. If $p,\ell\equiv 3$ mod $4$, then $p$ is a square mod $\ell$ if and only if $\ell$ is not a square mod $p$.
\end{thm}

The \emph{Legendre symbol} for an odd prime $p$ is defined as
\begin{align*}
(a/p)=
\begin{cases}
1 & \textrm{if } a \textrm{ is a square mod }p\\
-1 & \textrm{if } a \textrm{ is not a square mod }p\end{cases}
\end{align*} 
when $a$ is an integer relatively prime to $p$, respectively $(a/p)=0$ when $a$ is an integer divisible by $p$. Clearly, $(a/p)$ only depends on the residue class $a$ mod $p$. On $\Z_p$, 
the Legendre symbol $(\cdot/p)$ is just another notation for the quadratic signature. In terms of the Legendre symbol, the law of quadratic reciprocity can be compactly stated as follows: for distinct odd primes $p$ and $\ell$ we have
\begin{align}
(p/\ell)\: (\ell/p)=(-1)^{(p-1)(\ell-1)/4}.\tag{QR}
\end{align}
A convenient reformulation, especially for the proofs below, is
\begin{align*}
(\ell/p)=(p^*/\ell), \qquad p^*=(-1/p)\:p\tag{QR*}.
\end{align*}
 Indeed, $(-1/p)= (-1)^{(p-1)/2}$ so $((-1/p)/\ell)= (-1)^{(p-1)(\ell-1)/4}$.

\begin{proof}[Proof of Theorem~\ref{thm: QR}] Let $G$ be defined as in the previous lemma, with respect to $p$. We compute $G^\ell$ mod $\ell$ in two ways. On the one hand, we have
\begin{align*}
G^\ell=\Big(\sum_{a\in \Z_p}  (a/p)\: \zeta^a\Big)^\ell \equiv \sum_{a\in \Z_p} \big((a/p)\: \zeta^a\big)^\ell \mod \ell,
\end{align*}
and 
\begin{align*}
\sum_{a\in \Z_p} \big((a/p)\: \zeta^a\big)^\ell=\sum_{a\in \Z_p} (a/p)\: \zeta^{a\ell}=\sum_{a\in \Z_p} (a\ell^{-1}/p)\: \zeta^{a}=(\ell/p)\: G.
\end{align*}
In short, $G^\ell\equiv (\ell/p)\: G$ mod $\ell$. 

On the other hand, by Gauss's Lemma we may write
\begin{align*}
G^{\ell-1}=(G^2)^{(\ell-1)/2}=(p^*)^{(\ell-1)/2},
\end{align*}
and 
\begin{align*}
(p^*)^{(\ell-1)/2}\equiv (p^*/\ell)\mod \ell
\end{align*}
by Euler's formula. It follows that $G^\ell\equiv (p^*/\ell)\: G$ mod $\ell$.

The two congruences for $G^\ell$ imply that $(\ell/p)=(p^*/\ell)$. For otherwise, we would have $G\equiv -G$, that is $2G\equiv 0$ mod $\ell$. Squaring, and using Gauss's Lemma once again, leads to $\pm 4p \equiv 0$ mod $\ell$. This is a congruence in the cyclotomic ring $\Z[\zeta]$, and it means that the rational number $4p/\ell$ is in $\Z[\zeta]$. In particular, $4p/\ell$ is an algebraic integer. Now, if a rational number is an algebraic integer, then it is an integer. This is a contradiction, as $4p/\ell$ is certainly not an integer.
\end{proof}

The last step of the previous proof relied on the notion of algebraic integer. Let us give a variation which avoids this notion. The idea is the following: instead of working with $G$ modulo $\ell$ in characteristic $0$, we make sense of $G$ in characteristic $\ell$.

\begin{proof}[Proof of Theorem~\ref{thm: QR}, variation.] Let $\K$ be an extension of $\Z_\ell$ containing a $p$-th root of unity $z$. For this to happen, $p$ must divide $|\K^*|=\ell^n-1$, where $n$ is the degree of $\K$. So we may take $n=p-1$, thanks to Fermat's little theorem. 

As $z^p=1$, we may consider
\begin{align*}
g=\sum_{a\in \Z_p} (a/p)\: z^a \in \K.
\end{align*}
The same argument as in Gauss's lemma shows that $g^2=p^*$. So $p^*$ is a square in $\Z_\ell$ precisely when $g$ is in $\Z_\ell$. The latter holds if and only if $g^\ell=g$, that is, $g$ is fixed by the Frobenius automorphism of the extension $\K/\Z_\ell$. Now
\begin{align*}
g^\ell=\sum_{a\in \Z_p} (a/p)\: z^{a\ell}=\sum_{a\in \Z_p} (a\ell^{-1}/p)\: z^{a}=(\ell/p)\: g.
\end{align*}
In summary, $(p^*/\ell)=1$ if and only if $(\ell/p)=1$. This means that $(\ell/p)=(p^*/\ell)$.
\end{proof}

The law of quadratic reciprocity is complemented by the following: if $p$ is an odd prime, then $2$ is a square mod $p$ if and only if $p\equiv \pm 1$ mod $8$. This can be proved by looking at $(1+i)^p$ mod $p$ in two different ways. The details are left to the reader.

\begin{exer}\label{exer: lebesgue} Let $p$ and $\ell$ be distinct odd primes. Give a different proof for the law of quadratic reciprocity, $(\ell/p)=(p^*/\ell)$, by arguing as follows. Let $N$ denote the number of solutions to the equation $X_1^2+\ldots+X_\ell^2=1$ in $\Z_p$. Show that $N\equiv 1+(p^*/\ell)$ mod $\ell$, but also $N\equiv 1+(\ell/p)$ mod $\ell$.
\end{exer}

The earliest arithmetical jewel is Fermat's theorem that a prime power congruent to $1$ mod $4$ is a sum of two integral squares. Facts as old as this one have many elegant proofs. The following result, our third application of Lemma~\ref{lem: convolution}, gives an \emph{explicit} proof of Fermat's theorem.

\begin{thm}[Jacobsthal]\label{thm: jaco}
Assume that $q\equiv 1$ mod $4$. For $a\in \F^*$, consider the sum
\begin{align*}
S(a)=\sum_{x\in \F} \sigma(x^3+ax).
\end{align*}
Then the following hold:
\begin{itemize}
\item[(i)] $S(a)$ is an even integer, whose absolute value only depends on whether $a$ is a square or not;
\item[(ii)] if $2A$ and $2B$ denote the values of $|S(\cdot)|$ on squares, respectively on non-squares, then the positive integers $A$ and $B$ satisfy $A^2+B^2=q$. 
\end{itemize}
\end{thm}

\begin{proof} (i) The map $x\mapsto \sigma(x^3+ax)$ is symmetric: $\sigma((-x)^3+a(-x))=\sigma(x^3+ax)$, as $\sigma(-1)=1$. Therefore $S(a)$ is even. To see that $|S(a)|$ only depends on whether $a$ is a square or not, we compute for $z\neq 0$ that
\begin{align*}
S(az^2)=\sum \sigma(x^3+az^2x)\stackrel{x\mapsto xz}{=} \sigma(z^3)\sum \sigma(x^3+ax) = \sigma(z) \: S(a).
\end{align*}
(ii) On the one hand, we have
\begin{align*}
\sum_a S(a)^2=\frac{q-1}{2}\big((2A)^2+(2B)^2\big)=2(q-1)(A^2+B^2).
\end{align*}
Here, we have included the harmless sum $S(0)=\sum \sigma(x^3)=\sum \sigma(x)=0$. 

On the other hand, we write
\begin{align*}
\sum_{a} S(a)^2= \sum_{a}\sum_{x,y} \sigma(x^3+ax)\: \sigma(y^3+ay)=\sum_{x,y} \sigma(xy)\:\Big( \sum_{a} \sigma(x^2+a)\: \sigma(y^2+a)\Big).
\end{align*}
At this point, we bring in Lemma~\ref{lem: convolution}:
\begin{align*}
\sum_a \sigma(x^2+a)\:\sigma(y^2+a)=\sigma(-1)\:\sum_a \sigma(x^2+a)\:\sigma(-y^2-a)=
\begin{cases}
q-1 & \textrm{ if } x^2-y^2= 0,\\ 
-1& \textrm{ if } x^2-y^2\neq 0. 
\end{cases}
\end{align*}
Thus
\begin{align*}
\sum_{a} S(a)^2=(q-1) \sum_{x^2=y^2} \sigma(xy) - \sum_{x^2\neq y^2} \sigma(xy)=q\sum_{x^2=y^2} \sigma(xy)-\sum_{x,y}\: \sigma(xy).
\end{align*}
Here $\sum_{x,y}\: \sigma(xy)=\big(\sum_x\sigma(x)\big)\big(\sum_y\sigma(y)\big)=0$, and
\begin{align*}
\sum_{x^2= y^2} \sigma(xy)=\sum_{x} \sigma(x^2)+\sum_{x} \sigma(-x^2)=2(q-1)
\end{align*}
where we use again the fact that $\sigma(-1)=1$. We conclude that $A^2+B^2=q$.
\end{proof}

\begin{notes} Gauss's Lemma says that $G=\pm \sqrt{p}$ for $p\equiv 1$ mod $4$, respectively $G=\pm i\sqrt{p}$ for $p\equiv 3$ mod $4$. Gauss also showed that the correct choice of sign in each case turns out to be the plus sign. 

Quadratic reciprocity played a fundamental role in the history of number theory. Euler discovered the square laws experimentally, and he was able to prove the complementary laws concerning $(-1/p)$ and $(2/p)$. Legendre had a partial proof of the quadratic reciprocity law, a name which he introduced, along with the notation that we now call the Legendre symbol. In the course of his investigations, Legendre also put forth the claim that there are infinitely many primes congruent to $a$ mod $b$, for every pair or relatively prime integers $a$ and $b$. This is a fact that we now know as Dirichlet's theorem. The first complete proof of the quadratic reciprocity law is due to Gauss, and so is the second, the third, the fourth, the fifth and the sixth--all essentially different. He added two more proofs later on. More than two hundred proofs of the quadratic reciprocity law have been published by now. The one suggested in Exercise~\ref{exer: lebesgue} is due to V.-A. Lebesgue. 

The reference for Theorem~\ref{thm: jaco} is Jacobsthal (\emph{\"Uber die Darstellung der Primzahlen der Form $4n+1$ als Summe zweier Quadrate}, J. Reine Angew. Math. 1907).
\end{notes}

\bigskip
\subsection*{Paley graphs}\label{sec: paley graphs}\
\medskip

We usually view the non-zero squares of $\F$ in a multiplicative light, but here we will take an \emph{additive} perspective on them. The fact that the equation $X^2+Y^2=c$ is solvable for each $c\in \F^*$ means that every element in $\F^*$ is a sum of at most two non-zero squares. As for $0$, it is a sum of two non-zero squares precisely when $q\equiv 1$ mod $4$. If $q\equiv 3$ mod $4$, then $0$ is a sum of three non-zero squares since $-1$ is a sum of two non-zero squares. In summary, the set of non-zero squares $\F^{*2}$ forms an additive generating set for $\F$. On the other hand, $\F^{*2}$ is additively symmetric, that is, $\F^{*2}=-\F^{*2}$, if and only if $q\equiv 1$ mod $4$. These considerations motivate the following definition.

For $q\equiv 1$ mod $4$, the \emph{Paley graph} $P(q)$ is the Cayley graph of the additive group of $\F$ with respect to the set of non-zero squares $\F^{*2}$. 

The Paley graph $P(q)$ has $q$ vertices, and it is regular of degree $\tfrac{1}{2}(q-1)$. The first in the sequence of Paley graphs, $P(5)$, is the $5$-cycle. The next one, $P(9)$, is the product $K_3\times K_3$. 

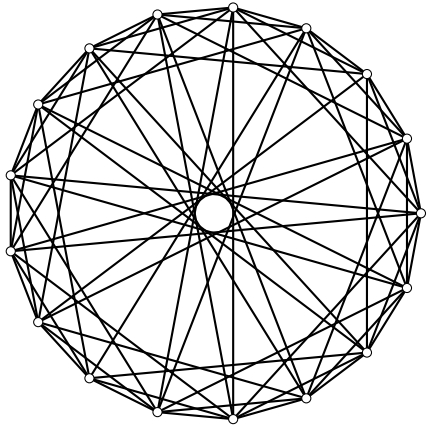
\begin{figure}[!ht]\medskip
\centering
\GraphInit[vstyle=Simple]
\tikzset{VertexStyle/.style = {
shape = circle,
inner sep = 0pt,
minimum size = .8ex,
draw}}
\begin{tikzpicture}[scale=.55]
  \grEmptyCycle[RA=5pt]{17}
  \Edges( a0,  a1,  a2,  a3,  a4,  a5,  a6,  a7,  a8)
  \Edges( a8,  a9, a10, a11, a12, a13, a14, a15, a16, a0)
  \Edges( a0,  a2,  a4,  a6,  a8, a10, a12, a14, a16)
  \Edges(a16,  a1,  a3,  a5,  a7,  a9, a11, a13, a15, a0)
  \Edges( a0,  a4,  a8, a12, a16,  a3,  a7, a11, a15)
  \Edges(a15,  a2,  a6, a10, a14,  a1,  a5,  a9, a13, a0)
  \Edges( a0,  a8, a16,  a7, a15)
  \Edges(a15,  a6, a14,  a5, a13)
  \Edges(a13,  a4, a12,  a3, a11)
  \Edges( a2, a10,  a1,  a9,  a0)
\end{tikzpicture}
\caption{The Paley graph $P(17)$.}
\end{figure}

\begin{thm}\label{Paley is strongly regular}
The Paley graph $P(q)$ is strongly regular: any two adjacent vertices have $a=\tfrac{1}{4}(q-5)$ common neighbours, and any two non-adjacent vertices have $c=\tfrac{1}{4}(q-1)$ common neighbours. 
\end{thm}

\begin{proof}
Let $v$ and $w$ be distinct vertices. The number of common neighbours is the number of representations of $v-w$ as an ordered sum of two non-zero squares in $\F$. By Proposition~\ref{prop: conic}, there are $q-1$ pairs $(x,y)\in \F\times \F$ satisfying $x^2+y^2=v-w$. If $v-w$ is not a square (i.e., $v$ and $w$ are not adjacent) then both $x$ and $y$ are non-zero, and $v-w$ is a sum of two non-zero squares in $\tfrac{1}{4}(q-1)$ ordered ways. If $v-w$ is a square (i.e., $v$ and $w$ are adjacent) then there are two solutions when $x=0$ and two solutions when $y= 0$, hence $q-5$ solutions for which both $x$ and $y$ are non-zero. Thus $v-w$ is a sum of two non-zero squares in $\tfrac{1}{4}(q-5)$ ordered ways.
\end{proof}

The field $\F$ has permutations coming from its two operations, and it is instructive to know their effect on the graph $P(q)$. The translation $+t: \F\to\F$ by $t\in \F$ is, of course, a graph automorphism. More interesting is the scaling $\times r: \F\to\F$ by $r\in \F^*$.  When $r$ is a square, this is a graph automorphism fixing the vertex $0$. When $r$ is a non-square, we get a `switch' changing edges to non-edges, respectively non-edges to edges.

\begin{exer}\label{exer: ramsey 4}
The \emph{Ramsey number} $R(k)$ is the smallest number with the property that, in any group of $R(k)$ persons, either there are $k$ of them who know each other, or there are $k$ of them who do not know each other. Show that $R(4)\geq 18$.
\end{exer}

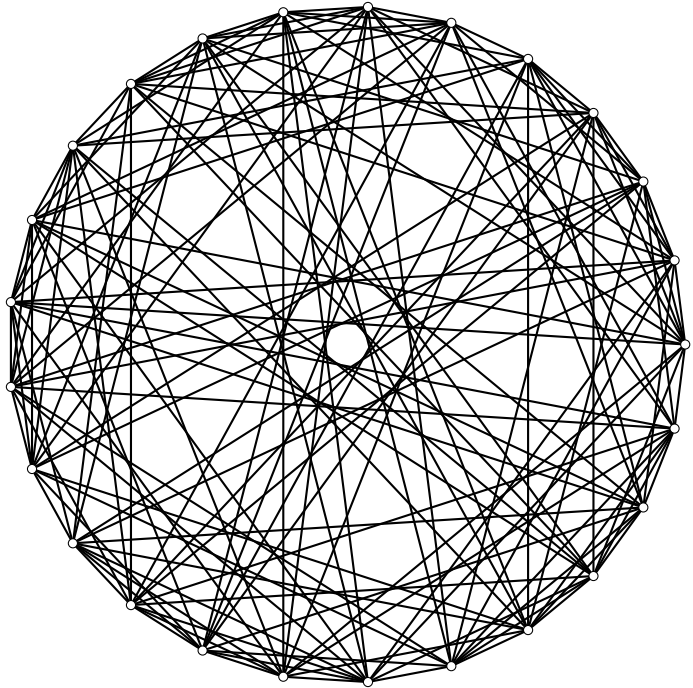
\begin{figure}[!ht]\medskip
\centering
\GraphInit[vstyle=Simple]
\tikzset{VertexStyle/.style = {
shape = circle,
inner sep = 0pt,
minimum size = .8ex,
draw}}
\begin{tikzpicture}[scale=.9]
  \grEmptyCycle[RA=5pt]{25}
\Edges(a1, a0)
\Edges(a2, a0)
\Edges(a2, a1)
\Edges(a3, a0)
\Edges(a3, a1)
\Edges(a3, a2)
\Edges(a4, a0)
\Edges(a4, a1)
\Edges(a4, a2)
\Edges(a4, a3)
\Edges(a5, a2)
\Edges(a5, a4)
\Edges(a6, a0)
\Edges(a6, a3)
\Edges(a6, a5)
\Edges(a7, a1)
\Edges(a7, a4)
\Edges(a7, a5)
\Edges(a7, a6)
\Edges(a8, a0)
\Edges(a8, a2)
\Edges(a8, a5)
\Edges(a8, a6)
\Edges(a8, a7)
\Edges(a9, a1)
\Edges(a9, a3)
\Edges(a9, a5)
\Edges(a9, a6)
\Edges(a9, a7)
\Edges(a9, a8)
\Edges(a10, a3)
\Edges(a10, a4)
\Edges(a10, a7)
\Edges(a10, a9)
\Edges(a11, a0)
\Edges(a11, a4)
\Edges(a11, a5)
\Edges(a11, a8)
\Edges(a11, a10)
\Edges(a12, a0)
\Edges(a12, a1)
\Edges(a12, a6)
\Edges(a12, a9)
\Edges(a12, a10)
\Edges(a12, a11)
\Edges(a13, a1)
\Edges(a13, a2)
\Edges(a13, a5)
\Edges(a13, a7)
\Edges(a13, a10)
\Edges(a13, a11)
\Edges(a13, a12)
\Edges(a14, a2)
\Edges(a14, a3)
\Edges(a14, a6)
\Edges(a14, a8)
\Edges(a14, a10)
\Edges(a14, a11)
\Edges(a14, a12)
\Edges(a14, a13)
\Edges(a15, a1)
\Edges(a15, a2)
\Edges(a15, a8)
\Edges(a15, a9)
\Edges(a15, a12)
\Edges(a15, a14)
\Edges(a16, a2)
\Edges(a16, a3)
\Edges(a16, a5)
\Edges(a16, a9)
\Edges(a16, a10)
\Edges(a16, a13)
\Edges(a16, a15)
\Edges(a17, a3)
\Edges(a17, a4)
\Edges(a17, a5)
\Edges(a17, a6)
\Edges(a17, a11)
\Edges(a17, a14)
\Edges(a17, a15)
\Edges(a17, a16)
\Edges(a18, a0)
\Edges(a18, a4)
\Edges(a18, a6)
\Edges(a18, a7)
\Edges(a18, a10)
\Edges(a18, a12)
\Edges(a18, a15)
\Edges(a18, a16)
\Edges(a18, a17)
\Edges(a19, a0)
\Edges(a19, a1)
\Edges(a19, a7)
\Edges(a19, a8)
\Edges(a19, a11)
\Edges(a19, a13)
\Edges(a19, a15)
\Edges(a19, a16)
\Edges(a19, a17)
\Edges(a19, a18)
\Edges(a20, a1)
\Edges(a20, a3)
\Edges(a20, a6)
\Edges(a20, a7)
\Edges(a20, a13)
\Edges(a20, a14)
\Edges(a20, a17)
\Edges(a20, a19)
\Edges(a21, a2)
\Edges(a21, a4)
\Edges(a21, a7)
\Edges(a21, a8)
\Edges(a21, a10)
\Edges(a21, a14)
\Edges(a21, a15)
\Edges(a21, a18)
\Edges(a21, a20)
\Edges(a22, a0)
\Edges(a22, a3)
\Edges(a22, a8)
\Edges(a22, a9)
\Edges(a22, a10)
\Edges(a22, a11)
\Edges(a22, a16)
\Edges(a22, a19)
\Edges(a22, a20)
\Edges(a22, a21)
\Edges(a23, a1)
\Edges(a23, a4)
\Edges(a23, a5)
\Edges(a23, a9)
\Edges(a23, a11)
\Edges(a23, a12)
\Edges(a23, a15)
\Edges(a23, a17)
\Edges(a23, a20)
\Edges(a23, a21)
\Edges(a23, a22)
\Edges(a24, a0)
\Edges(a24, a2)
\Edges(a24, a5)
\Edges(a24, a6)
\Edges(a24, a12)
\Edges(a24, a13)
\Edges(a24, a16)
\Edges(a24, a18)
\Edges(a24, a20)
\Edges(a24, a21)
\Edges(a24, a22)
\Edges(a24, a23)
\end{tikzpicture}
\caption{The Paley graph $P(25)$.}
\end{figure}

Here is a further fact concerning Paley graphs. 

\begin{prop}\label{prop: chromatic paley}
The Paley graph $P(q^2)$ has chromatic number $q$, and independence number $q$.
\end{prop}

\begin{proof}
Firstly, note that $q^2\equiv 1$ mod $4$, so we may indeed consider $P(q^2)$. Let $\K$ be a field with $q^2$ elements, and let $\F$ be a subfield with $q$ elements. Every element of $\F$ is a square in $\K$, so $\F$ defines a complete subgraph of $P(q^2)$. It follows that $\chi\geq q$. Moreover, we may partition the elements of $\K$ into $q$ cosets of $\F$, and each such coset $C$ forms a complete subgraph of $P(q^2)$. Hence $\iota\leq q$.

Now let $r$ be a non-square in $\K$. Then $r\F$ is an independent subset of $P(q^2)$, so $\iota\geq q$. Actually $rC$ is an independent subset for every coset $C$ as before. Painting each $rC$ monochromatically shows that $\chi\leq q$.
\end{proof}

For $q\equiv 3$ mod $4$, we can no longer use a Cayley graph construction. We define the \emph{bi-Paley graph} $BP(q)$ as the bi-Cayley graph of the additive group of $\F$ with respect to the non-zero squares $\F^{*2}$.

We discard $BP(3)$, as it is rather degenerate, namely, it is the disjoint union of three copies of $K_2$. So the first graph of the bi-Paley family is $BP(7)$, which turns out to be the Heawood graph.

\begin{thm}
The bi-Paley graph $BP(q)$ is a design graph, with parameters
\begin{align*}
m=q, \qquad d=\tfrac{1}{2}(q-1), \qquad c=\tfrac{1}{4}(q-3).
\end{align*}
\end{thm}

\begin{proof}
The half-size $m$ and the degree $d$ are obvious. We check the design property: two vertices of the same colour have $c=\tfrac{1}{4}(q-3)$ common neighbours. It suffices to argue the case of two distinct black vertices, $v_\bl$ and $w_\bl$. The number of white vertices adjacent to both $v_\bl$ and $w_\bl$ is the number of representations of $v-w$ as a difference of two non-zero squares in $\F$. There are $q-1$ pairs $(x,y)\in \F\times \F$ satisfying $x^2-y^2=v-w$.

Now observe that $\F^{*2}$ and $-\F^{*2}$ partition $\F^*$, as $-1$ is not a square. If $v-w$ is a non-zero square then necessarily $x\neq 0$, and there are two solutions with $y=0$. Thus $v-w$ is a difference of two-squares in $\tfrac{1}{4}(q-3)$ ways. If $w-v$ is a non-zero square then $y\neq 0$, and there are two solutions with $x=0$. Again, $v-w$ is a difference of two-squares in $\tfrac{1}{4}(q-3)$ ways.
\end{proof}

\begin{exer}\label{exer: nobip BP}  Show that the bi-Paley graph $BP(q)$ is not a bipartite double.
\end{exer}

\begin{notes} Paley graphs are named after Raymond Paley, who used the quadratic signature to construct Hadarmard matrices. These are square matrices whose entries are $\pm 1$, and whose rows are orthogonal. The graph-theoretic implementation of the same idea came much later.

Ramsey numbers are named after Frank Ramsey. It is actually the case, and not too difficult to prove, that $R(4)=18$. Quite easily, $R(3)=6$. Quite strikingly, $R(5)$ is not known! The best available knowledge at the time of writing is that $43\leq R(5)\leq 49$. 
\end{notes}

\bigskip
\subsection*{A comparison}\label{sec: compare}\
\medskip

We have defined two families of design graphs from finite fields. For a prime power $q\equiv 3$ mod $4$, the bi-Paley graph $BP(q)$ has parameters
\begin{align*}
(m,d,c)=\bigg(q, \frac{q-1}{2}, \frac{q-3}{4}\bigg).
\end{align*}
The incidence graph $I_n(q)$, where $n\geq 3$ and $q$ is an arbitrary prime power, has parameters
\begin{align*}
(m,d,c)=\bigg(\frac{q^n-1}{q-1},\frac{q^{n-1}-1}{q-1}, \frac{q^{n-2}-1}{q-1}\bigg).
\end{align*}
We wish to compare these two families, and we start by comparing their parameters.

\begin{prop} 
The parameters of a bi-Paley graph and those of an incidence graph agree only for $BP(q)$ and $I_n(2)$, where $q=2^n-1$ is a prime.
\end{prop}

\begin{proof}
A bi-Paley graph satisfies $m=2d+1$. Imposing this relation on an incidence graph $I_n(q')$, we find that $q'=2$. The parameters of $I_n(2)$ are $(2^n-1, 2^{n-1}-1, 2^{n-2}-1)$, and these are the parameters of a bi-Paley graph precisely when $2^n-1$ is a prime power. Note here that $2^n-1\equiv 3$ mod $4$, as $n\geq 3$.

We are left with arguing that $2^n-1$ is a prime power if and only if $2^n-1$ is actually a prime. So let $2^n-1=p^d$, where $p$ is a prime and $d$ is a positive integer that we claim to be $1$. As $2^n-1\equiv 3$ mod $4$, we first deduce that $d$ is odd. We can then write $2^n=p^d+1=(p+1)(p^{d-1}-\ldots+1)$. But the latter factor is odd, since both $p$ and $d$ are odd, so it has to equal $1$. Therefore $2^n-1=p$. \end{proof}

The prime sequence that appears in the previous proposition is in fact a notorious one. A prime of the form $2^n-1$ is known as a \emph{Mersenne prime}. Note that $n$ is, a fortiori, prime as well. Very few Mersenne primes are known: at the time of writing, the 49th has just been found. The sequence of primes $n$ that yield a Mersenne prime starts with $2$, $3$, $5$, $7$, $13$, $17$, $19$, $31$, $61$, $89$,$\dots$. It could be that there are infinitely many Mersenne primes.

Next, we consider possible isomorphisms between the two families. The graphs $BP(q)$ and $I_n(2)$ are isomorphic for the first allowed values, $q=7$ respectively $n=3$. This can be seen in several ways. One is simply by drawing: both turn out to be the Heawood graph. The following exercise suggests a more sophisticated approach. 

\begin{exer}\label{exer: BP(7) is I(2)}
Show that the bi-Paley graph $BP(7)$ is isomorphic to the incidence graph $I_3(2)$ by employing a notion of bi-Cayley graph isomorphism.
\end{exer}

It turns out that there is no other isomorphism between bi-Paley graphs and incidence graphs. To see the somewhat subtle difference, let us take a closer look at the design property in incidence graphs. Thinking of $I_n(q)$ in the incidence picture, consider two lines and the hyperplanes containing them. Then these hyperplanes contain the entire plane spanned by the two lines, so they contain $q-1$ additional lines. Dually, if we take two hyperplanes and the lines that they contain, then these lines are contained in $q-1$ further hyperplanes. One way to phrase this is that an incidence graph is a design graph in which we cannot recover a pair of monochromatic vertices from its joint neighbours. Now, the contrary holds in bi-Paley graphs--with the exception of $BP(7)$.

\begin{thm}\label{thm: no K3 in BP}
Let $q\equiv 3$ mod $4$, and $q>7$. Then the bi-Paley graph $BP(q)$ has the property that any two distinct vertices of the same colour are determined by their $\tfrac{1}{4}(q-3)$ common neighbours. 
\end{thm}

\begin{proof}
Arguing by contradiction, let us assume that there are three distinct vertices of the same colour sharing $\tfrac{1}{4}(q-3)$ common neighbours. In terms of the underlying field $\F$, this means that there are three distinct elements $a$, $b$, and $c$, and a subset $S\subseteq \F$ of size $\tfrac{1}{4}(q-3)$, such that $x-a$, $x-b$, are $x-c$ are non-zero squares whenever $x\in S$. Furthermore, if any two of $x-a$, $x-b$, are $x-c$ are non-zero squares, then $x\in S$. So we may partition $\F\setminus\{a,b,c\}$ into
\begin{align*}
S&=\{x: \sigma(x-a)=\sigma(x-b)=\sigma(x-c)=1\}\\
N&=\{x: \sigma(x-a)=\sigma(x-b)=\sigma(x-c)=-1\}\\
S_a&=\{x: \sigma(x-a)=1, \sigma(x-b)=\sigma(x-c)=-1\}
\end{align*}
and two more sets, $S_b$ and $S_c$, defined analogously. 

We claim that $N$ is, in fact, empty. On the one hand, the relation $\sum \sigma(x-a)=0$ says that
\begin{align*}
|S|+|S_a|-|S_b|-|S_c|-|N|+\sigma(b-a)+\sigma(c-a)=0.
\end{align*}
On the other hand, we have $\sum \sigma(x-b)\:\sigma(x-c)=\sigma(-1)\sum \sigma(b-x)\:\sigma(x-c)=-1$, thanks to Lemma~\ref{lem: convolution}. Hence
\begin{align*}
|S|+|S_a|-|S_b|-|S_c|+|N|+\sigma(a-b)\:\sigma(a-c)=-1.
\end{align*}
The two displayed relations imply that
\begin{align*}
2|N|+\big(1-\sigma(b-a)\big)\big(1-\sigma(c-a)\big)=0.
\end{align*}
Since both terms are non-negative, we must have $|N|=0$, as claimed.

Therefore $\F\setminus\{a,b,c\}$ is actually partitioned into $S$, $S_a$, $S_b$ and $S_c$. This means that $f(x)=(x-a)(x-b)(x-c)$ is a non-zero square whenever $x\neq a,b,c$. We reach a contradiction in light of the proposition below, which we treat as an independent fact interesting on its own.
\end{proof} 

\begin{prop}\label{prop: elliptic}
Let $f\in \F[X]$ be a cubic polynomial. If $q>9$, then some non-zero value of $f$ is a non-square. 
\end{prop}

\begin{proof}
Assume that every non-zero value of $f$ is a square. Then the polynomial $f(f^{(q-1)/2}-1)=f^{(q+1)/2}-f$ takes only the zero value. But the same holds for the polynomial $P=X^q-X$. A non-zero polynomial vanishing on $\F$ cannot have degree smaller than $q$ so, using division with remainder, it must be divisible by $P$. Write
\begin{align*}
f^{(q+1)/2}-f=Ph
\end{align*}
where $h\in \F[X]$ has degree $\tfrac{1}{2}(q+3)$. Taking the formal derivative, we get
\begin{align*}
\frac{q+1}{2}f'f^{(q-1)/2}-f'=Ph'+P'h=Ph'-h.
\end{align*}
Now the mod $P$ congruences
\begin{align*}
f^{(q+1)/2}\equiv f, \qquad \frac{q+1}{2}f'f^{(q-1)/2}\equiv f'-h
\end{align*}
imply that
\begin{align*}
\frac{q-1}{2}f'f+hf \equiv 0.
\end{align*}
So the degree of the above polynomial, $\tfrac{1}{2}(q+3)+3$, has to be at least $\mathrm{deg}(P)=q$. We find that $q\leq 9$, which is a contradiction.
\end{proof}

\begin{notes} Zannier (\emph{Polynomials modulo $p$ whose values are squares (elementary improvements on some consequences of Weil's bounds)}, Enseign. Math. 1998) gives an elementary argument for the following basic fact from the theory of elliptic curves over finite fields: the equation $y^2=f(x)$, where $f$ is a cubic polynomial, has solutions for $q>3$. After rescaling $f$ by a non-square, Proposition~\ref{prop: elliptic} says that $y^2=f(x)$ has a solution with $y\neq 0$ for $q>9$. The proof of Proposition~\ref{prop: elliptic} adapts Zannier's nice argument.
\end{notes}

\bigskip
\section{Characters}

We start with a general discussion of characters on finite abelian groups. Then we discuss character sums over finite fields.

\bigskip
\subsection*{Characters of finite abelian groups} \label{sec: characters}\
\medskip

Let $G$ be a finite abelian group, whose operation is written multiplicatively. A \emph{character} $\chi$ of $G$ is a group homomorphism $\chi: G\to \C^*$. 

\begin{ex} Let $\F$ be a finite field of odd characteristic. Then the quadratic signature $\sigma$ is a character of the multiplicative group $\F^*$. 
\end{ex}

\begin{ex}On the additive group $\Z_n$, a list of $n$ different characters is given by 
\begin{align*}
\chi_k(a)=e^{2\pi ika/n}, \qquad k=0,\dots,n-1. 
\end{align*}
\end{ex}

\begin{ex} On the additive group $(\Z_2)^n$, a list of $2^n$ different characters is given by 
\begin{align*}
\chi_I(v)= (-1)^{\sum_{i\in I} v_i}, \qquad I\subseteq \{1,\dots,n\}.
\end{align*}
\end{ex}

The characters of $G$ take values in the unit circle $\{z: |z|=1 \}$. In fact, the values are roots of unity of order $|G|$. Let $G^{\wedge}$ denote the set of characters of $G$. Then $G^{\wedge}$ is a group under pointwise multiplication, the neutral element being the trivial character $\ct$ defined by $g\mapsto 1$ for all $g\in G$. The inverse of a character $\chi$ is the conjugate character $\overline{\chi}$. The finite abelian group $G^{\wedge}$ is called the \emph{dual} of $G$. 

\begin{thm} 
The dual $G^{\wedge}$ is isomorphic to $G$. In particular, $G$ has exactly $|G|$ characters.
\end{thm}

\begin{proof}
This follows by combining three facts: $G$ is isomorphic to a direct product of finite cyclic groups; taking duals is compatible with direct products, in the sense that $(G_1\times G_2)^\wedge$ is isomorphic to $G_1^{\wedge}\times G_2^{\wedge}$; if $G$ is a finite cyclic group, then $G^{\wedge}$ is isomorphic to $G$. The detailed verification is left to the reader. 
\end{proof}

Character theory for finite abelian groups is grounded on two facts. One is the previous theorem, and the other is the following fundamental relation for non-trivial characters.

\begin{thm}\label{prop: character vanishing sum}
If $\chi$ is a non-trivial character of $G$, then
\begin{align*}
\sum_{g\in G} \chi(g)=0.
\end{align*}
\end{thm}

\begin{proof}
For each $h\in G$ we have
 \begin{align*}
 \chi(h)\sum_{g\in G} \chi(g)=\sum_{g\in G} \chi(hg)=\sum_{g\in G} \chi(g).
 \end{align*}
 As $\chi$ is non-trivial, there exists $h\in G$ with $\chi(h)\neq 1$. The desired relation follows.
\end{proof}

We deduce, in sequence, several consequences. The first concerns the way in which characters sit in $\ell^2G$. This is the finite-dimensional space of complex-valued functions on $G$, endowed with the usual inner product 
\begin{align*}
\la \phi,\psi\ra=\sum_{g\in G} \phi(g)\: \overline{\psi(g)}.
\end{align*}

\begin{cor}
The characters of $G$ form an orthogonal basis of $\ell^2G$.
\end{cor}

\begin{proof}
If $\chi_1$ and $\chi_2$ are distinct characters of $G$, then
\begin{align*}
\la \chi_1,\chi_2\ra=\sum_{g\in G} \chi_1(g)\:\overline{\chi_2}(g)=\sum_{g\in G} (\chi_1\overline{\chi_2})(g)=0
\end{align*}
as $\chi_1\overline{\chi_2}$ is a non-trivial character. Thus, the characters are orthogonal. As there are $|G|=\dim \ell^2G$ characters, we conclude that they form a basis.
\end{proof}

\begin{cor}
The characters of $G$ distinguish the elements of $G$. Namely, if $g\neq 1$ is a non-trivial element in $G$, then there exists a character $\chi$ such that $\chi(g)\neq 1$.
\end{cor}

\begin{proof}
Let $\ct_g:G\to \C$ denote the characteristic function of $g$. By the previous result, $\ct_g$ is a linear combination of characters, say $\ct_g=\sum a_\chi\: \chi$. Evaluating both sides at $g$, respectively at $1$, we obtain
\begin{align*}
1=\sum a_\chi\: \chi(g), \qquad 0= \sum a_\chi.
\end{align*}
So there must be some $\chi$ satisfying $\chi(g)\neq 1$.
\end{proof}

\begin{cor}
If $g\neq 1$ is a non-trivial element in $G$, then 
\begin{align*}
\sum_{\chi\in G^{\wedge}} \chi(g)=0.
\end{align*}
\end{cor}

\begin{proof}
The identity to be proved is the dual of the fundamental character sum, and we will use the dual trick. Thanks to the previous result, there is a character $\chi'$ of $G$ satisfying $\chi'(g)\neq 1$. Then
\begin{align*}
\chi'(g) \sum_{\chi\in G^{\wedge}} \chi(g)=\sum_{\chi\in G^{\wedge}} (\chi'\:\chi)(g)=\sum_{\chi\in G^{\wedge}} \chi(g)
\end{align*}
and the claimed relation follows.
\end{proof}

\begin{exer}\label{exer: d-powers} Let $d$ be a positive integer, and $g\in G$. Then:
 \begin{align*}
 \sum_{\chi^d=\ct}\: \chi(g)=\big|\{h\in G: h^d=g\}\big|
 \end{align*}
 \end{exer}

\bigskip
\subsection*{Character sums over finite fields} \label{sec: character sums}\
\medskip

The characters of a finite field $\F$ come in two flavours: characters of the additive group $(\F,+)$, and characters of the multiplicative group $(\F^*,\cdot)$. In what follows, an additive character is usually denoted by $\psi$, while a multiplicative one by $\chi$. It is convenient to extend the multiplicative characters to the whole of $\F$, and we do so by setting $\chi(0)=0$ for $\chi\neq \ct$, respectively $\chi(0)=1$ for $\chi=\ct$.

If $\K$ is an extension of $\F$, then characters of $\K$ restrict to characters of $\F$. The other way around, characters of $\F$ induce characters of $\K$ via the trace and the norm maps: if $\psi$ is an additive character of $\F$ then $\psi\circ \Tr$ is an additive character of $\K$, and if $\chi$ is a multiplicative character of $\F$ then $\chi\circ \Nm$ is a multiplicative character of $\K$. Note that non-trivial characters of $\F$ induce non-trivial characters of $\K$, thanks to the surjectivity of the trace and the norm.

Character sums arise by summing values of additive characters or multiplicative characters, or products thereof, over the whole field or just a subset of it. As this vague description already suggests, character sums come in a kaleidoscopic variety. 

A \emph{Gauss sum} over $\F$ is a character sum of the form\begin{align*}
G(\psi,\chi)=\sum_{s\in \F^*} \psi(s)\:\chi(s) 
\end{align*}
where $\psi$ is an additive character, and $\chi$ is a multiplicative character of $\F$. 

Recall that Gauss's Lemma concerns a sum of the above form. Specifically, the cyclotomic integer $G$ is the Gauss sum over the field $\Z_p$, corresponding to the additive character $a\mapsto \zeta^a$, and the quadratic signature as the multiplicative character. A general Gauss sum over a field $\F$ is, however, no longer a cyclotomic integer, for it combines roots of unity of order $q$ with roots of unity of order $q-1$. Here $q$ is, as usual, the number of elements of $\F$.

The Gauss sums involving trivial characters are easily computed:
\begin{align*}
G(\ct,\ct)=q-1,\qquad  G(\ct,\chi)=0 \textrm{ for } \chi\neq \ct,\qquad G(\psi,\ct)=-1 \textrm{ for } \psi\neq \ct
\end{align*}

Far more interesting is the following fact.

\begin{thm}\label{thm: abs gauss} If $\psi$ and $\chi$ are non-trivial characters, then
\begin{align*}
\big|G(\psi, \chi)\big|=\sqrt{q}. 
\end{align*}
\end{thm}

\begin{proof}
We start by expanding $\big|G(\psi, \chi)\big|^2=G(\psi,\chi)\;\overline{G(\psi,\chi)}$:
\begin{align*} 
\big|G(\psi, \chi)\big|^2=\sum_{s,t\in \F^*} \psi(s)\:\overline{\psi(t)}\:\chi(s)\:\overline{\chi(t)}=\sum_{t\in \F^*} \sum_{s\in \F^*} \psi(s-t)\:\chi(s/t) 
\end{align*}
We make the change of variable $s\mapsto ts$ in the inner sum, and we continue:
\begin{align*} 
\big|G(\psi, \chi)\big|^2=\sum_{t\in \F^*} \sum_{s\in \F^*} \psi(t(s-1))\:\chi(s)=\sum_{s\in \F^*} \chi(s)\sum_{t\in \F^*} \psi(t(s-1))
\end{align*}
For $s=1$ the inner sum equals $q-1$, while for $s\neq 1$ it equals $-\psi(0)=-1$. Therefore
\begin{align*} 
\big|G(\psi, \chi)\big|^2&=q-1-\sum_{1\neq s\in \F^*} \chi(s)=q-\sum_{s\in \F^*} \chi(s)=q
\end{align*}
as claimed.
\end{proof}

The previous theorem is, in effect, a generalization of Gauss's Lemma. To see this, observe that
\begin{align*}
\overline{G(\psi,\chi)}&=\sum_{s\in \F^*} \overline{\psi(s)}\;\overline{\chi(s)}=\sum_{s\in \F^*} \psi(-s)\:\overline{\chi}(s)\\
&=\sum_{s\in \F^*} \psi(s)\:\overline{\chi}(-s)=\chi(-1) \sum_{s\in \F^*} \psi(s)\:\overline{\chi}(s)=\chi(-1) \: G(\psi,\overline{\chi}).
\end{align*}
So we may restate the theorem as saying that
\begin{align*}
G(\psi,\chi)\:G(\psi,\overline{\chi})=\chi(-1)\: q
\end{align*}
whenever $\psi$ and $\chi$ are non-trivial. In particular, if $q$ is odd then we may take $\chi$ to be the quadratic signature $\sigma$. Note that $\sigma$ is real-valued, in fact it is the only non-trivial multiplicative character which is real-valued. We obtain $G(\psi,\sigma)^2=\sigma(-1)\: q$ for any non-trivial $\psi$, and this relation is a mild extension of Gauss's Lemma.

A \emph{Jacobi sum} over $\F$ is a character sum of the form
\begin{align*}
J(\chi_1,\chi_2)=\sum_{s+t=1} \chi_1(s)\:\chi_2(t) 
\end{align*}
where $\chi_1$ and $\chi_2$ are multiplicative characters. 

We have already encountered a Jacobi sum in Lemma~\ref{lem: convolution}, where we found that $J(\sigma,\sigma)=-\sigma(-1)$. Note also that $J(\ct,\ct)=q$, and $J(\ct,\chi)=J(\chi,\ct)=0$ whenever $\chi$ is non-trivial. In what concerns the Jacobi sums involving non-trivial characters, their absolute values can be computed as well. This is the aim of the exercise below.

\begin{exer}\label{exer: jacobi sums} (i) Let  $\chi_1$ and $\chi_2$ be non-trivial multiplicative characters. Show that 
\begin{align*}
\big|J(\chi_1,\chi_2)\big|=
\begin{cases} 
\sqrt{q} & \textrm{ if } \chi_1\chi_2\neq \ct\\
1 & \textrm{ if } \chi_1\chi_2=\ct
\end{cases}
\end{align*} 
by proceeding as follows. In the first case, establish the following relation with Gauss sums: $G(\psi,\chi_1)\:G(\psi,\chi_2)=J(\chi_1,\chi_2)\: G(\psi,\chi_1\chi_2)$ for each non-trivial additive character $\psi$. For the second case, argue directly that $J(\chi,\chi^{-1})=-\chi(-1)$ for any non-trivial $\chi$.

\smallskip
(ii) Use the above computation, for a suitable choice of $\chi_1$ and $\chi_2$, to give another proof of Fermat's theorem: if $q\equiv 1$ mod $4$, then $q$ is a sum of two integral squares. 
\end{exer}

A Gauss sum is a mixed-type character sum, as it combines additive and multiplicative characters. A Jacobi sum is a convolution-type character sum. The following definition introduces yet another type: a restricted character sum.

An \emph{Eisenstein sum} for an extension $\K/\F$ is a character sum of the form
\begin{align*}
E(\chi)=\sum_{\Tr(s)=1} \chi(s), \qquad 
\end{align*}
where $\chi$ is a multiplicative character of $\K$. One could also consider the `singular' Eisenstein sum
\begin{align*}
E_0(\chi)=\sum_{\substack{\Tr(s)=0 \\ s\neq 0}} \chi(s)
\end{align*}
but this sum plays only a secondary role, as the next lemma will show. Note at this point that, for the trivial character, we have $E(\ct)=q^{n-1}$ and $E_0(\ct)=q^{n-1}-1$. As usual, $n$ denotes the degree of the extension $\K/\F$.

\begin{lem}\label{lem: sing eins} Let $\chi$ be a non-trivial character of $\K$. Then
\begin{align*}
E_0(\chi)=
\begin{cases}
0 & \textrm{ if } \chi \textrm{ is non-trivial on } \F,\\
-(q-1)\: E(\chi) & \textrm{ if } \chi \textrm{ is trivial on } \F.
\end{cases}
\end{align*}
\end{lem}

\begin{proof}
Decomposing over trace values, we write
\begin{align*}
\sum_{s\in \K^*} \chi(s)=E_0(\chi)+\sum_{c\in \F^*}\sum_{\Tr(s)=c} \chi(s).
\end{align*}
The left-hand sum is $0$. On the right-hand side, for each $c\neq 0$ we have that
\begin{align*}
\sum_{\Tr(s)=c} \chi(s)=\chi(c)\: E(\chi)
\end{align*}
by a change of variable $s\mapsto cs$. Thus
\begin{align*}
0=E_0(\chi)+E(\chi) \sum_{c\in \F^*} \chi(c).
\end{align*}
The latter sum is $0$ or $q-1$ according to whether $\chi$ is non-trivial or trivial on $\F$. The claimed formulas for $E_0(\chi)$ follow.
\end{proof}

The statement and the proof of the previous lemma are helpful in establishing the following fact.

\begin{thm}\label{thm: abs eins} 
Let $\chi$ be a non-trivial character of $\K$. Then
\begin{align*}
\big|E(\chi)\big|=
\begin{cases}
q^{(n-1)/2} & \textrm{ if } \chi \textrm{ is non-trivial on } \F,\\
q^{n/2-1} & \textrm{ if } \chi \textrm{ is trivial on } \F.
\end{cases}
\end{align*}
\end{thm}

\begin{proof}
Let $\chi_\mathrm{res}$ denote the multiplicative character of $\F$ obtained by restricting $\chi$. Let also $\psi$ be a non-trivial additive character of $\F$, and denote by $\psi^\mathrm{ind}$ the additive character of $\K$ induced via the trace. Note that $\psi^{\mathrm{ind}}$ is also non-trivial.

We relate the following two Gauss sums, one over $\K$ and the other over $\F$:
\begin{align*}
G(\psi^\mathrm{ind}, \chi)=\sum_{s\in \K^*} \psi(\Tr(s))\:\chi(s),\qquad
G(\psi, \chi_\mathrm{res})=\sum_{c\in \F^*} \psi(c)\:\chi(c)
\end{align*}
We decompose the Gauss sum over $\K$ according to the values of the trace map, and we write
\begin{align*}
G(\psi^{\mathrm{ind}}, \chi)&=\psi(0)\: E_0(\chi)+\sum_{c\in \F^*} \psi(c)\: \Big(\sum_{\Tr(s)=c} \chi(s)\Big)\\
&=E_0(\chi)+\sum_{c\in \F^*} \psi(c) \: \chi(c)\: E(\chi)\\
&=E_0(\chi)+G(\psi, \chi_{\mathrm{res}})\: E(\chi).
\end{align*}

Assume $\chi_\mathrm{res}$ is non-trivial. Then $E_0(\chi)=0$, and the above relation simplifies to 
\begin{align*}
G(\psi^{\mathrm{ind}}, \chi)=G(\psi, \chi_{\mathrm{res}})\: E(\chi).
\end{align*} 
As $|G(\psi^{\mathrm{ind}}, \chi)|=q^{n/2}$, and $|G(\psi, \chi_{\mathrm{res}})|=q^{1/2}$, we get $|E(\chi)|=q^{(n-1)/2}$.

\smallskip
Assume $\chi_\mathrm{res}$ is trivial. Then $E_0(\chi)=-(q-1)\: E(\chi)$ and $G(\psi, \chi_{\mathrm{res}})=-1$. In this case, we get
\begin{align*}
G(\psi^{\mathrm{ind}}, \chi)=-q\: E(\chi).
\end{align*} 
Using $|G(\psi^{\mathrm{ind}}, \chi)|=q^{n/2}$ once again, we conclude that $|E(\chi)|=q^{n/2-1}$.
\end{proof}

\bigskip
\subsection*{More character sums over finite fields} \label{sec: character sums+}\
\medskip

For the character sums considered so far--Gauss, Jacobi, Eisenstein sums--we have succeeded in computing their absolute values. To gain some appreciation for our computations and their relative simplicity, we need only enlarge the perspective very slightly. Let us consider some similar looking sums.

A \emph{Kloosterman sum} over $\F$ is a character sum given by the convolution formula
\begin{align*}
K(\psi_1,\psi_2)=\sum_{st=1} \psi_1(s)\:\psi_2(t) 
\end{align*}
where $\psi_1$ and $\psi_2$ are additive characters. This is the counterpart of a Jacobi sum. The Kloosterman sums involving non-trivial characters can be bounded as follows.

\begin{thm}[Weil]
If $\psi_1$ and $\psi_2$ are non-trivial, then $\big|K(\psi_1,\psi_2)\big|\leq 2\sqrt{q}$.
\end{thm}

Now let $\K/\F$ be an extension of degree $n$. The restricted character sum
\begin{align*}
S(\psi)=\sum_{\Nm(s)=1}\psi(s)
\end{align*}
where $\psi$ is an additive character of $\K$, is the counterpart of an Eisenstein sum. Fairly simple considerations lead to the bound $|S(\psi)|\leq q^{n/2}$ for non-trivial $\psi$, see Exercise~\ref{exer: norm sums} below. However, when $n$ is small with respect to $q$, a better bound holds.

\begin{thm}[Deligne]
If $\psi$ is non-trivial, then $\big|S(\psi)\big|\leq n\: q^{(n-1)/2}$.
\end{thm}

\begin{exer}\label{exer: norm sums} (i) Let $\psi$ be a non-trivial additive character of $\F$, and $d$ a positive integer. Show that
\begin{align*}
\Big|\sum_{s\in \F} \psi(s^d)\Big|\leq (d-1)\sqrt{q}
\end{align*}
via the identity $\sum_{s\in \F^*} \psi(s^d)=\sum_{\chi^d=\ct} G(\psi,\chi)$.

\smallskip
(ii) Let $\K/\F$ be an extension of degree $n$, and let $\psi$ be a non-trivial additive character on $\K$. Using the above estimate, conclude that $|S(\psi)|\leq q^{n/2}$.
\end{exer}

The previous exercise is an instance of character sums restricted to polynomial values. The following result addresses this type of character sums, in much greater generality.

\begin{thm}[Weil]\label{thm: weil}
Let $f\in \F[X]$ be a monic polynomial of degree $d$.

\begin{itemize}
\item[(i)] Let $\chi$ be a non-trivial multiplicative character of order $m$. If $f$ is not of the form $g^m$ for some $g\in \F[X]$, then
\begin{align*}
\Big|\sum_{s\in \F} \chi(f(s))\Big|\leq (d-1)\sqrt{q}.
\end{align*}
\item[(ii)] Let $\psi$ be a non-trivial additive character. If $d$ is relatively prime to $q$, then
\begin{align*}
\Big|\sum_{s\in \F} \psi(f(s))\Big|\leq (d-1)\sqrt{q}.
\end{align*}
\end{itemize}
\end{thm}

These bounds are better than the trivial bound, $q$, when the degree $d$ is small compared to $\sqrt{q}$. The hypothesis in part (i) is clearly needed. To motivate the hypothesis in part (ii), consider the following example. Assume $q=p^e$, where $e\geq 2$. Let $\psi$ be an additive character of $\F$ induced by a non-trivial character of its prime field. Let also $f=X^p-X$. Then $\psi(f(s))=1$ for each $s\in \F$. Now $\sum_{s\in \F} \psi(f(s))=q$, for which $(p-1)\sqrt{q}$ is not an upper bound.

The following result is an important consequence of Theorem~\ref{thm: weil}. Historically, it predates, and it motivated, the Weil estimates. 

\begin{cor}[Hasse] Let $N$ be the number of solutions to the equation $y^2=f(x)$, where $f$ is a cubic polynomial. Then $|N-q|\leq 2\sqrt{q}$.
\end{cor}

\begin{proof}
For each $x$, there are $1+\sigma(f(x))$ solutions to $y^2=f(x)$. Thus 
\begin{align*}
N=\sum_x \big(1+\sigma(f(x))\big)=q+\sum_x \sigma(f(x)).
\end{align*}
Now
\begin{align*}
\Big|\sum_{x} \sigma(f(x))\Big|\leq 2\sqrt{q}.
\end{align*}
by applying part (i) of Theorem~\ref{thm: weil}.
\end{proof}

\begin{notes} Character sums over finite fields are important in algebraic and analytic number theory. We have used them in the proof of quadratic reciprocity, and for counting the number of solutions for certain polynomial equations. As we will see, they also come up in the spectral study of various algebraically-defined graphs.

The references for the theorems of Weil and Deligne that we have quoted are Weil (\emph{On some exponential sums}, Proc. Nat. Acad. Sci. U. S. A. 1948) respectively Deligne (\emph{Cohomologie \'etale, SGA $4\tfrac{1}{2}$}, Lecture Notes in Mathematics 569, Springer 1977). These are deep and difficult results, relying on ideas and techniques from algebraic geometry. Elementary approaches to Weil's theorem \ref{thm: weil} are nowadays available, but they are still very involved. See Lidl and Niederreiter (\emph{Finite fields}, Second edition, Cambridge University Press 1997).
\end{notes}

\bigskip
\subsection*{An application to Paley graphs}\
\medskip

Despite their deterministic construction, Paley graphs behave like random graphs. Informally, a random graph is obtained by starting with a set of vertices, and for each pair of vertices one draws an edge or not as decided by a coin flip. The following result, a rather striking universal property, illustrates the random-like nature of the Paley graphs. For a feature of randomness is that every possible substructure eventually appears.

\begin{thm}\label{thm: paley is universal}
Let $n$ be a positive integer. Then, for $q$ large enough with respect to $n$, the Paley graph $P(q)$ contains every graph on $n$ vertices as an induced subgraph. 
\end{thm}

In this theorem, the graphs on $n$ vertices need not be connected. An induced subgraph is a subgraph inheriting all edges of the ambient graph.

\begin{proof} A simple-minded strategy for realizing a graph on $n$ vertices inside $P(q)$ is to do it vertex by vertex. For this to succeed, we need to show that $P(q)$ enjoys the following property for $q$ large enough: given a set of vertices $A$ with $|A|=k\in \{1,\dots, n-1\}$, and a partition of $A$ into two subsets $C$ (`connect') and $D$ (`disconnect'), there is a `new' vertex $v\notin A$ such that $v$ is connected to each vertex in $C$ but to no vertex in $D$. 

The desired vertex $v\notin A$ has to satisfy $\sigma(v-c)=1$ for all $c\in C$, and $\sigma(v-d)\neq 1$ for all $d\in D$. In other words, we need $\tau(v)\neq 0$, where
\begin{align*}
\tau(v)=\prod_{c\in C}\big(1+\sigma(v-c)\big) \prod_{d\in D}\big(1-\sigma(v-d)\big).
\end{align*}
We will show that such a vertex exists by showing that $\sum_{v\notin A} \tau(v)>0$ for $q$ large enough. To that end, we need an upper bound for $\sum_{v\in A} \tau(v)$, and a lower bound for $\sum_{v\in \F} \tau(v)$. 

When $v\in A$, $\tau(v)$ is either $0$ or $2^{k-1}$. As $k=|A|$ is small in comparison to $q=|\F|$, we can afford the following bound: 
\begin{align}\label{eq: upper}
\sum_{v\in A} \tau(v)\leq k\: 2^{k-1}
\end{align}
On the other hand, expanding the product that defines $\tau(v)$, we get
\begin{align*}
\tau(v)= 1+\sum_{\emptyset\neq S\subseteq A} (-1)^{|S\cap D|} \sigma(f_S(v)), \qquad f_S(v):=\prod_{s\in S} (v-s)
\end{align*}
so
\begin{align*}
\sum _{v\in \F} \tau(v)= q-\sum_{\emptyset\neq S\subseteq A} (-1)^{|S\cap D|}\sum_{v\in \F} \sigma(f_S(v)).
\end{align*}
At this point, we employ the Weil character bound: 
\begin{align*}
\Big|\sum_{v\in \F} \sigma(f_S(v))\Big|\leq \big(|S|-1\big)\sqrt{q}
\end{align*} 
We obtain:
\begin{align*}
\sum _{v\in \F} \tau(v)\geq q- \sum_{\emptyset\neq S\subseteq A} \big(|S|-1\big)\sqrt{q}=q-\sqrt{q}\: \sum_{i=1}^k (i-1) \begin{pmatrix} k \\ i \end{pmatrix}
\end{align*}
and
\begin{align*}
\sum_{i=1}^k (i-1) \begin{pmatrix} k \\ i \end{pmatrix} = k \sum_{i=1}^k \begin{pmatrix} k-1 \\ i-1\end{pmatrix} -\sum_{i=1}^k \begin{pmatrix} k \\ i \end{pmatrix} =k2^{k-1}-(2^k-1)\leq k2^{k-1}-1.
\end{align*}
Thus
\begin{align}\label{eq: lower}
\sum _{v\in \F}\tau(v)\geq q- (k\: 2^{k-1}-1) \sqrt{q}.
\end{align}
Now \eqref{eq: upper} and \eqref{eq: lower} imply that $\sum_{v\notin A} \tau(v)>0$ whenever $q- (k\: 2^{k-1}-1) \sqrt{q}>k\: 2^{k-1}$, that is $\sqrt{q}>k\: 2^{k-1}$. This holds for each $k\in \{1,\dots, n-1\}$ when $\sqrt{q}> (n-1)\: 2^{n-2}$.
\end{proof}

\begin{notes}
Theorem~\ref{thm: paley is universal} is due to Bollob\'as and Thomason (\emph{Graphs which contain all small graphs}, European J. Combin. 1981), and to Blass, Exoo, and Harary (\emph{Paley graphs satisfy all first-order adjacency axioms}, J. Graph Theory 1981). It is a slight variation of the following earlier result of Graham and Spencer (\emph{A constructive solution to a tournament problem}, Canad. Math. Bull. 1971). Let $\F$ be a finite field with $q$ elements, where $q\equiv 3$ mod $4$. Thus, $-1$ is not a square in $\F$. Now consider the complete graph on the elements of $\F$, and orient each edge by the rule that $x\to y$ if and only if $x-y$ is a square. Such a directed complete graph can be thought of as a tournament in which every player meets every other player, and no game ends in a draw. The Graham - Spencer result is that the `Paley tournament' we have just described enjoys the following paradoxical property: for $q$ large enough with respect to a given $n$, every set of $n$ players is beaten by some other player. This is an organizers' nightmare, for who gets to win such a tournament?
\end{notes}

\bigskip
\section{Eigenvalues of graphs}
This section opens the second part of this text, devoted to the spectral perspective on graphs. The previous sections serve as a long motivation.

\bigskip
\subsection*{Adjacency and laplacian eigenvalues} \label{sec: some matrices}\ 
\medskip

The two most important matrices associated to a graph are the adjacency matrix and the laplacian matrix. Both are square matrices indexed by the vertex set $V$. The \emph{adjacency matrix} $A$ is given by
\begin{align*}
A(u,v)=
\begin{cases}
1 & \textrm{ if } u\sim v\\
0 & \textrm{ otherwise}
\end{cases}
\end{align*}
while the \emph{laplacian matrix} $L$ is defined as follows: 
\begin{align*}
L(u,v)=
\begin{cases}
\mathrm{deg}(v) & \textrm{ if } u=v\\
-1 & \textrm{ if } u\sim v\\
0 & \textrm{ otherwise}
\end{cases}
\end{align*}
The two matrices are related by the formula 
\begin{align*}
A+L=\mathrm{diag}(\mathrm{deg}),
\end{align*} where $\mathrm{diag}(\mathrm{deg})$ denotes the diagonal matrix recording the degrees.

We often view the adjacency matrix and the laplacian matrix as operators on $\ell^2V$. Recall that $\ell^2V$ is the finite-dimensional space of complex-valued functions on $V$, endowed with the inner product 
\begin{align*}
\la f,g\ra=\sum_{v} f(v)\: \overline{g(v)}.
\end{align*}
The adjacency operator $A:\ell^2V\to \ell^2V$ and the laplacian operator $L:\ell^2V\to\ell^2V$ are given by
\begin{align*}
Af(v)=\sum_{u: \: u\sim v} f(u),
\end{align*}
respectively
\begin{align*}
L f(v)=\mathrm{deg}(v)\: f(v)-\sum_{u: \: u\sim v} f(u).
\end{align*}
We have the following useful formulas:
\begin{align}
\la A f,f\ra &=\sum_{u\sim v} f(u)\:\overline{f(v)}= \sum_{\{u,v\}\in E} 2\: \mathrm{Re} \big(f(u)\:\overline{f(v)}\big) \label{eq: adjacency}\\
\la L f,f\ra &=\frac{1}{2}\sum_{u\sim v} \big|f(u)-f(v)\big|^2=\sum_{\{u,v\}\in E} \big|f(u)-f(v)\big|^2 \label{eq: laplacian}
\end{align}
Note the distinction between summing over adjacencies and summing over edges. Formula \eqref{eq: adjacency} is immediate from the definition of $A$, and \eqref{eq: laplacian} can be deduced from \eqref{eq: adjacency}. Formula \eqref{eq: laplacian} is particularly appealing: the right-most sum can be interpreted as an overall edge differential.

The adjacency matrix $A$ and the laplacian matrix $L$ are real symmetric matrices. Recall, at this point, the Spectral Theorem: if $M$ is a real, symmetric $n\times n$ matrix, then there is an orthogonal basis consisting of eigenvectors, and $M$ has $n$ real eigenvalues, counted with multiplicities. Thus, both $A$ and $L$ have $n=|V|$ real eigenvalues. In fact, the eigenvalues are algebraic integers (roots of monic polynomials with integral coefficients) since both $A$ and $L$ have integral entries.

Our convention is that the adjacency, respectively the laplacian eigenvalues are denoted and ordered as follows:
\begin{align*}
\alpha_{\min}=\alpha_n\leq \ldots \leq \alpha_2\leq \alpha_1=\alpha_{\max}\\
 \lambda_{\min}=\lambda_1\leq \lambda_2 \leq\ldots \leq \lambda_n=\lambda_{\max}
 \end{align*}
 For the purposes of relating eigenvalues to graph invariants, the most important eigenvalues will turn out to be the extremal ones, as well as $\lambda_2$ and $\alpha_2$.
 
The \emph{spectrum} is the multiset of eigenvalues, that is the set of eigenvalues repeated according to their multiplicity. Occasionally, we write $\adspec$ for the adjacency spectrum, respectively $\lapspec$ for the laplacian spectrum. 

For regular graphs, the laplacian matrix and the adjacency matrix carry the same spectral information. Indeed, $A+L=dI$, so $\alpha_k+\lambda_k=d$ and the corresponding eigenvectors are the same. Thus, using the laplacian or the adjacency spectrum is mostly a matter of convenience in the regular case. Even for non-regular graphs, there is a kind of silent duality between adjacency and laplacian eigenvalues.

A guiding principle of spectral graph theory is that much knowledge about a graph can be extracted from spectral information. But could it be that spectral information gives complete knowledge about a graph? The answer is a resounding no. As we will see, there are many examples of \emph{isospectral} but non-isomorphic regular graphs. As one might guess, two regular graphs are said to be isospectral when they have the same adjacency (equivalently, laplacian) spectrum. Note that we only consider isospectrality in the context of regular graphs. Of course, the same issue can be pursued for non-regular graphs, but then there are two distinct sides to the story.

\begin{notes} 
Perhaps the most conceptual way of motivating the combinatorial laplacian is the following. Choose an orientation on the edges, and define the operator $D: \ell^2V\to \ell^2E$ as $Df(e)=f(e^+)-f(e^-)$, where $e^+$ and $e^-$ denote the terminal, respectively the initial vertex of the edge $e$. We think of $D$ as a discrete differential operator. Then $L=D^*D$, where $D^*$ is the adjoint of  $D$. An analogous formula defines the geometric laplacian $\Delta$ on manifolds, with the gradient operator in place of $D$. This analogy explains the name attached to $L$, as well as the other common notation for the graph-theoretic laplacian, $\Delta$.

The laplacian of a graph can be traced back to work of Kirchhoff from 1847 on electrical networks. Not only the time, but also the result is surprising: in a modern formulation, Kirchhoff's Matrix-Tree theorem gives a formula for the number of spanning trees in a graph in terms of its laplacian eigenvalues. Interest in the laplacian eigenvalues was revived in the early 1970's by works of Fiedler (\emph{Algebraic connectivity of graphs}, Czechoslovak Math. J. 1973) and Anderson - Morley (\emph{Eigenvalues of the Laplacian of a graph}, Preprint 1971, Linear and Multilinear Algebra 1985).
\end{notes}

\bigskip
\subsection*{First properties}\
\medskip

There are similarities between the adjacency approach and the laplacian approach to spectra of graphs. But sometimes one viewpoint sees what the other does not. One crucial point in favour of the laplacian is the following observation.

\begin{thm}\label{thm: plus lap}
$0$ is a laplacian eigenvalue, having the constant function $\ct$ as an eigenfunction. Furthermore, $0$ is a simple eigenvalue.
\end{thm}

\begin{proof}
The first part is easily checked: $L\ct(v)=\mathrm{deg}(v)-\sum_{u: u\sim v}1=0$ for each vertex $v$. Now let us prove that $0$ is a simple laplacian eigenvalue. If $f$ is an associated eigenfunction, then 
\begin{align*}
0=\la L f,f\ra =\sum_{\{u, v\}\in E} |f(u)-f(v)|^2.
\end{align*} 
Therefore $f(u)=f(v)$ whenever $u$ is adjacent to $v$, and so $f$ is constant.
 \end{proof}
 
This theorem prompts us to refer to $0$ as the trivial laplacian eigenvalue. In the proof, we have implicitly used the assumption that graphs are connected. In general, the proof shows that the multiplicity of $0$ as a laplacian eigenvalue is the number of connected components of a graph. So the connectivity of a graph can be read off from the laplacian spectrum. Let us point out that, in general, one cannot see the connectivity of a graph from its adjacency spectrum.

 \begin{cor}\label{cor: regtrivial}
 For a regular graph, the degree $d$ is a simple eigenvalue, with the constant  function $\ct$ as an eigenfunction.
 \end{cor}

Note that, conversely, if $\ct$ is an adjacency eigenfunction then the graph is regular. For a regular graph, the degree is the trivial adjacency eigenvalue.

The adjacency spectrum of a graph has a combinatorial content that is manifested by the formula
\begin{align}
\sum \alpha_k^r = \Tr\: A^r, \qquad r=1,2,\dots. \label{eq: adjcomb}
\end{align} 
The point of this formula is that the right-hand side counts the closed paths of length $r$ in the graph. Indeed, the $(u,v)$-entry of $A^r$ equals the number of paths of length $r$ joining the two vertices $u$ and $v$. For $r=1$ and $r=2$ we have: 
\begin{align*}
\sum \alpha_k &= \Tr\: A=0\\
\sum \alpha_k^2&= \Tr\: A^2=\sum \mathrm{deg}(v)=2|E|
\end{align*}

Among other things, formula \eqref{eq: adjcomb} is relevant for the following important fact.

\begin{thm}
A graph is bipartite if and only if its adjacency spectrum is symmetric about $0$.
\end{thm}

\begin{proof}
Consider a bipartite graph. Let $\alpha$ be an adjacency eigenvalue, and let $f$ be a corresponding eigenfunction. Define a new function by switching the sign of $f$ on one side of the graph: $g=f$ on $V_\bl$ and $g=-f$ on $V_\wh$. Then $Ag=-\alpha g$, so $-\alpha$ is an adjacency eigenvalue as well.

Conversely, assume that the adjacency spectrum of a graph is symmetric about $0$. Then, for each odd $r\geq 1$, we have $\mathrm{Tr}\: A^r=\sum\alpha_k^r=0$. This means that there are no closed paths of odd length, so the given graph is bipartite.
\end{proof}

Combining the previous theorem with Corollary~\ref{cor: regtrivial}, we deduce the following: a $d$-regular bipartite graph has $d$ and $-d$ as simple adjacency eigenvalues. In this case, $\pm d$ are deemed to be the trivial adjacency eigenvalues.

\begin{exer}\label{exer: hearreg-lap} Show that the regularity of a graph can be read off from its laplacian spectrum.
\end{exer}

One of our main concerns is understanding the spectral distribution of a graph. The size of a graph counts the number of eigenvalues with multiplicity. The maximal degree, on the other hand, is the main quantifier for their location. We will see several manifestations of this idea, the most basic being the following:

\begin{thm}\label{thm: basic location}
The adjacency eigenvalues lie in the interval $[-d,d]$. The laplacian eigenvalues lie in the interval $[0,2d]$. 
\end{thm}

In particular, $\lambda_{\min}=0$ in light of Theorem~\ref{thm: plus lap}.

\begin{proof}
Let $\alpha$ be an adjacency eigenvalue, with eigenfunction $f$. Consider the eigenrelation
\begin{align*}
\alpha f(v)=\sum_{u:\: u\sim v} f(u)
\end{align*}
at a vertex $v$ where $|f|$ achieves its maximum. Then $ |\alpha||f(v)|\leq \deg(v)\: |f(v)|\leq d\:|f(v)|$, so $|\alpha|\leq d$. 

Similarly, let $\lambda$ be a laplacian eigenvalue, with eigenfunction $g$. Arguing as before, we use the eigenrelation
\begin{align*}
(\deg(v)-\lambda) g(v)=\sum_{u:\: u\sim v} g(u)
\end{align*}
at a vertex $v$ where $|g|$ achieves its maximum. We obtain $|\deg(v)-\lambda|\leq \deg(v)$, hence $0\leq \lambda\leq 2\deg(v)\leq 2d$. \end{proof}

The number of \emph{distinct} eigenvalues can be nicely and usefully related to the diameter. This is a first illustration of another one of our main themes, that of relating adjacency and laplacian eigenvalues with graph-theoretical invariants.

\begin{thm}
The number of distinct adjacency, respectively laplacian, eigenvalues is at least $\delta+1$.
\end{thm}

\begin{proof}
Let $u$ and $v$ be two distinct vertices. Then $A^r(u,v)=0$ for $r=0,\dots,\mathrm{dist}(u,v)-1$, while $A^r(u,v)\neq 0$ for $r=\mathrm{dist}(u,v)$. Note that the same is true for $L$. On the other hand, a symmetric matrix $M$ with exactly $s$ distinct eigenvalues satisfies a monic polynomial relation of degree $s$. Thus, $M^s$ is a linear combination of $\{M^r: r=0,\dots, s-1\}$, and this continues to hold for higher powers $M^r$, $r\geq s$. We deduce that the number of distinct eigenvalues - adjacency or laplacian - is greater than $\mathrm{dist}(u,v)$ for any two vertices $u$ and $v$, hence greater than the diameter.
\end{proof}
 
It is certainly possible for a graph to have more than $\mathrm{diam}+1$ distinct eigenvalues. Two families of graphs of diameter $2$ but arbitrarily many distinct eigenvalues are the wheel graphs and the Andrasf\'ai graphs; these are discussed in ensuing exercises. Nevertheless, most graphs considered herein turn out to have exactly $\mathrm{diam}+1$ distinct eigenvalues. There are worse habits than mentally checking, whenever faced with a concrete list of eigenvalues for a graph, whether this diameter bound is attained or not.

\bigskip
\subsection*{First examples}\label{sec: basic examples}\
\medskip

In addition to the general theory, we wish to build a stock of concrete spectral computations. We start with some basic examples, including $K_n$, $C_n$, $Q_n$.

\begin{ex} The complete graph $K_n$ has adjacency matrix $A=J-I$, where $J$ is the matrix all of whose entries are $1$. The matrix $J$ has eigenvalues $0$, with multiplicity $n-1$, and $n$. So the adjacency eigenvalues are $-1$, with multiplicity $n-1$, and $n-1$. The laplacian eigenvalues are $0$, respectively $n$ with multiplicity $n-1$.
\begin{align*}
\adspec(K_n)=\{-1,n-1\} \qquad \lapspec(K_n)=\{0,n\}
\end{align*}
\end{ex}

\begin{ex} Consider the cycle graph $C_n$, and label its vertices by $1,\dots, n$ in a circular fashion. Let $\alpha$ be an adjacency eigenvalue, with eigenfunction $f$. Then the following relations hold: generically, $\alpha f_k=f_{k-1}+f_{k+1}$ for $k=1,\dots,n-1$; exceptionally, $\alpha f_1=f_2+f_n$, , and $\alpha f_n=f_{n-1}+f_1$.

The generic relations describe a quadratic recurrence whose solution is $f_k=ax_1^{k}+bx_2^{k}$, where $x_1$ and $x_2$ are the roots of $x^2-\alpha x+1=0$. The exceptional relations read as initial and terminal conditions: $f_{0}=f_n$, $f_{1}=f_{n+1}$. This means that the two relations $a+b=ax_1^{n}+bx_2^{n}$ and $ax_1+bx_2=ax_1^{n+1}+bx_2^{n+1}$ admit a non-zero solution $(a,b)$. Keeping in mind that $x_1x_2=1$, simple arguments lead to $x_1^n=x_2^n=1$. Put $x_1=\xi$, $x_2=\bar{\xi}$ where $\xi^{n}=1$. Then $\alpha=x_1+x_2=2\:\mathrm{Re}(\xi)$, and $f_k=a\xi^{k}+b\bar{\xi}^k$ for $k=0,\dots,n$.
\begin{align*}
\adspec(C_n)&=\Big\{2\:\cos \frac{2\pi k}{n}: k=1,\dots,n \Big\}\\
\lapspec(C_n)&=\Big\{2-2\:\cos \frac{2\pi k}{n}: k=1,\dots,n \Big\}
\end{align*}
\end{ex}

\begin{Rule} Consider the product $X\times Y$ of two graphs $X$ and $Y$. Given two maps  $f:V(X)\to \C$ and $g:V(Y)\to \C$, let $f\times g: V(X\times Y)\to \C$ denote the map given by $(x,y)\mapsto f(x)\: g(y)$. On such functions, the adjacency and the laplacian operators for $X\times Y$ can be expressed as follows: 
\begin{align*}
A(f\times g)&=A_X(f)\times g+ f \times A_Y(g)\\
L(f\times g)&=L_X(f)\times g+ f \times L_Y(g)
\end{align*}
The verification is left to the reader. Consequently, if $f$ is an eigenfunction on $X$ and $g$ is an eigenfunction on $Y$, then $f\times g$ is an eigenfunction on $X\times Y$, with eigenvalue equal to the sum of the eigenvalues for $f$ and $g$. There are no further eigenvalues for $X\times Y$ besides the ones obtained in this way. Indeed, note that $\la f\times g, f'\times g'\ra=\la f, f'\ra\: \la g, g'\ra$. So, if $\{f_i\}_i$ and $\{g_j\}_j$ are orthogonal bases of eigenfunctions on $X$ respectively $Y$, then $\{f_i\times g_j\}_{i,j}$ is an orthogonal basis of eigenfunctions on $X\times Y$. We conclude that the eigenvalues of the product $X\times Y$ are obtained by summing the eigenvalues of $X$ to those of $Y$.
\begin{align*}
\adspec (X\times Y)&=\adspec (X)+\adspec(Y)\\
 \lapspec (X\times Y)&=\lapspec (X)+\lapspec(Y)
\end{align*}
\end{Rule}

\begin{ex} Consider the cube graph $Q_n$, thought of as the $n$-fold product of $K_2\times\dots\times K_2$. The graph $K_2$ has two simple eigenvalues, $1$ and $-1$. Hence the adjacency eigenvalues of $Q_n$ are $(n-k)\cdot 1+k\cdot (-1)=n-2k$ with multiplicity $\bigl(\begin{smallmatrix} n\\ k \end{smallmatrix} \bigr)$, for $k=0,\dots, n$. 
\begin{align*}
\adspec(Q_n)&=\{n-2k: k=0,\dots, n\}\\
 \lapspec(Q_n)&=\{2k: k=0,\dots, n\}
\end{align*}
\end{ex}

\begin{Rule} For a disconnected graph, the adjacency and the laplacian operators admit a natural block-diagonal form, with a block for each connected component. So the eigenvalues of a disconnected graph are obtained by piling together the eigenvalues of its connected components. 
\begin{align*}
\adspec (X\sqcup Y)&=\adspec (X)\cup\adspec(Y)\\ \lapspec (X\sqcup Y)&=\lapspec (X)\cup\lapspec(Y)
\end{align*}
\end{Rule}

\begin{exer}\label{exer: complement} The \emph{complement} of a (possibly disconnected) graph $X$ has the same vertex set as $X$, and two distinct vertices are adjacent whenever they are not adjacent in $X$. 
\begin{itemize}
\item[(i)] Compute the laplacian spectrum of the complement of $X$ in terms of the laplacian spectrum of $X$. 
\item[(ii)] Find the laplacian spectrum of the complete $k$-partite graph $K_{n_1,\dots,n_k}$. This is the graph defined as follows: the vertex set is the disjoint union of $k\geq 2$ nonempty sets of size $n_1,\dots,n_k$, with edges connecting vertices lying in different sets. 
\end{itemize}
\end{exer}

\begin{exer}\label{exer: cone} The \emph{cone} over a (possibly disconnected) graph $X$ is the graph obtained by adding a new vertex, and joining it to all vertices of $X$. \begin{itemize}
\item[(i)] Describe the laplacian spectrum of the cone over $X$ in terms of the laplacian spectrum of $X$. 
\item[(ii)] Do the same for the adjacency spectrum, assuming that $X$ is regular. 
\item[(iii)] Find the adjacency and laplacian spectra of a star graph and a windmill graph. Compare the diameter and the number of distinct eigenvalues for a wheel graph with a large enough number of spokes.
\end{itemize}
\end{exer}

\begin{Rule} Let $X$ be a non-bipartite graph, and consider its bipartite double $X'$. Then the adjacency matrix of $X'$ has the form
\begin{align*}
A'=
\begin{pmatrix}
0 & A\\
A & 0 
\end{pmatrix}
\end{align*}
where $A$ is the adjacency matrix of $X$. Therefore
\begin{align*}
A'^{\: 2}=
\begin{pmatrix}
A^2 & 0\\
0 & A^2 
\end{pmatrix}.
\end{align*}
Keeping in mind that the spectrum of $A'$ is symmetric, it follows that the eigenvalues of $A'$ are $\pm \alpha$ with $\alpha$ running over the eigenvalues of $A$. In short, the adjacency spectrum of a bipartite double is the symmetrized adjacency spectrum of the original graph.
\begin{align*}
\adspec(X')=\pm \adspec(X)
\end{align*}
\end{Rule}

\begin{ex}\label{ex: ad Kmn} Consider the complete bipartite graph $K_{m,n}$. Squaring the adjacency matrix $A$ of $K_{m,n}$ yields
\begin{align*}
A^2=
\begin{pmatrix}
C_\bl & 0\\
0 & C_\wh 
\end{pmatrix}, \qquad C_\bl=nJ_m, \quad C_\wh=mJ_n.
\end{align*}
The eigenvalues of $C_\bl=nJ_m$ are $mn$, with multiplicity $1$, and $0$ with multiplicity $m-1$. Interchanging $m$ and $n$, $C_\wh=mJ_n$ has eigenvalues $mn$, with multiplicity $1$, and $0$ with multiplicity $n-1$. So the eigenvalues of $A$ are $\pm \sqrt{mn}$, each with multiplicity $1$, and $0$ with multiplicity $m+n-2$. 
\begin{align*}
\adspec(K_{m,n})=\{\pm \sqrt{mn}, 0\}
\end{align*}
\end{ex}

\bigskip
\section{Eigenvalue computations}
We have already computed the eigenvalues of some graphs. The focus of this section is on methods that are more or less explicit, and that apply to certain classes of regular graphs. We restrict our attention to their adjacency eigenvalues. 

\bigskip
\subsection*{Cayley graphs of abelian groups}\label{sec: cayley graphs}\
\medskip

Due to their algebraic origin, Cayley graphs can be quite tractable when it comes to understanding their spectral properties. The most basic result in this direction addresses the case of abelian groups.

\begin{thm}\label{thm: ab-spec}
Consider the Cayley graph of a finite abelian group $G$ with respect to $S$, a symmetric generating subset not containing the identity. Then the adjacency eigenvalues are
\begin{align*}
\bigg\{\sum_{s\in S} \:\chi(s): \chi \textrm{ character of } G\bigg\}.
\end{align*} 
\end{thm}

\begin{proof}
Let $\chi$ be a character of $G$. Then
\begin{align*}
A\chi(g)=\sum_{h:\: h\sim g} \:\chi(h)=\sum_{s\in S} \:\chi(gs)=\bigg(\sum_{s\in S} \:\chi(s)\bigg)\:\chi(g)
\end{align*}
and so $\chi$ is an adjacency eigenfunction with eigenvalue $\sum_{s\in S} \:\chi(s)$. As the characters form a basis, there cannot be any further eigenvalues. 
\end{proof}

Note that the trivial character yields the trivial adjacency eigenvalue: $\sum_{s\in S} \ct(s)=|S|$.

\begin{ex} Consider the cycle graph $C_n$, and let us view it as the Cayley graph of the cyclic group $\Z_n$ with respect to $\{\pm 1\}$. The characters of $\Z_n$ are $\chi_k$, $k=0,\dots, n-1$, where $\chi_k$ is defined by $\chi_k(1)=\cos (2\pi k/n)+ i\sin (2\pi k/n)$. The adjacency eigenvalues of $C_n$ are $\chi_k(1)+ \chi_k(-1)=2\cos (2\pi k/n)$ for $k=0,\dots, n-1$.
\end{ex}

\begin{ex} Consider the cube graph $Q_n$. We view it as the Cayley graph of $(\Z_2)^n$ with respect to $\{e_i: 1\leq i\leq n\}$. The characters of $(\Z_2)^n$ are $\chi_{I}(v)=(-1)^{\sum_{i\in I} v_i}$ for $I\subseteq \{1,\dots,n\}$. The adjacency eigenvalue corresponding to $\chi_I$ is
\begin{align*}
\sum_j \chi_{I}(e_j)=\sum_{j\in I} (-1)+\sum_{j\notin I} (1)=n-2|I|.
\end{align*}
So the adjacency eigenvalues of $Q_n$ are $\{n-2k: k=0,\dots,n\}$.
\end{ex}

\begin{ex} Consider the Paley graph $P(q)$, where $q\equiv 1$ mod $4$. The trivial adjacency eigenvalue is $\tfrac{1}{2}(q-1)$. The non-trivial eigenvalues are given by the character sums
\begin{align*}
\alpha_\psi=\sum_{s\in \F^{*2}} \psi(s)
\end{align*}
as $\psi$ runs over the non-trivial additive characters of $\F$. We relate $\alpha_\psi$ with the Gauss sum $G(\psi,\sigma)$, where $\sigma$ is the quadratic signature, as follows:
\begin{align*}
G(\psi,\sigma)&=\sum_{s\in\F^{*}} \psi(s)\:\sigma(s)=\sum_{s\in\F^{*2}} \psi(s)-\sum_{s\in\F^*\setminus \F^{*2}} \psi(s)=2 \sum_{s\in\F^{*2}} \psi(s)-\Big(\sum_{s\in\F^*} \psi(s)\Big)\\
&=2\alpha_\psi +1
\end{align*}
Recall, from the discussion following Theorem~\ref{thm: abs gauss}, that $G(\psi,\sigma)^2=\sigma(-1)\: q$. Here $\sigma(-1)=1$, as $q\equiv 1$ mod $4$, so $G(\psi,\sigma)=\pm \sqrt{q}$. But we also have
\begin{align*}
\sum_{\psi\neq \ct} G(\psi,\sigma)&=\sum_{s\in\F^{*}} \sigma(s)\sum_{\psi\neq \ct} \psi(s)=-\sum_{s\in\F^{*}} \sigma(s)=0
\end{align*}
meaning that the sign of $G(\psi,\sigma)$ is equally often positive as it is negative. 

We conclude that the non-trivial adjacency eigenvalues of $P(q)$ are $\tfrac{1}{2}(\pm \sqrt{q}-1)$, and they both have multiplicity $\tfrac{1}{2}(q-1)$.
\end{ex}

\begin{exer}\label{exer: re-compute}
Find the laplacian eigenvalues of the halved cube graph $\tfrac{1}{2}Q_n$, by using Theorem~\ref{thm: ab-spec}.
\end{exer}

\begin{exer}\label{exer: and5} Find the adjacency spectrum of the Andrasf\'ai graph $A_n$. Compare the diameter and the number of distinct eigenvalues.
\end{exer}

\begin{notes} For a Cayley graph coming from an abelian group, the eigenvalues are given by character sums. It turns out that the same holds when general groups are involved. This requires, however, a general theory of characters which is outside the scope of this text. See Babai (\emph{Spectra of Cayley graphs}, J. Combin. Theory Ser. B 1979) and also the discussion of Murty (\emph{Ramanujan graphs}, J. Ramanujan Math. Soc. 2003). Still, we will take a step in the non-abelian direction while still relying on character theory for abelian groups. In the next paragraph, we compute the eigenvalues of bi-Cayley graphs of abelian groups. As we know, these are certain Cayley graphs of generalized dihedral groups. 
\end{notes}


\bigskip
\subsection*{Bi-Cayley graphs of abelian groups}\label{sec: bi-cayley graphs}\
\medskip

The following is an adaptation of Theorem~\ref{thm: ab-spec} to a bipartite setting. 

\begin{thm}\label{thm: mock Cayley}
Consider the bi-Cayley graph of a finite abelian $G$ with respect to $S$. Then the adjacency eigenvalues are
\begin{align*}
\bigg\{\pm\Big|\sum_{s\in S}\: \chi(s)\Big|: \chi \textrm{ character of } G\bigg\}.
\end{align*} 
\end{thm}

\begin{proof}
Let $\chi$ be a character of $G$, and put 
\begin{align*}
\alpha_\chi=\sum_{s\in S}\: \chi(s).
\end{align*} 
We look for an adjacency eigenfunction $\theta$ of the form $\theta=\chi$ on $G_\bl$, and $\theta=c(\chi)\:\chi$ on $G_\wh$ for a suitable non-zero scalar $c(\chi)$. A good choice is 
\begin{align*}
c(\chi)=\frac{\overline{\alpha_\chi}}{|\alpha_\chi|}
\end{align*}
with the convention that $c(\chi)=1$ when $\alpha_\chi=0$. Indeed, we have:
\begin{align*}
A\theta (g_\bl)&=\sum_{h_\wh\sim g_\bl} \theta(h_\wh)=c(\chi)\: \sum_{s\in S}\:\chi(gs)=c(\chi)\:\alpha_\chi\:\chi(g)=|\alpha_\chi|\:\theta (g_\bl)\\
A\theta (g_\wh)&=\sum_{h_\bl\sim g_\wh} \theta(h_\bl)=\sum_{s\in S}\:\chi(gs^{-1})=\overline{\alpha_\chi}\:\chi(g)=|\alpha_\chi|\:\theta (g_\wh)
\end{align*}
and so $\theta$ is an eigenfunction with eigenvalue $|\alpha_\chi|$. Switching the sign of $\theta$ on $G_\wh$ produces an eigenfunction $\theta'$ with eigenvalue $-|\alpha_\chi|$. The $2|G|$ eigenfunctions we have found, $\theta=\theta(\chi)$ and $\theta'=\theta'(\chi)$ for $\chi$ running over the characters of $G$, are easily seen to be linearly independent. So there are no further eigenvalues.
\end{proof}

For Cayley graphs of abelian groups, the characters provide a basis of eigenfunctions that does not depend on the choice of $S$. Here, characters provide a basis of eigenfunctions that does depend on $S$, but in a rather weak way: namely, the phase $c(\chi)$ depends on $S$. For later use, we remark that $c(\chi)=\pm 1$ whenever $\alpha_\chi$ is real.

\begin{ex}\label{ex: paley via cayley}
 Consider the bi-Paley graph $BP(q)$, for $q\equiv 3$ mod $4$. If we put again
\begin{align*}
\alpha_\psi=\sum_{s\in \F^{*2}} \psi(s)
\end{align*}
then the non-trivial eigenvalues of $BP(q)$ are $\pm |\alpha_\psi|$ as $\psi$ runs over the non-trivial additive characters of $\F$. Now $2\alpha_\psi +1=G(\psi,\sigma)$, and $G(\psi,\sigma)^2=\sigma(-1)\: q=-q$. Thus $\alpha_\psi=\tfrac{1}{2}(-1\pm \mathrm{i}\sqrt{q})$, and so $|\alpha_\psi|=\tfrac{1}{2}\sqrt{q+1}$. In conclusion, the adjacency eigenvalues of $BP(q)$ are $\pm \tfrac{1}{2}(q-1)$ and $\tfrac{1}{2}\sqrt{q+1}$.
\end{ex}

\begin{ex}\label{ex: compute inc} The incidence graph $I_n(q)$ may be viewed as the bi-Cayley graph of the multiplicative group $\K^*/\F^*$ with respect to the subset $\{[s]: \Tr(s)=0\}$. Here $\K$ is a degree $n$ extension of $\F$, and $\Tr:\K\to\F$ is the trace. The characters of $\K^*/\F^*$ correspond, through the quotient homomorphism $[\cdot]: \K^*\to \K^*/\F^*$, to the multiplicative characters of $\K$ that are trivial on $\F$. So the adjacency eigenvalues are $\pm |\alpha_{[\chi]}|$, where
\begin{align*}
\alpha_{[\chi]}=\sum_{\substack{[s]\in \K^*/\F^* \\ \Tr(s)=0}} [\chi]([s])= \frac{1}{q-1}\:\sum_{\substack{s\in \K^*\\ \Tr(s)=0}}\chi(s)=\frac{1}{q-1}\:E_0(\chi).
\end{align*}
Recall, $E_0(\chi)$ denotes the singular Eisenstein sum associated to $\chi$. The trivial character yields the simple eigenvalues $\pm |\alpha_{[\ct]}|=\pm (q^{n-1}-1)/(q-1)$. For a non-trivial $\chi$, we have $E_0(\chi)=-(q-1)\: E(\chi)$ by Lemma~\ref{lem: sing eins}, and so 
\begin{align*}
\alpha_{[\chi]}=-E(\chi).
\end{align*}
Applying Theorem~\ref{thm: abs eins}, we obtain the eigenvalues $\pm |\alpha_{[\chi]}|=\pm q^{n/2-1}$, each with multiplicity $|\K^*/\F^*|-1=(q^n-q)/(q-1)$.
\end{ex}

\begin{exer}\label{exer: SP1} Let $\F$ be a finite field with $q\geq 3$ elements. The \emph{sum-product graph} $SP(q)$ is the bipartite graph on two copies of $\F\times \F^*$, in which vertices $(a,x)_\bl$ and $(b,y)_\wh$ are connected when $a+b=xy$. Compute the adjacency eigenvalues of $SP(q)$.
\end{exer}

\begin{notes}
In its directed version, the sum-product graph was first considered by Friedman (\emph{Some graphs with small second eigenvalue}, Combinatorica 1995).
\end{notes}


\bigskip
\subsection*{Strongly regular graphs}\label{sec: Strongly regular graphs}\
\medskip

Recall that a strongly regular graph is a regular graph with the property that, for some non-negative integers $a$ and $c$, any two adjacent vertices have $a$ common neighbours, and any two distinct non-adjacent vertices have $c$ common neighbours.

A strongly regular graph has diameter $2$, so we expect at least three distinct eigenvalues. The following proposition explicitly computes the adjacency eigenvalues, and their multiplicities, for a strongly regular graph.

\begin{thm}\label{thm: strong spectra}
A strongly regular graph has three distinct adjacency eigenvalues. In terms of the parameters $(n,d,a,c)$, they are given as follows: the simple eigenvalue $\alpha_1=d$, and two more eigenvalues
 \begin{align*}
 \alpha_{2,3}=\frac{a-c\pm \sqrt{\mathrm{disc}}}{2}, \qquad \mathrm{disc}=(a-c)^2+4(d-c)>0
 \end{align*}
with multiplicities
\begin{align*}
m_{2,3}=\frac{1}{2}\Bigg(n-1\mp \frac{(n-1)(a-c)+2d}{\sqrt{\mathrm{disc}}}\Bigg).
\end{align*}
Conversely, a regular graph having three distinct adjacency eigenvalues is strongly regular.
\end{thm}

\begin{proof}
Squaring the adjacency matrix $A$, we see that $A^2(u,v)$ equals $d$ if $u=v$, $a$ if $u\sim v$, and $c$ otherwise. Therefore $A^2=(a-c)A+(d-c)I+c J$, where $J$ is the $n\times n$ matrix all of whose entries are $1$. 

Now let $\alpha$ be a non-trivial eigenvalue, with eigenfunction $f$. Applying the above relation to $f$, and keeping in mind that $Jf=0$, we get $\alpha^2=(a-c)\alpha+(d-c)$. This quadratic equation yields the claimed formulas for $\alpha_2$ and $\alpha_3$.

The corresponding multiplicities $m_{2,3}$ can be determined by combining the dimensional relation $1+m_2+m_3=n$ with the tracial relation $d+m_2\alpha_2+m_3\alpha_3=0$. Solving this system, and using the formulas for $\alpha_2$ and $\alpha_3$, we are led to  the desired explicit formulas for $m_2$ and $m_3$. 

For the converse, let $\alpha_2$ and $\alpha_3$ denote the non-trivial eigenvalues, and consider $M:=A^2-(\alpha_2+\alpha_3)A+\alpha_2\alpha_3 I$. We know that $(A-d I)M=0$. This means that every column of $M$ is an eigenvector for $d$, hence $M$ is constant along columns. On the other hand, $M$ is constant along the diagonal, by regularity. Thus $M$ is a constant multiple of $J$, and strong regularity follows. Furthermore, we may read off the parameters: besides $d=\alpha_1$, we have $a-c=\alpha_2+\alpha_3$ and $d-c=-\alpha_2\alpha_3$. \end{proof}

\begin{ex} The Paley graph $P(q)$ has adjacency eigenvalues $\frac{1}{2}(\pm\sqrt{q}-1)$, each with multiplicity $\frac{1}{2}(q-1)$, as well as the simple eigenvalue $\frac{1}{2}(q-1)$. 
\end{ex}

\begin{ex} The twin graphs have adjacency eigenvalues $-2,2,6$, with multiplicities $\times 9$, $\times 6$, respectively $\times 1$. 
\end{ex}

\begin{ex}
The Petersen graph has adjacency eigenvalues $-2,1,3$, with multiplicities $\times 4$, $\times 5$, respectively $\times 1$. 
\end{ex}

Theorem~\ref{thm: strong spectra} implies that the adjacency spectrum of a strongly regular graph is completely determined by its parameters. In particular, strongly regular graphs with the same parameters are isospectral. In this way we get examples of isospectral but not isomorphic graphs, for instance the twin graphs, or the graphs provided by Corollary~\ref{cor: machine}.

A more surprising upshot of Theorem~\ref{thm: strong spectra} has to do with the multiplicities $m_{2,3}$ of the non-trivial eigenvalues. Namely, their integrality imposes strong restrictions on the parameters $(n,d,a,c)$. There are two compatible possibilities, named according to the nature of the eigenvalues.

\smallskip
\textbf{Quadratic type:} $(n-1)(a-c)+2d=0$. As $d<n-1$, we must have $c-a=1$ and $n-1=2d$. Using Proposition~\ref{prop: double counting}, we then deduce that $2c=d$. Thus $(n,d,a,c)=(4c+1, 2c, c-1, c)$. The non-trivial eigenvalues are $\alpha_{2,3}=(-1\pm \sqrt{n})/2$, so they are real quadratic integers, and they have equal multiplicities. The Paley graph $P(q)$ is of quadratic type. 
 
\smallskip
\textbf{Integral type:} the discriminant $\mathrm{disc}$ is a square, say $t^2$, and $t$ divides $(n-1)(a-c)+2d$. As $t$ and $a-c$ have the same parity, the non-trivial eigenvalues $\alpha_{2,3}=(a-c\pm t)/2$ are integral. They have different multiplicities, unless we are in the quadratic type as well. The Petersen graph, the twin graphs, and $K_n\times K_n$ are of integral type.

\bigskip
\subsection*{Two applications}\
\medskip

To illustrate just how useful the integrality of the multiplicities can be, let us consider two classical, but nevertheless striking, applications. Much of their appeal comes from the fact that we are establishing purely combinatorial results by exploiting spectral constraints! 
  
A graph of diameter $2$ has girth at most $5$. This bound is achieved by two familiar graphs, the cycle $C_5$ and the Petersen graph. Our first gem addresses the following simple question: are there any other regular graphs of diameter $2$ and girth $5$? 

\begin{thm}[Hoffman - Singleton]\label{thm: h-s}
If a regular graph has diameter $2$ and girth $5$, then the size and the degree are one of the following: $n=5$ and $d=2$; $n=10$ and $d=3$; $n=50$ and $d=7$; $n=3250$ and $d=57$.
\end{thm}

\begin{proof}
A regular graph with diameter $2$ and girth $5$ is strongly regular with parameters $a=0$ and $c=1$. Indeed, adjacent vertices have no common neighbours since there are no $3$-cycles. Distinct non-adjacent vertices must have common neighbours, by the diameter assumption. There is in fact just one common neighbour, since there are no $4$-cycles.

The only case of quadratic type is $(n,d)=(5,2)$. Turning to the integral type, we let $t^2=\mathrm{disc}=4d-3$, so $t$ divides $(n-1)-2d$. But $n=d^2+1$ by Proposition~\ref{prop: double counting}. Now $t$ divides $4d-3$ and $d^2-2d$, and simple arithmetic manipulations imply that $t$ divides $15$. If $t=1$ then $d=1$, which is impossible. The other possibilities, $t=3$, $t=5$, or $t=15$, lead to $(n,d)$ being one of $(10,3)$, $(50,7)$, $(3250,57)$. 
\end{proof}

In the previous theorem, the first two cases are uniquely realized by the two graphs that we already know, the $5$-cycle and the Petersen graph. It turns out that the third case, $n=50$ and $d=7$, is also realized by a unique graph--nowadays called the \emph{Hoffman - Singleton graph}. Amazingly, it is still not known whether there is a graph realizing the last case, $n=3250$ and $d=57$.

Our second gem is often stated as the \emph{Friendship theorem}: if, in a group of people, any two persons have exactly one common friend, then there is a person who is everybody's friend. A more precise graph-theoretic statement is the following.

\begin{thm}[Erd\H os - R\'enyi - S\'os]\label{thm: friend}
 A graph with the property that any two distinct vertices have exactly one common neighbour, is a windmill graph.
\end{thm}
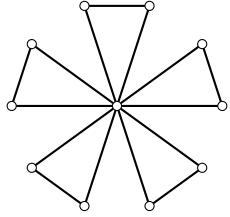
\begin{figure}[!ht]\bigskip
\centering
\GraphInit[vstyle=Simple]
\tikzset{VertexStyle/.style = {
shape = circle,
inner sep = 0pt,
minimum size = .8ex,
draw}}
\begin{tikzpicture}
  \grStar[RA=1.4]{11}
  \Edges(a0, a1)
  \Edges(a2, a3)
  \Edges(a4, a5)
  \Edges(a6, a7)
  \Edges(a8, a9)
\end{tikzpicture}
\end{figure}

\medskip
\begin{proof}
Let us first consider the familiar context of regular graphs. If the given graph is complete, then it must be $K_3$, the one-bladed windmill graph. Otherwise, the graph is strongly regular with $a=c=1$, so it is of integral type. Let $t^2=\mathrm{disc}=4(d-1)$, where $t$ divides $2d$. Putting $t=2t'$, we have that $t'$ divides $d-1$ and $d$, hence $t'=1$ and so $d=2$. However, no cycle can have $a=c=1$.

In the irregular case, the argument is combinatorial. The first observation is that $\mathrm{deg}(u)=\mathrm{deg}(v)$ whenever $u$ and $v$ are non-adjacent vertices. Indeed, let $u_1,\dots, u_k$ be the neighbours of $u$, and let $v_i$ be the common neighbour of $v$ and $u_i$. If $v_i=v_j$ for $i\neq j$, then that vertex and $u$ have at least two common neighbours, namely $u_i$ and $u_j$, and this is a contradiction. Thus $v_1,\dots, v_k$ are distinct neighbours of $v$, hence $\mathrm{deg}(v)\geq \mathrm{deg}(u)$. Interchanging the roles of $u$ and $v$, we get $\mathrm{deg}(u)=\mathrm{deg}(v)$ as desired. 

Let $x$ and $y$ be adjacent vertices having different degrees. Then every other vertex is adjacent to $x$ or to $y$, possibly to both. We claim that, in fact, one of the following holds: every vertex different from $x$ is adjacent to $x$, or every vertex different from $y$ is adjacent to $y$. Indeed, if this is not the case, then there are two vertices $x'$ and $y'$ such that $x'$ is adjacent to $x$ but not to $y$, and $y'$ is adjacent to $y$ but not to $x$. Let $z$ be the common neighbour of $x'$ and $y'$. Now $z$ is adjacent to $x$ or to $y$, say $z\sim y$. But then $x$ and $z$ have two common neighbours, namely $x'$ and $y$. This is a contradiction, and our claim is proved.

So there is a vertex $w$ which is connected to every other vertex. For each $v$ adjacent to $w$, there is a unique $v'$ adjacent to $w$ and to $v$. This shows that the vertices different from $w$ are joined in disjoint pairs, so the graph is a windmill graph.
\end{proof}

\begin{exer}\label{exer: nobip} Show that the incidence graph $I_n(q)$ is not a bipartite double.
\end{exer}

\begin{notes}
Theorem~\ref{thm: h-s} is due to Hoffman and Singleton (\emph{On Moore graphs with diameters $2$ and $3$}, IBM J. Res. Develop. 1960). Theorem~\ref{thm: friend} is due to Erd\H os, R\'enyi, and S\'os (\emph{On a problem of graph theory}, Studia Sci. Math. Hungar. 1966). 
\end{notes}


\bigskip
\subsection*{Design graphs}\label{sec: designs}\
\medskip

Recall that a design graph is a regular bipartite graph with the property that two distinct vertices of the same colour have the same number of common neighbours. The relevant parameters are $m$, the half-size; $d$, the degree; $c$, the number of neighbours shared by any monochromatic pair of vertices.

\begin{thm}\label{thm: eigen-design}
A design graph has four distinct adjacency eigenvalues. In therms of the parameters $(m,d,c)$, they are given as follows: the simple eigenvalues $\pm d$, and $\pm \sqrt{d-c}$, each with multiplicity $m-1$.

Conversely, a regular bipartite graph having four distinct eigenvalues is a design graph. 
\end{thm}

\begin{proof} As for any bipartite graph, the adjacency matrix of a design graph satisfies
\begin{align*}
A^2=
\begin{pmatrix}
C_\bl & 0\\
0 & C_\wh 
\end{pmatrix}
\end{align*}
where $C_\bl$ and $C_\wh$ are square matrices indexed by the black, respectively by the white vertices. These two matrices have a simple combinatorial description: $C_\bl$ and $C_\wh$ record, diagonally, the degrees in each colour, and off-diagonally the number of common neighbours for distinct vertices of the same colour.

The design property means that
\begin{align*}
C_\bl=C_\wh=(d-c)\:I_m+c\:J_m.
\end{align*}
So the eigenvalues of both $C_\bl$ and $C_\wh$ are $(d-c)+cm=d^2$, with multiplicity $1$, and $d-c$, with multiplicity $m-1$. The claim follows. 

For the converse, let $\pm d, \pm \alpha$ be the distinct eigenvalues. As in the proof of Theorem~\ref{thm: strong spectra}, consider the matrix
\begin{align*}
M:=(A+dI)(A^2-\alpha^2I)=A^3+dA^2-\alpha^2A-d\alpha^2I.
\end{align*} 
Then $(A-dI)M=0$ and $M$ is constant along the diagonal, so $M$ is constant multiple of $J$, say $M=kJ$. Now if $u$ and $v$ are two distinct vertices of the same colour, then $k=M(u,v)=dA^2(u,v)$. This is precisely the design property. \end{proof}

\begin{ex} The incidence graph $I_n(q)$ is a design graph, with parameters 
\begin{align*}
(m,d,c)=\Big(\frac{q^n-1}{q-1},\frac{q^{n-1}-1}{q-1},\frac{q^{n-2}-1}{q-1}\Big).
\end{align*} 
So the adjacency eigenvalues are $\pm (q^{n-1}-1)/(q-1)$ and $\pm q^{n/2-1}$.
\end{ex}

\begin{ex} The bi-Paley graph $BP(q)$, where $q\equiv 3$ mod $4$, is a design graph with parameters
\begin{align*}
(m,d,c)=\Big(q, \tfrac{1}{2}(q-1), \tfrac{1}{4}(q-3)\Big).
\end{align*} 
So the adjacency eigenvalues are $\pm \frac{1}{2}(q-1)$ and $\pm \frac{1}{2}\sqrt{q+1}$.
\end{ex}

Design graphs with the same parameters are isospectral. For example: if $q=2^n-1$ is a Mersenne prime, where $n>3$, then the bi-Paley graph $BP(q)$ and the incidence graph $I_n(2)$ are isospectral but not isomorphic.

\bigskip
\subsection*{Partial design graphs}\label{sec: pardesigns}\
\medskip

Recall that partial design graph is a regular bipartite graph with the property that there are only two possible values for the number of neighbours shared by two vertices of the same colour. 

To a partial design graph with parameters $(m,d,c_1,c_2)$ we associate a derived graph, the \emph{$c_1$-graph}. This is the (possibly disconnected) graph on the $m$ black vertices of the partial design graph, in which edges connect pairs of vertices that have exactly $c_1$ common neighbours.

\begin{thm}\label{thm: pdg}
Given a partial design graph with parameters $(m,d,c_1,c_2)$, the following hold.
\begin{itemize}
\item[(i)] The $c_1$-graph is regular of degree 
\begin{align*}
d'=\frac{d(d-1)-(m-1)c_2}{c_1-c_2}.
\end{align*}
\item[(ii)] Let $\alpha_1'=d', \alpha_2', \ldots, \alpha_m'$ be the adjacency eigenvalues of the $c_1$-graph. Then the adjacency eigenvalues of the partial design graph are as follows: 
\begin{align*}
\pm d,\;\pm\sqrt{(d-c_2)+(c_1-c_2)\alpha_k'} \qquad (k=2,\dots,m)
\end{align*}\end{itemize}
\end{thm}

The proof uses the following:

\begin{lem} Let $M$ and $N$ be square matrices of the same size. Then $MN$ and $NM$ have the same eigenvalues.
\end{lem}

\begin{proof}
We claim that $\det (z I -MN)=\det (z I -NM)$ for each $z\in \C$. This is clear when $z=0$. As for non-zero $z$, it suffices to treat the case $z=1$. In the identity
\begin{align*}
\begin{pmatrix}
I & 0\\
N & I 
\end{pmatrix}
\begin{pmatrix}
I & M\\
0 & I-NM 
\end{pmatrix}=
\begin{pmatrix}
I & M\\
N & I 
\end{pmatrix}=
\begin{pmatrix}
I-MN & M\\
0 & I 
\end{pmatrix}
\begin{pmatrix}
I & 0\\
N & I 
\end{pmatrix}
\end{align*}
the determinant of the left-most product is $\det(I-NM)$, whereas for the right-most product it is $\det(I-MN)$.
\end{proof}

\begin{proof}[Proof of Theorem~\ref{thm: pdg}]
Let 
\begin{align*}
A=\begin{pmatrix}
0 & B\\
B^t & 0 
\end{pmatrix}
\end{align*}
be the adjacency matrix of the partial design graph, where $B$ is an  $m\times m$ matrix. As in the case of design graphs, we consider the square of $A$:
\begin{align*}
A^2=
\begin{pmatrix}
C_\bl & 0\\
0 & C_\wh 
\end{pmatrix}
\end{align*}
where $C_\bl=BB^t$ and $C_\wh=B^tB$. The two matrices $C_\bl$ and $C_\wh$ have the same eigenvalues, thanks to the previous lemma. Hence the eigenvalues of $A$ are $\pm \sqrt{\mu}$ for $\mu$ running over the eigenvalues of $C_\bl$. Each of $\pm \sqrt{\mu}$ has a multiplicity equal to that of $\mu$ when $\mu\neq 0$. If $\mu=0$, then $0$ is an eigenvalue of $A$ with twice the multiplicity of $\mu=0$.

The characteristic property for a partial design means that the $m\times m$ matrix $C_\bl$ can be expressed as follows. Its $(u_\bl,v_\bl)$ entry is $d$ if $u_\bl=v_\bl$; $c_1$ if $u_\bl$ and $v_\bl$ are adjacent in the $c_1$-graph; $c_2$ if $u_\bl$ and $v_\bl$ are distinct and not adjacent in the $c_1$-graph. In short,
\begin{align*}
C_\bl=(d-c_2)\: I+c_2\:J+(c_1-c_2)\: A'
\end{align*}
where $A'$ is the adjacency matrix of the $c_1$-graph.

(i) Thanks to regularity, $A^2$ has $\ct$ as an eigenvector, with eigenvalue $d^2$. It follows, by restricting, that the same holds for $C_\bl$. Applying the above expression for $C_\bl$ to $\ct$, we obtain
\begin{align*}
A'\: \ct=\frac{d(d-1)-(m-1)c_2}{c_1-c_2}\: \ct
\end{align*}
so the $c_1$-graph is regular, of degree $d'$ equal to the coefficient on the right-hand side.

(ii) Let $f_1'=\ct, f_2',\dots,f_m'$ be an orthogonal basis of eigenvectors corresponding to the eigenvalues $\alpha_1'=d', \alpha_2', \ldots, \alpha_m'$ of $A'$. Then $C_\bl f_1'=d^2 f_1'$, and
\begin{align*}
C_\bl f_k'=\big((d-c_2)+(c_1-c_2)\alpha_k'\big) f_k'\qquad k=2,\dots,m
\end{align*} 
since $Jf_k'=0$. Thus, the eigenvalues of $C_\bl$ are $d^2$, and $(d-c_2)+(c_1-c_2)\alpha_k'$ for $k=2,\dots,m$. We deduce that the eigenvalues of $A$ are as claimed.
\end{proof}

\begin{ex} Consider the Tutte - Coxeter graph (Example~\ref{ex: T-C}). It is a partial design graph, with parameters $(m,d,c_1,c_2)=(15, 3, 0, 1)$. The $0$-graph is a regular graph of degree $8$ on $15$ vertices, which can be described as follows: the vertices are the edges of $K_6$, and two edges are adjacent when they share a common endpoint. This is, in fact, a strongly regular graph with parameters $a=c=4$. We first find that the $0$-graph has adjacency eigenvalues $8$; $2$ with multiplicity $\times 5$; $-2$ with multiplicity $\times 9$. Then, we deduce the adjacency eigenvalues of the Tutte - Coxeter graph: $\pm 3$; $\pm 2$, each with multiplicity $\times 9$; $0$ with multiplicity $\times 10$.
\end{ex}

\begin{ex} Let $\F$ be a finite field with $q\geq 3$ elements. The sum-product graph $SP(q)$ is the bipartite graph on two copies of $\F\times \F^*$, in which $(a,x)_\bl$ and $(b,y)_\wh$ are connected whenever $a+b=xy$. Then $SP(q)$ is a partial design graph with parameters $(m,d,c_1,c_2)=(q(q-1), q-1, 0,1)$. Thanks to the automorphism that switches the colour of each vertex, it suffices to check the partial design property on black vertices. Let $(a,x)_\bl$ and $(a',x')_\bl$ be distinct vertices. A vertex $(b,y)_\wh$ is adjacent to both if and only if $b=xy-a=x'y-a'$. Thus $b$ is determined by $y$, and $a-a'=(x-x')y$. There is no solution $y$ when $a=a'$ or $x=x'$, otherwise there is a unique solution $y$. The $0$-graph has vertex set $\F\times \F^*$, and edges connect $(a,x)$ and $(a',x')$ whenever $a=a'$ or $x=x'$. This is the product $K_q\times K_{q-1}$. Its adjacency eigenvalues are $2q-3$; $q-2$ with multiplicity $\times(q-2)$; $q-3$ with multiplicity $\times(q-1)$; $-2$ with multiplicity $\times(q-1)(q-2)$. Hence $SP(q)$ has adjacency eigenvalues $\pm (q-1)$; $\pm \sqrt{q}$ with multiplicity $\times(q-1)(q-2)$; $\pm 1$ with multiplicity $\times(q-1)$; $0$ with multiplicity $\times 2(q-2)$.
\end{ex}

\begin{ex} Consider now the following extended relative of $SP(q)$. The \emph{ full sum-product graph} $FSP(q)$ is the bipartite graph on two copies of $\F\times \F$, in which $(a,x)_\bl$ and $(b,y)_\wh$ are connected whenever $a+b=xy$. This is a partial design graph with parameters $(m,d,c_1,c_2)=(q^2, q, 0,1)$. As before, let $(a,x)_\bl$ and $(a',x')_\bl$ be distinct vertices in $FSP(q)$. A common neighbour $(b,y)_\wh$ is determined by $b=xy-a=x'y-a'$. Now $a-a'=(x-x')y$ has no solution when $x=x'$, and one solution otherwise. The $0$-graph has vertex set $\F\times \F$, and edges connect $(a,x)$ and $(a',x')$ whenever $x=x'$. This is a disconnected union of $q$ copies of $K_q$, and its adjacency eigenvalues are $q-1$ with multiplicity $\times q$, and $-1$ with multiplicity $\times q(q-1)$. We find that $FSP(q)$ has adjacency eigenvalues $\pm q$; $\pm \sqrt{q}$ with multiplicity $\times q(q-1)$; $0$ with multiplicity $\times 2(q-1)$.
\end{ex}

\begin{exer}\label{exer: halved design} 
Compute the laplacian eigenvalues of the halved cube graph $\tfrac{1}{2}Q_n$ by using Theorem~\ref{thm: pdg}.
\end{exer}

\begin{exer}\label{exer: higherincgraph} 
Let $V$ be a linear space of dimension $n\geq 5$ over a field with $q$ elements. Compute the adjacency eigenvalues of the following incidence graph: the vertices are the $2$-spaces and the $(n-2)$-spaces, joined according to inclusion.
\end{exer}

\bigskip
\section{Largest eigenvalues}\label{sec: extremal}

\bigskip
\subsection*{Extremal eigenvalues of symmetric matrices}\
\medskip

Both the adjacency matrix and the laplacian matrix of a graph are real symmetric matrices. Let us discuss a simple principle concerning the extremal eigenvalues of real symmetric matrices.

Let $M$ be a real symmetric $n\times n$ matrix, with eigenvalues
\begin{align*}
\mu_{\min}=\mu_1\leq \ldots\leq \mu_n=\mu_{\max}.
\end{align*} 
The following result characterizes the extremal eigenvalues as extremal values of the \emph{Rayleigh ratio}
\begin{align*}
R(f):=\frac{\la Mf, f\ra }{\la f,f\ra},
\end{align*}
where $0\neq f\in \C^n$. Note that the Rayleigh ratio is real-valued.

\begin{thm}\label{thm: unconstrained} We have
\begin{align*}
\mu_{\min}=\min_{f\neq 0}\: R(f), \qquad \mu_{\max}=\max_{f\neq 0}\: R(f)
\end{align*}
and the extremal values are attained precisely on the corresponding eigenvectors.
\end{thm}

\begin{proof}
Let $f_1, f_2, \dots, f_n$ be an orthonormal basis of eigenvectors. Decomposing $f\neq 0$ as $f=\sum c_i f_i$, we have $Mf=\sum \mu_i c_i f_i$, so
\begin{align*}
R(f)=\frac{\la Mf, f\ra }{\la f,f\ra}=\frac{\sum \mu_i |c_i|^2}{\sum |c_i|^2}.
\end{align*}
Hence $\mu_{\min}\leq R(f)\leq\mu_{\max}$. If $f$ is an eigenvector for $\mu_{\min}$ then $R(f)=\mu_{\min}$, and similarly for $\mu_{\max}$. Conversely, if $R(f)$ equals, say, $\mu_{\min}$, then $c_i=0$ whenever $\mu_i\neq \mu_{\min}$. Thus $f$ is in the $\mu_{\min}$-eigenspace. 
\end{proof}

As a first illustration of this spectral principle, let us prove the following fact concerning the largest laplacian eigenvalue. We know that $\lambda_{\max}\leq 2d$, but when is the upper bound achieved?

\begin{prop}
We have $\lambda_{\max}\leq 2d$, with equality if and only if the graph is regular and bipartite.
\end{prop}

\begin{proof}
The Rayleigh ratio for the laplacian is bounded above by $2d$. Indeed, for every $f\in \ell^2V$ we have
\begin{align*}
\la L f,f\ra &=\frac{1}{2}\sum_{u\sim v} \big|f(u)-f(v)\big|^2\leq \sum_{u\sim v} \big(|f(u)|^2+|f(v)|^2\big)=2\sum_v\mathrm{deg}(v)\: |f(v)|^2\\
&\leq 2d\: \sum_{v}|f(v)|^2=2d\: \la f,f\ra.
\end{align*}
We see, once again, that $\lambda_{\max}\leq 2d$. But now we can analyze the case of equality. If $\lambda_{\max}= 2d$ then $\la L f,f\ra = 2d\: \la f,f\ra$ for some $f\neq 0$. This  means that, in the above estimate, both inequalities are in fact equalities. The first is an equality when $f(u)+f(v)=0$ for every pair of adjacent vertices $u$ and $v$. By connectivity, there is a constant $c\neq 0$ such that $f$ is $\pm c$-valued. The choice of sign corresponds to a bipartite structure on the graph. As $f$ never vanishes, the second inequality is an equality when every vertex has degree $d$. In conclusion, the graph is regular and bipartite. Conversely, if the graph is regular and bipartite then $2d$ is an eigenvalue. 
\end{proof}

\bigskip
\subsection*{Largest adjacency eigenvalue}\
\medskip

\begin{thm}\label{thm: simple-positive}
The adjacency eigenvalue $\alpha_{\max}$ is simple, and it has a positive eigenfunction.
\end{thm}

This theorem is the single most important fact about $\alpha_{\max}$. Let us explain it, prove it, and then use it to derive a number of consequences.

A function on a graph is \emph{positive} if all its values are positive. On the laplacian side, we know that $\lambda_{\min}=0$ is a simple eigenvalue, and it has a positive eigenfunction--namely, the constant function $\mathbb{1}$. Theorem~\ref{thm: simple-positive} should be seen as an adjacency dual of this laplacian fact.

\begin{proof}
Let $f$ be an eigenfunction of $\alpha_{\max}$. We claim that $|f|$ is a positive eigenfunction of $\alpha_{\max}$. We have
\begin{align*}
\la A|f|,|f|\ra&=\sum_{u\sim v} |f|(u)\: |f|(v)\geq \Big|\sum_{u\sim v} f(u)\: \bar{f}(v)\Big|=\big|\la Af,f\ra\big|
\end{align*}
and  \begin{align*}
\big|\la Af,f\ra\big|=\big|\alpha_{\max} \la f,f\ra\big|=\alpha_{\max}\: \la f,f\ra=\alpha_{\max}\:\la |f|,|f|\ra.
\end{align*} 
Here we are using the fact that $\alpha_{\max}\geq 0$; this follows from $\sum \alpha_k=0$.

Therefore $\la A|f|,|f|\ra=\alpha_{\max}\:\la |f|,|f|\ra$, so $|f|$ is indeed an eigenfunction of $\alpha_{\max}$. To show that $|f|$ is positive, consider the eigenrelation
\begin{align*}
\sum_{u: u\sim v} |f|(u)=\alpha_{\max}\:|f|(v).
\end{align*}
If $|f|$ vanishes at a vertex, then $|f|$ vanishes at each one of its neighbours. Thus, vanishing is contagious, but then connectivity would it pandemic: that is, $|f|\equiv 0$. This is a contradiction.

In addition, the inequality $\sum_{u\sim v} |f|(u)\: |f|(v)\geq \big|\sum_{u\sim v} f(u)\: \bar{f}(v)\big|$ we used along the way must be an equality. Now, if $\{z_j\}$ are complex numbers satisfying $\big|\sum z_j\big|=\sum |z_j|$, then they have the same phase: there is $\omega\in \C$ with $|\omega|=1$ such that $z_j=\omega|z_j|$ for all $j$. In our case, $f(u)\: \bar{f}(v)=\omega\: |f(u)|\: |f(v)|$ or
\begin{align}
\frac{f(u)}{|f(u)|}=\omega\: \frac{f(v)}{|f(v)|}\label{eq: quasi-prop}
\end{align}
 whenever $u$ and $v$ are adjacent. As $f$ and $|f|$ are eigenfunctions of $\alpha_{\max}$, we have
 \begin{align*}
\alpha_{\max}\:f(v)=\sum_{u: u\sim v} f(u)=\omega\: \frac{f(v)}{|f(v)|} \sum_{u: u\sim v} |f(u)|=\omega\: \frac{f(v)}{|f(v)|}\: \alpha_{\max}\:|f(v)|=\omega\:\alpha_{\max}\:f(v).
\end{align*}
for an arbitrarily fixed vertex $v$. Thus $\omega=1$, and \eqref{eq: quasi-prop} says that $f/|f|$ is constant across edges. Thanks to connectivity, it follows that $f/|f|$ is constant. In other words, $f$ is a scalar multiple of a positive function. 

We may now conclude that the eigenspace of $\alpha_{\max}$ is one-dimensional. Otherwise, we could pick two orthogonal eigenfunctions $f$ and $g$. Up to rescaling, we could also assume $f$ and $g$ positive. But then $0=\la f,g\ra=\sum f(v)\: g(v)$ provides a contradiction, as every term on the right-hand side is positive. 
\end{proof}
 
One upshot of this theorem is that it provides a tool for recognizing the largest adjacency eigenvalue. This is an intrinsic certificate, in the sense that we do not need to know the other eigenvalues!

 \begin{cor}\label{cor: recog}
 The only adjacency eigenvalue admitting a positive eigenfunction is $\alpha_{\max}$.
 \end{cor}
 
 \begin{proof}
 Distinct eigenvalues have orthogonal eigenfunctions. Two orthogonal functions cannot be both positive.
 \end{proof}

The next corollary concerns the smallest adjacency eigenvalue. We already know that $0\geq \alpha_{\min}\geq -d$, but now we can say a bit more.

\begin{cor}
We have $\alpha_{\min}\geq -\alpha_{\max}$, with equality if and only if the graph is bipartite. \end{cor}

Note that the implication '$\alpha_{\min}= -\alpha_{\max}$ implies bipartite' improves the implication `symmetric adjacency spectrum implies bipartite' we already had.

\begin{proof}[First proof]
We adapt the proof of Theorem~\ref{thm: simple-positive}. Let $f$ be an eigenfunction of $\alpha_{\min}$, and consider $|f|$. Then:
\begin{align*}
\la A|f|,|f|\ra&=\sum_{u\sim v} |f|(u)\: |f|(v) \geq \Big|\sum_{u\sim v} f(u)\: \bar{f}(v)\Big|=\big|\la Af,f\ra\big|=-\alpha_{\min}\la |f|,|f|\ra
\end{align*}
Therefore $\alpha_{\max}\geq -\alpha_{\min}$. Now let us analyze the case of equality. Firstly, $|f|$ has to be an eigenfunction of $\alpha_{\max}$, in particular $f$ never vanishes. Secondly, we have \eqref{eq: quasi-prop} once again. Adjacency being a symmetric relation, we must have $\omega=\pm 1$. If $\omega =1$, then $f$ is a scalar multiple of $|f|$. Hence $f$ is an eigenfunction of $\alpha_{\max}$, a contradiction. Thus $\omega=-1$. This means that $f/|f|$ switches sign across each edge, and we get a bipartition of our graph.
\end{proof}

\begin{proof}[Second proof]
Let $r\geq 1$ be odd. Then 
\begin{align*}
0\leq \mathrm{Tr}\: A^r=\sum \alpha_k^r\leq \alpha_{\min}^r+(n-1)\alpha_{\max}^r,
\end{align*}
so 
\begin{align*}
(n-1)\:\alpha_{\max}^r\geq (-\alpha_{\min})^r.
\end{align*} 
Taking the $r$-th root, and letting $r\to \infty$, we obtain $\alpha_{\max}\geq -\alpha_{\min}$. Assume now that equality holds, $\alpha_{\max}=-\alpha_{\min}$. If the graph is non-bipartite, then the largest adjacency eigenvalue of its bipartite double is not simple--a contradiction. \end{proof}

\begin{notes}  
Part of what we said about the extremal adjacency eigenvalues falls under the scope of Perron - Frobenius theory. The aspect that is relevant for our discussion concerns square matrices $M$ that are non-negative, in the sense that $M(i,j)\geq 0$ for each $i,j$, and irreducible, meaning that for each $i,j$ there is some $k\geq 0$ such that $M^k(i,j)>0$. A symmetric non-negative $n\times n$ matrix $M$ has an associated underlying graph: the vertex set is $\{1,\dots,n\}$, and $i\neq j$ are connected if $M(i,j)> 0$. Loops can be added as well, to account for positive diagonal entries. Irreducibility of $M$ then corresponds to connectivity of the underlying graph. Conversely, consider a connected graph, possibly with loops, in which edges have positive weights. Then the accordingly-weighted adjacency matrix is a symmetric, non-negative, irreducible matrix. Adjacency results established above can be extended as follows. Let $M$ be a symmetric, non-negative, irreducible $n\times n$ matrix. Then (i) the largest eigenvalue $\mu_{\max}$ is simple, and it has a positive eigenvector, and (ii) the smallest eigenvalue satisfies $\mu_{\min}\geq -\mu_{\max}$.
\end{notes}

\bigskip
\subsection*{The average degree}\
\medskip

The average vertex degree of a graph is denoted $d_{\mathrm{ave}}$, and it is given, as the name suggests, by
\begin{align*}
d_{\mathrm{ave}}=\frac{1}{n}\sum_v \mathrm{deg}(v).
\end{align*}
Using this notion, we can sharpen the already known fact that $\alpha_{\max}\leq d$.

\begin{thm}
We have $d_\mathrm{ave}\leq \alpha_{\max}\leq d$, with equality on either side if and only if the graph is regular.
\end{thm}

The above double bound for $\alpha_{\max}$ suggests that it can be interpreted as a spectral notion of degree. 

\begin{proof}
Let $f$ be a positive eigenfunction of $\alpha_{\max}$. Then $\alpha_{\max}\:f(v)=\sum_{u: u\sim v} f(u)$ for each vertex $v$ and, adding all these relations, we get
\begin{align*}
\alpha_{\max}\:\sum f(v)=\sum \deg(v)\: f(v)\leq d\:\sum f(v).
\end{align*}
Hence $\alpha_{\max}\leq d$, with equality if and only if $\deg(v)=d$ for each $v$, i.e., the graph is regular. 

For the constant function $\mathbb{1}$, the Rayleigh ratio is
\begin{align*}
\frac{\la A\mathbb{1},\mathbb{1}\ra}{\la \mathbb{1},\mathbb{1}\ra}=\frac{\sum \mathrm{deg}(v)}{n}=d_{\mathrm{ave}}
\end{align*}
so $d_{\mathrm{ave}}\leq \alpha_{\max}$. If equality holds then $\ct$ is an eigenfunction for $\alpha_{\max}$, and so the graph must be regular.
\end{proof}

The previous theorem provides a useful tool for hearing regularity. The spectral criterion for regularity is that the largest adjacency eigenvalue $\alpha_{\max}$ should equal the average degree. We note that the average degree is determined by the spectrum as
\begin{align*}
d_{\mathrm{ave}}=\frac{1}{n}\sum\alpha_k^2.
\end{align*}
So, given the adjacency spectrum of a graph, it can checked whether the graph is regular by verifying the relation $\sum\alpha_k^2=n\alpha_{\max}$.

Sometimes the list of \emph{distinct} adjacency eigenvalues suffices for hearing regularity. The following result illustrates this idea. Let us recall that an extremal design graph of degree $d$ has adjacency eigenvalues $\pm d$ and $\pm \sqrt{d-1}$.

\begin{thm}
Assume that the distinct adjacency eigenvalues of a graph are $\pm d$ and $\pm \sqrt{d-1}$, where $d\geq 3$. Then the graph is an extremal design graph of degree $d$.
\end{thm}

\begin{proof} 
The extremal adjacency eigenvalues have opposite signs, so the graph is bipartite. We aim to show that the graph is also regular. For then we will be able to conclude by applying Theorem~\ref{thm: eigen-design}. We show regularity by arguing that $d$, which denotes the largest adjacency eigenvalue, equals the average degree $d_{\mathrm{ave}}$.

Since the eigenvalues of $A$ are $\pm d$, $\pm \sqrt{d-1}$, we have $(A^2-d^2 I)(A^2-(d-1)I)=0$ and so $A^4=(d^2+d-1)A^2-d^2(d-1)I$. Taking traces, we get
\begin{align*}
\Tr (A^4)=(d^2+d-1) n d_{\mathrm{ave}}-d^2(d-1)n.
\end{align*}
On the other, we may understand $\Tr (A^4)$ in a combinatorial way. If $u$ is a vertex, then 
\begin{align*}
A^4(u,u)=c_4(u)+\deg(u)^2+\sum_{v:\: v\sim u} \big(\deg(v)-1\big)
\end{align*}
where the first term counts the $4$-cycles starting and ending at $u$, the second term counts the closed paths of the form $u\sim v\sim u\sim v'\sim u$, and the third counts the closed paths of the form $u\sim v\sim w(\neq u)\sim v\sim u$. Therefore:
\begin{align*}
\Tr (A^4)&=c_4+\sum_u \deg(u)^2 +\sum_u \sum_{v:\: v\sim u} \deg(v)-\sum_u \deg(u)\\
&=c_4+2\sum_u \deg(u)^2-\sum_u \deg(u)
\end{align*} 
where $c_4$ denotes the number of based $4$-cycles in the graph. Now $\sum \deg(u)=nd_{\mathrm{ave}}$, and we can safely bound 
\begin{align*}
\sum \deg(u)^2\geq \frac{1}{n}\Big(\sum \deg(u)\Big)^2=nd_{\mathrm{ave}}^2
\end{align*} 
as we expect the graph to be regular. Furthermore, we disregard the $c_4$ term since the girth of our graph should be $6$. In summary, we have the following combinatorial bound:
\begin{align*}
\Tr(A^4)\geq 2nd_{\mathrm{ave}}^2-nd_{\mathrm{ave}}
\end{align*} 
Combining this bound with the previous computation of $\Tr (A^4)$, we deduce
\begin{align*}
(d^2+d-1) d_{\mathrm{ave}}-d^2(d-1)\geq 2d_{\mathrm{ave}}^2-d_{\mathrm{ave}}
\end{align*}
that is
\begin{align*}
(d_{\mathrm{ave}}-d)(d^2-d-2d_{\mathrm{ave}})\geq 0.
\end{align*} 
Here $d\geq d_{\mathrm{ave}}$ and $d^2-d\geq 2d\geq 2d_{\mathrm{ave}}$, as $d\geq 3$. So it must be that $d=d_{\mathrm{ave}}$. 
\end{proof}

\begin{exer}\label{exer: petersen SD} Show that a graph whose adjacency eigenvalues are $-2,1,3$ must be the Petersen graph. True or false: a graph whose laplacian eigenvalues are $0,2,5$ must be the Petersen graph.
\end{exer}

\bigskip
\subsection*{A spectral Tur\'an theorem}\label{sec: turan}\
\medskip
 
The following theorem is a true classic. A whole subfield, that of extremal graph theory, originates from this highly influential result.

\begin{thm}[Tur\'an]\label{thm: turan}
If a graph does not contain the complete graph $K_{k+1}$ as a subgraph, then the number of edges can be bounded as follows:
\begin{align*}
|E|\leq \frac{1}{2}\Big(1-\frac{1}{k}\Big)\:n^2
\end{align*} 
\end{thm}

Recall that $\alpha_{\max}\geq d_{\mathrm{ave}}=2|E|/n$. So the next theorem can be viewed as a spectral enhancement of Tur\'an's theorem. 

\begin{thm}[Wilf]\label{thm: specturan}
If a graph does not contain $K_{k+1}$ as a subgraph, then its largest adjacency eigenvalue satisfies the following upper bound:
\begin{align*}
\alpha_{\max}\leq \Big(1-\frac{1}{k}\Big)\:n
\end{align*} 
\end{thm}

The key ingredient in the proof is the following elegant fact, interesting in its own right. For convenience, we introduce a new graph invariant: the \emph{clique number} $\omega$ the size of the largest complete subgraph. This is an anti-independence number of sorts.

\begin{lem}[Motzkin - Straus]
Let $f$ be a function on a graph, having non-negative values and satisfying $\sum_v f(v)=1$. Then
\begin{align*}
\sum_{u\sim v} f(u)f(v)\leq 1-\frac{1}{\omega}.
\end{align*}
\end{lem}

\begin{proof}
Consider the support graph of $f$, that is to say, the subgraph induced by those vertices where $f$ is non-zero. If the support graph is complete, then it has at most $\omega$ vertices. In this case, we have
\begin{align*}
\la Af,f\ra=\sum_{u\neq v} f(u)f(v)=\Big(\sum_v f(v)\Big)^2-\Big(\sum_v f(v)^2\Big)=1-\sum_v f(v)^2
\end{align*}
and, by Cauchy - Schwarz,
\begin{align*}
\sum_v f(v)^2\geq \frac{1}{\omega}\Big(\sum_v f(v)\Big)^2= \frac{1}{\omega}.
\end{align*}
This establishes the desired bound in the complete case.

Assume now that the support graph of $f$ is not complete. The plan is to `flow' $f$ towards the complete case. Let $x$ and $y$ be distinct, non-adjacent vertices, and modify the weight distribution by transferring the weight of $y$ to $x$. Namely, define $f'$ as follows: $f'(y)=0$, $f'(x)=f(x)+f(y)$, and $f'(z)=f(z)$ for all other vertices. Thus $f'$ also has non-negative entries, and satisfies $\sum_v f'(v)=1$, but it has smaller support. We wish to compare $\la Af',f'\ra$ to $\la Af,f\ra$. To that end, it is convenient to write $f'=f+\e$ where $\e$ is the perturbation given by $\e(x)=f(y)$, $\e(y)=-f(y)$, and $\e=0$ elsewhere. Now
\begin{align*}
\la Af',f'\ra=\la Af,f\ra+\la Af,\e\ra+\la A\e,f\ra+\la A\e,\e\ra=\la Af,f\ra+2\:\la Af,\e\ra
\end{align*}
so:
\begin{align*}
\tfrac{1}{2}\big(\la Af',f'\ra-\la Af,f\ra\big)&=\la Af,\e\ra=\e(x)\: Af(x)+\e(y)\: Af(y)\\
&=f(y)\:\big(Af(x)-Af(y)\big)
\end{align*}
This gives us the rule for deciding which one of the two non-adjacent vertices is the weight receiver: we transfer the weight of $y$ to $x$ when $Af(x)\geq Af(y)$. Under this rule, $\la Af',f'\ra\geq \la Af,f\ra$. This means that, as long as the support graph is non-complete, we may shrink the support without decreasing the quantity we are interesting in. After performing this procedure a finite number of times, we reach a support graph which is complete.\end{proof}

Tur\'an's theorem follows from the lemma, by taking $f$ to be constant (that is, $f(v)=1/n$ for each vertex $v$). The idea of its spectral strengthening is to use a better choice for $f$.

\begin{proof}[Proof of Theorem~\ref{thm: specturan}] Let $f$ be a positive eigenfunction of $\alpha_{\max}$. We have $\omega\leq k$, by hypothesis, so the previous lemma yields
\begin{align*}
\la Af,f \ra=\sum_{u\sim v} f(u)f(v)\leq \Big(1-\frac{1}{k}\Big) \Big(\sum_v f(v)\Big)^2.
\end{align*}
Now $\la Af,f\ra=\alpha_{\max}\:\la f,f\ra$, and 
\begin{align*}
\big(\sum_v f(v)\big)^2\leq n\sum_v f(v)^2=n\:\la f,f\ra.
\end{align*}
The desired bound follows.
\end{proof}

\begin{notes} Tur\'an's theorem (\emph{On an extremal problem in graph theory}, Mat. Fiz. Lapok 1941, in Hungarian; \emph{On the theory of graphs}, Colloquium Math. 1954 in English) is actually a bit sharper, and it says the following: among all graphs with $n$ vertices that do not contain $K_{k+1}$, there is an explicit `Tur\'an graph' $\mathrm{Tu}(n,k)$ having the largest number of edges. The graph $\mathrm{Tu}(n,k)$ is the complete multipartite graph obtained by partitioning $n$ nodes into $k$ subsets as evenly as possible. A concrete formula for the number of edges of $\mathrm{Tu}(n,k)$ can be easily written down, but it is somewhat cumbersome. The weaker bound stated above, $\tfrac{1}{2}(1-1/k)\:n^2$, is convenient and essentially sharp, for it is attained when $k$ divides $n$. The case $k=2$ of Tur\'an's theorem is an early result of Mantel (1906), and it appeared as Exercise~\ref{exer: mantel} in this text. 

Theorem~\ref{thm: specturan} is due to Wilf (\emph{Spectral bounds for the clique and independence numbers of graphs}, J. Combin. Theory Ser. B 1986). The key lemma is a result of Motzkin - Straus (\emph{Maxima for graphs and a new proof of a theorem of Tur\'an}, Canad. J. Math. 1965).
\end{notes}

 \bigskip
\subsection*{Largest laplacian eigenvalue of bipartite graphs}\
\medskip

Consider a bipartite graph. Then its adjacency spectrum is symmetric with respect to $0$, so $\alpha_{\min}$ is a simple eigenvalue. Furthermore, $\alpha_{\min}$ has an eigenfunction that takes positive values on one side of the bipartition, and negative values on the other side of the bipartition. Such a function is said to be \emph{alternating}. These facts have a laplacian dual. 

\begin{thm}
The laplacian eigenvalue $\lambda_{\max}$ is simple, and it has an alternating eigenfunction.
\end{thm}

\begin{proof}
We run a similar argument to the one of Theorem~\ref{thm: simple-positive}. Let $g$ be an eigenfunction for $\lambda_{\max}$. Define $g'$ by $g'=|g|$ on $V_\bl$, respectively $g'=-|g|$ on $V_\wh$. We claim that $g'$ is an alternating eigenfunction of $\lambda_{\max}$. For each edge $\{u,v\}$, we have $|g'(u)-g'(v)|=|g(u)|+|g(v)| \geq |g(u)-g(v)|$. Consequently,
\begin{align*}
\la Lg',g'\ra=\sum_{\{u, v\}\in E} \big|g'(u)-g'(v)\big|^2\geq \sum_{\{u, v\}\in E} \big|g(u)-g(v)\big|^2=\la Lg,g\ra.
\end{align*}
Also,
\begin{align*}
\la Lg,g\ra=\lambda_{\max}\: \la g,g\ra=\lambda_{\max}\:\la g',g'\ra.
\end{align*}
since $|g|=|g'|$. Hence $\la Lg',g'\ra=\lambda_{\max}\: \la g',g'\ra$, and it follows that $g'$ is an eigenfunction for $\lambda_{\max}$. Next, we use the eigenrelation for $g'$ to argue that $g'$ never vanishes. For each vertex $v$ we have
\begin{align*}
\sum_{u: u\sim v} g'(u)=(\deg(v)-\lambda_{\max})\:g'(v).
\end{align*}
In addition, $g'(u)$ does not change sign as $u$ runs over the neighbours of $v$. So, if $g'(v)=0$ then $g'(u)=0$ for each neighbour $u$ of $v$, and then connectivity leads us to the contradiction $g\equiv 0$. 

Moreover, we must have $|g(u)|+|g(v)| = |g(u)-g(v)|$ for every edge $\{u,v\}$. This means that $g(u)$ and $g(v)$ are collinear with, and on opposite sides of, $0$. In other words,
\begin{align*}
\frac{g(u)}{|g(u)|}+\frac{g(v)}{|g(v)|}=0
\end{align*}
for every edge $\{u,v\}$. It follows that, for some non-zero scalar $c$, $g/|g|=c$ on $V_\bl$ and $g/|g|=-c$ on $V_\wh$. Thus $g$ is a scalar multiple of $g'$, an alternating function.

We deduce that the eigenspace of $\lambda_{\max}$ is one-dimensional. The reason is that two orthogonal functions, say $f$ and $g$, cannot be both alternating: in $\la f,g\ra=\sum f(v_\bl)\: g(v_\bl)+\sum f(v_\wh)\: g(v_\wh)$, all terms have one and the same sign.
\end{proof}

The observation that two orthogonal functions cannot be both alternating also provides a way of recognizing $\lambda_{\max}$ without knowing the full laplacian spectrum. The intrinsic certificate is provided, again, by the eigenfunction.

\begin{cor}
The only laplacian eigenvalue with an alternating eigenfunction is $\lambda_{\max}$. 
\end{cor}

\bigskip
\subsection*{Subgraphs}\
\medskip

How do eigenvalues change upon passing to subgraphs? The next result addresses the behaviour of the largest eigenvalues. We will return to this question later on, with results pertaining to the entire spectrum.

\begin{thm}\label{thm: redeem}
Let $X'$ be a proper subgraph of a graph $X$. Then the largest eigenvalues satisfy $\alpha_{\max}'< \alpha_{\max}$, and $\lambda_{\max}'\leq \lambda_{\max}$.
\end{thm}

\begin{proof}
Let $f'$ be a positive eigenfunction for $\alpha'_{\max}$ having positive entries. Extend it to a function $f$ on $X$ by setting $f\equiv 0$ off the vertices of $X'$. Then $\la f',f'\ra=\la f,f\ra$ and
\begin{align*}
\la A'f',f'\ra=2\sum_{\{u,v\}\in E'} f'(u) f'(v)\leq 2\sum_{\{u,v\}\in E} f'(u) f'(v)=\la Af,f\ra.
\end{align*}
Hence
\begin{align*}
\alpha_{\max}'=\frac{\la A'f',f'\ra}{\la f',f'\ra}\leq \frac{\la Af,f\ra}{\la f,f\ra}\leq \alpha_{\max}.
\end{align*}
If $\alpha_{\max}'= \alpha_{\max}$ then equalities hold throughout the above estimate. On the one hand, $f$ is an eigenfunction for $\alpha_{\max}$. In particular, $f$ never vanishes so $V'=V$. On the other hand, it follows that $E'=E$. Therefore $X'$ is $X$ itself, contradicting the properness assumption. 

The argument for the largest laplacian eigenvalue is similar, though easier. Let $g'$ be an eigenfunction for $\lambda'_{\max}$, and extend it to a function $g$ on $X$ by setting $g\equiv 0$ off the vertices of $X'$. Then $\la g',g'\ra=\la g,g\ra$ and
\begin{align*}
\la L'g',g'\ra=\sum_{\{u,v\}\in E'} \big|g'(u)- g'(v)\big|^2\leq \sum_{\{u,v\}\in E} \big|g(u)- g(v)\big|^2=\la Lg,g\ra.
\end{align*}
Taking Rayleigh ratios, it follows that $\lambda_{\max}'\leq \lambda_{\max}$.
\end{proof}

\begin{exer} \label{exer: another lambda_n} (i) Show that $\lambda_{\max}\geq d+1$ and $\alpha_{\max}\geq \sqrt{d}$. Describe, in each case, the graphs for which equality holds. (ii) Which graphs satisfy $\lambda_{\max}\leq 4$?
\end{exer}

The following result, a rather striking classification, combines the adjacency part of Theorem~\ref{thm: redeem} with Corollary~\ref{cor: recog}.

\begin{thm}[Smith]\label{thm: ADE} 
The graphs satisfying $\alpha_{\max}<2$ are the A-D-E graphs of Figure~\ref{fig: ADE}. The graphs satisfying $\alpha_{\max}=2$ are the extended A-D-E graphs of Figure~\ref{fig: extADE}.
\end{thm}

\begin{figure}[!ht]
\centering
\GraphInit[vstyle=Classic]
\tikzset{VertexStyle/.style = {
shape = circle,
inner sep = 0pt,
minimum size = .8ex,
draw}}
\vspace{.6cm}
\flushleft
\qquad\qquad\qquad\qquad $A_n$\qquad\begin{tikzpicture}[scale=.8]
\SetVertexNoLabel
\Vertex{A}\WE(A){G}
\EA(A){B} \EA(B){C} \EA(C){D} \EA(D){E} 
\Edges(G, A,B, C, D)
\Edge[label=$...$](D)(E)
\end{tikzpicture}

\vspace{.5cm}
\qquad\qquad\qquad\qquad $D_n$\qquad\begin{tikzpicture}[scale=.8]
\SetVertexNoLabel
\Vertex{A}\NO(A){G}\WE(A){I}
\EA(A){B} \EA(B){C} \EA(C){D} \EA(D){E} 
\Edges(G, A,B, C, D)
\Edges(I,A)
\Edge[label=$...$](D)(E)
\end{tikzpicture}

\vspace{.5cm}
\qquad\qquad\qquad\qquad $E_6$\qquad\begin{tikzpicture}[scale=.8]
\SetVertexNoLabel
\Vertex{A}\WE(A){G}\NO(A){I}\WE(G){H}
\EA(A){B} \EA(B){C}
\Edges(H, G, A,B, C)
\Edges(I,A)
\end{tikzpicture}

\vspace{.5cm}
\qquad\qquad\qquad\qquad $E_7$\qquad\begin{tikzpicture}[scale=.8]
\SetVertexNoLabel
\Vertex{A}\WE(A){G}\NO(A){I}\WE(G){H}
\EA(A){B} \EA(B){C} \EA(C){D} 
\Edges(H, G, A,B, C, D)
\Edges(I,A)
\end{tikzpicture}

\vspace{.5cm}
\qquad\qquad\qquad\qquad $E_8$\qquad\begin{tikzpicture}[scale=.8]
\SetVertexNoLabel
\Vertex{A}\WE(A){G}\NO(A){I}\WE(G){H}
\EA(A){B} \EA(B){C} \EA(C){D} \EA(D){E}
\Edges(H, G, A,B, C, D, E)
\Edges(I,A)
\end{tikzpicture}

\bigskip
\caption{The A-D-E graphs.}\label{fig: ADE}
\end{figure}
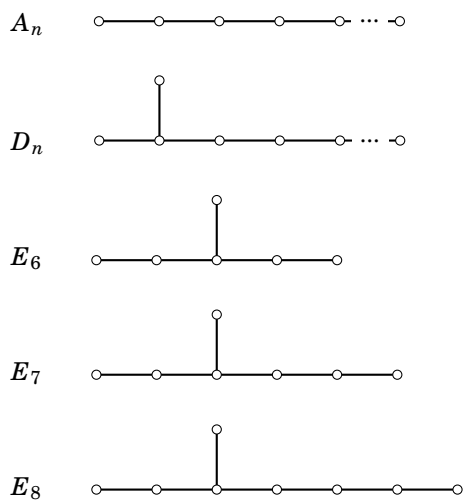

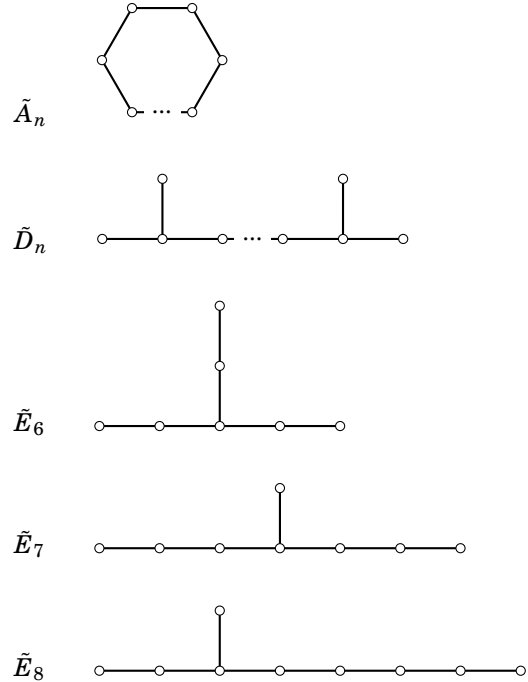
\begin{figure}[!ht]
\centering
\GraphInit[vstyle=Classic]
\tikzset{VertexStyle/.style = {
shape = circle,
inner sep = 0pt,
minimum size = .8ex,
draw}}
\vspace{.9cm}
\flushleft
\qquad\qquad\qquad\qquad $\tilde{A}_n$\qquad\begin{tikzpicture}[scale=.8]
\SetVertexNoLabel
\Vertices{circle}{A,B,C,D,E,F}
\Edges(F, A,B,C,D,E)
\Edge[label=$...$](E)(F)
\end{tikzpicture}

\vspace{.6cm}
\qquad\qquad\qquad\qquad $\tilde{D}_n$\qquad\begin{tikzpicture}[scale=.8]
\SetVertexNoLabel
\Vertex{A}\NO(A){G}\WE(A){I}
\EA(A){B} \EA(B){D}\EA(D){E} \EA(E){M} \NO(E){N}
\Edges(G, A, B)
\Edges(I,A) \Edges(E,M) \Edges(E,N)\Edges(E,D)
\Edge[label=$...$](B)(D)
\end{tikzpicture}

\vspace{.6cm}
\qquad\qquad\qquad\qquad $\tilde{E}_6$\qquad\begin{tikzpicture}[scale=.8]
\SetVertexNoLabel
\Vertex{A}\WE(A){G}\NO(A){I}\NO(I){J}\WE(G){H}
\EA(A){B} \EA(B){C}
\Edges(H, G, A,B, C)
\Edges(J,I,A)
\end{tikzpicture}

\vspace{.6cm}
\qquad\qquad\qquad\qquad $\tilde{E}_7$\qquad\begin{tikzpicture}[scale=.8]
\SetVertexNoLabel
\Vertex{A}\WE(A){$^G$}\NO(A){I}\WE(G){H}\WE(H){J}
\EA(A){B} \EA(B){C} \EA(C){D}
\Edges(J, H, G, A,B, C, D)
\Edges(I,A)
\end{tikzpicture}

\vspace{.6cm}
\qquad\qquad\qquad\qquad $\tilde{E}_8$\qquad\begin{tikzpicture}[scale=.8]
\SetVertexNoLabel
\Vertex{A}\WE(A){G}\NO(A){I}\WE(G){H}
\EA(A){B} \EA(B){C} \EA(C){D} \EA(D){E}\EA(E){J}
\Edges(H, G, A,B, C, D, E, J)
\Edges(I,A)
\end{tikzpicture}

\bigskip
\caption{The extended A-D-E graphs.}\label{fig: extADE}
\end{figure}

\begin{proof}
Firstly, the extended A-D-E graphs have $\alpha_{\max}=2$. A fun exercise, left to the reader, shows that each graph admits $2$ as an adjacency eigenvalue, and the corresponding eigenfunction has positive (and integral) entries. It follows that $2$ is, indeed, the largest adjacency eigenvalue.

Now consider a graph with $\alpha_{\max}<2$. Then none of the extended A-D-E graphs can appear as a subgraph. Our graph contains no cycle $\tilde{A}_n$, so it is a tree. A degenerate instance of $\tilde{D}_n$ is a star with $4$ pendant vertices; using it, we deduce that the maximal vertex degree in our graph is at most $3$. Another use of $\tilde{D}_n$ tells us that at most one vertex has degree $3$. We leave aside the case of a path, $A_n$ in our list, and we assume that there is precisely one trivalent vertex. To figure out the length of the three pending paths, we use the graphs $\tilde{E}_6$, $\tilde{E}_7$, and $\tilde{E}_8$. The first tells us that at least one of the paths has length $1$. Assume that there is precisely one such path, otherwise we are in the $D_n$ case. Now $\tilde{E}_7$ determines that at least one of the remaining two paths has length $2$. The third path can have length $2$, $3$, or $4$, but no larger, because of $\tilde{E}_8$. We have thus obtained the remaining graphs, $E_6$, $E_7$, $E_8$. To verify that the A-D-E graphs do satisfy $\alpha_{\max}<2$, we note that they are proper subgraphs of the extended A-D-E graphs in the way suggested by the notation. \end{proof}

\begin{notes}
Theorem~\ref{thm: ADE} is due to John H. Smith (\emph{Some properties of the spectrum of a graph}, in `Combinatorial Structures and their Applications', Gordon and Breach 1970). The A-D-E graphs and their extended versions are also known as the simple Coxeter - Dynkin diagrams and they occur, quite amazingly, in several classification results for seemingly unrelated objects. A few words about the notations in Figures~\ref{fig: ADE} and ~\ref{fig: extADE}. The number of nodes of a graph is given by the index in Figure~\ref{fig: ADE}, and by the index plus one in Figure~\ref{fig: extADE}. The notation for the A-type, standard in this context, is slightly at odds with some of our notations: $\tilde{A}_n$ is the cycle $C_{n+1}$, while $A_n$ is not an Andr\'asfai graph but a path graph, and should be denoted $P_n$.
\end{notes}

\bigskip
\subsection*{Largest eigenvalues of trees}\label{sec: trees}\
\medskip

Let us start with a fairly general fact. Consider a rooted tree, that is, a tree with a distinguished vertex. A function on the vertices of the tree is \emph{radial} if it is constant on spheres around the root vertex. The rooted tree is \emph{radially regular} if the degree function is radial. Two important examples are the trees $T_{d,R}$ and $\tilde{T}_{d,R}$. Recall that the root vertex has degree $d-1$ in $T_{d,R}$, and degree $d$ in $\tilde{T}_{d,R}$; the pendant vertices lie on a sphere of radius $R$ about the root; the remaining intermediate vertices all have degree $d$.

\begin{figure}[!ht]\bigskip
\centering
\GraphInit[vstyle=Classic]
\tikzset{VertexStyle/.style = {
shape = circle,
inner sep = 0pt,
minimum size = .8ex,
draw}}
\begin{minipage}[b]{0.48\linewidth}
\centering
\begin{tikzpicture}[scale=.7]
\SetVertexNoLabel
\Vertex[x=1.5,y=1]{A}
\Vertex[x=2,y=1]{B}
\Vertex[x=2.5,y=1]{C}
\Vertex[x=3,y=1]{D}
\Vertex[x=7,y=1]{a}
\Vertex[x=7.5,y=1]{b}
\Vertex[x=8,y=1]{c}
\Vertex[x=8.5,y=1]{d}
\Vertex[x=2,y=1.5]{E}
\Vertex[x=3,y=1.5]{F}
\Vertex[x=3,y=2.5]{G}
\Vertex[x=7,y=1.5]{e}
\Vertex[x=8,y=1.5]{f}
\Vertex[x=7,y=2.5]{g}
\Vertex[x=5,y=4.5]{R}
\Edges(R, G, E, A)
\Edges(G, F, D)
\Edges(F, C)
\Edges(E, B)
\Edges(R, g, e, a)
\Edges(g, f, d)
\Edges(f, c)
\Edges(e, b)
\end{tikzpicture}
\end{minipage}
\begin{minipage}[b]{0.48\linewidth}
\centering
\begin{tikzpicture}[scale=.7]
\SetVertexNoLabel
\Vertex[x=1.5,y=1]{A}
\Vertex[x=2,y=1]{B}
\Vertex[x=2.5,y=1]{C}
\Vertex[x=3,y=1]{D}
\Vertex[x=7,y=1]{a}
\Vertex[x=7.5,y=1]{b}
\Vertex[x=8,y=1]{c}
\Vertex[x=8.5,y=1]{d}
\Vertex[x=2,y=1.5]{E}
\Vertex[x=3,y=1.5]{F}
\Vertex[x=3,y=2.5]{G}
\Vertex[x=7,y=1.5]{e}
\Vertex[x=8,y=1.5]{f}
\Vertex[x=7,y=2.5]{g}
\Vertex[x=5,y=4.5]{R}
\Vertex[x=5, y=2.5]{u}
\Vertex[x=4.5, y=1.5]{v}
\Vertex[x=5.5, y=1.5]{w}
\Vertex[x=4.5, y=1]{x}
\Vertex[x=4, y=1]{y}
\Vertex[x=5.5, y=1]{z}
\Vertex[x=6, y=1]{t}
\Edges(R, G, E, A)
\Edges(G, F, D)
\Edges(F, C)
\Edges(E, B)
\Edges(R, g, e, a)
\Edges(g, f, d)
\Edges(f, c)
\Edges(e, b)
\Edges(R, u)
\Edges(v, u)
\Edges(w, u)
\Edges(v, x)
\Edges(w, z)
\Edges(v, y)
\Edges(w, t)
\end{tikzpicture}
\end{minipage}
\caption{$T_{3,3}$ and $\tilde{T}_{3,3}$, once again.}
\end{figure}
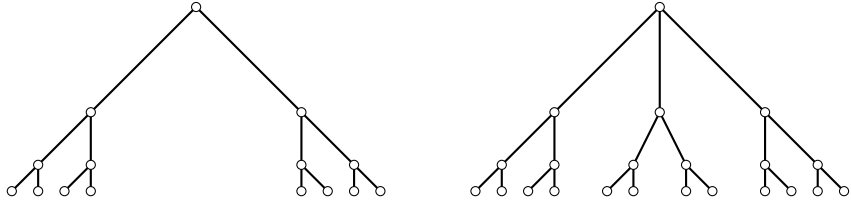

\begin{lem}\label{lem: rad reg}
For a radially regular tree, the eigenfunctions corresponding to the largest eigenvalues are radial.
\end{lem}

\begin{proof}
For a function $f$ defined on the vertices of the tree, we let $f^\#$ denote the radial function obtained by averaging $f$ on spheres about the root vertex. More precisely, 
we define $f^\#$ on $S_r$ by
\begin{align*}
f_r^\#\equiv \frac{1}{|S_r|} \sum_{v\in S_r} f(v)
\end{align*}
where $S_r$, for $r=0,1,\dots,R$, denotes the sphere of radius $r$ with respect to the root. We claim that the adjacency and the laplacian operators commute with averaging:
\begin{align*}
(Af)^\#=A(f^\#), \qquad (Lf)^\#=L(f^\#)
\end{align*}
The claim implies that averaging turns eigenfunctions into radial eigenfunctions, except when averaging yields the zero function. This degeneration cannot happen for the largest eigenvalues, since $\alpha_{\max}$ admits a positive eigenfunction, whereas $\lambda_{\max}$ admits an alternating eigenfunction - say, positive on even-radius spheres and negative on odd-radius spheres.  As the eigenfunctions of $\alpha_{\max}$ and $\lambda_{\max}$ are unique up to scaling, we conclude that they are radial.

Now let us prove the claim. It suffices to check that $A(f^\#)=(Af)^\#$ on each sphere. As the degree function is constant on spheres, the same holds for the laplacian. Note that $A(f^\#)$ is constant on a sphere of radius $r$, namely 
\begin{align*}
A(f^\#)\equiv f_{r-1}^\#+\frac{|S_{r+1}|}{|S_r|}\:f_{r+1}^\#\qquad \textrm{ on } S_r.
\end{align*}
But the right-hand quantity is precisely $(Af)_r^\#$, thanks to the following computation:
\begin{align*}
|S_r| \:(Af)_r^\# &=\sum_{v\in S_r} (Af)(v)= \frac{|S_{r}|}{|S_{r-1}|}\sum_{u\in S_{r-1}} f(u)+ \sum_{w\in S_{r+1}} f(w)\\
&= |S_{r}|\: f_{r-1}^\#+|S_{r+1}|\: f_{r+1}^\#
\end{align*}
A natural convention accommodates the extreme cases, $r=0$ (the root) and $r=R$ (the pendant vertices). The reader who feels that they should be handled separately is invited to do so. \end{proof}

We now focus on the tree $T_{d,R}$, and we determine the eigenvalues that admit a radial eigenfunction.

\begin{lem}\label{lem: rad eigen}
The adjacency eigenvalues of $T_{d,R}$ that admit a radial eigenfunction are
\begin{align*}
\Big\{2\sqrt{d-1}\:\cos \frac{\pi k}{R+2}: k=1,\dots,R+1 \Big\}.
\end{align*}
The laplacian eigenvalues of $T_{d,R}$ that admit a radial eigenfunction are
\begin{align*}
\Big\{0,\: d-2\sqrt{d-1}\:\cos \frac{\pi k}{R+1}: k=1,\dots,R \Big\}.
\end{align*}
\end{lem}

\begin{proof}
Let $\alpha$ be an adjacency eigenvalue admitting a radial eigenfunction $f$. If $f_r$ denotes the value on the sphere of radius $r$, then:
\begin{align*}
\alpha f_0&=(d-1) f_1\\
\alpha f_r&=f_{r-1}+(d-1)f_{r+1}\qquad r=1,\dots,R-1\\
\alpha f_R&=f_{R-1}
\end{align*}

The solution to this recurrence is a linear combination $f_r=ax_1^{r+1}+bx_2^{r+1}$, where $x_1$ and $x_2$ are the roots of the equation $(d-1)x^2-\alpha x+1=0$, subject to the initial and terminal conditions $f_{-1}=0$, $f_{R+1}=0$. Thus $a+b=0$ and $x_1^{R+2}=x_2^{R+2}$.  Since $x_1x_2=1/(d-1)$, we find that $x_1=\zeta/\sqrt{d-1}$ and $x_2=\bar{\zeta}/\sqrt{d-1}$, where $\zeta$ satisfies $\zeta^{2(R+2)}=1$. Then 
\begin{align*}
\alpha=(d-1)(x_1+x_2)=2\sqrt{d-1}\:\mathrm{Re}(\zeta).
\end{align*}

Up to multiplication by a scalar, $f$ is given by 
\begin{align*}
f_r=(d-1)^{\:-r/2}\:\mathrm{Im}(\zeta^{r+1}), \qquad r=0,\dots,R.
\end{align*} 
For $f$ to be non-zero, we have to have $\zeta\neq \pm 1$. Conversely, if $\zeta\neq \pm 1$ then $f$ is non-zero since $f_0\neq0$. We conclude that the adjacency eigenvalues admitting radial eigenfunctions are
\begin{align*}
\big\{2\sqrt{d-1}\:\mathrm{Re}(\zeta): \zeta^{2(R+2)}=1,\: \zeta\neq \pm 1 \big\}.
\end{align*}

Now let $\lambda$ be a laplacian eigenvalue admitting a radial eigenfunction $g$. If $g_r$ denotes, again, the value on the sphere of radius $r$, then:
\begin{align*}
(d-1-\lambda) g_0&=(d-1) g_1\\
(d-\lambda) g_r&=g_{r-1}+(d-1)g_{r+1}\qquad r=1,\dots,R-1\\
(1-\lambda) g_R&=g_{R-1}
\end{align*}

The solution is given by $g_r=ay_1^{r+1}+by_2^{r+1}$, where $y_1$ and $y_2$ are the roots of the equation $(d-1)y^2-(d-\lambda) y+1=0$, subject to the initial and terminal conditions $g_{-1}=g_0$, $g_R=g_{R+1}$. These conditions say that $a(y_1-1)=-b(y_2-1)$ and $a(y_1-1)y_1^{R+1}=-b(y_2-1)y_2^{R+1}$. If $y_1=1$ or $y_2=1$, then the quadratic relation they satisfy yields $\lambda=0$, and $g$ is readily seen to be constant.  If $y_1,y_2\neq 1$ then $y_1^{R+1}=y_2^{R+1}$. Since $y_1y_2=1/(d-1)$, we get $y_1=\xi/\sqrt{d-1}$ and $y_2=\bar{\xi}/\sqrt{d-1}$ where $\xi^{2(R+1)}=1$. The relation $d-\lambda=(d-1)(y_1+y_2)$ yields
\begin{align*}
\lambda=d-2\sqrt{d-1}\:\mathrm{Re}(\xi).
\end{align*}

Up to re-scaling $g$, we have $g_r=(y_2-1)y_1^{r+1}-(y_1-1)y_2^{r+1}=y_1y_2(y_1^{r}-y_2^{r})-(y_1^{r+1}-y_2^{r+1})$. Using again the relation $y_1y_2=1/(d-1)$, and re-scaling $g$ once again, we reach the formula 
\begin{align*}
g_r=(d-1)^{\:-r/2}\:\big(\mathrm{Im}(\xi^{r})-\sqrt{d-1}\:\mathrm{Im}(\xi^{r+1})\big), \qquad r=0,\dots,R.
\end{align*} 
We note that $g$ is non-zero if and only if $\xi\neq \pm 1$. Thus, the laplacian eigenvalues admitting radial eigenfunctions are
\begin{align*}
\big\{0,\: d-2\sqrt{d-1}\:\mathrm{Re}(\xi): \xi^{2(R+1)}=1, \: \xi\neq \pm 1\big\}.
\end{align*}
\end{proof}

The eigenvalues of $T_{d,R}$ that we have found are all simple, and there are certainly other eigenvalues. These are hiding in subtrees. Let us remark, however, that the tree $T_{2,R}$ is the path graph on $R+1$ vertices. In this case eigenfunctions are automatically radial, so we actually have all the eigenvalues. The upshot is that the adjacency, respectively the laplacian eigenvalues of $P_n$, the path graph on $n$ vertices, are
\begin{align*}
\adspec(P_n)&=\Big\{2\:\cos \frac{\pi k}{n+1}: k=1,\dots,n \Big\},\\
\lapspec(P_n)&=\Big\{2-2\:\cos \frac{\pi k}{n}: k=0,\dots,n-1 \Big\}.
\end{align*}

For our purposes, the main consequence of Lemma~\ref{lem: rad eigen} is that it spells out the largest eigenvalues of $T_{d,R}$. Indeed, thanks to Lemma~\ref{lem: rad reg}, we simply have to pick out the largest eigenvalue from each listing.

\begin{thm}\label{ex: fission tree}
The largest eigenvalues of $T_{d,R}$ are as follows:
\begin{align*}
\alpha_{\max}=2\sqrt{d-1}\: \cos \frac{\pi}{R+2},\qquad \lambda_{\max}&=d+2\sqrt{d-1}\: \cos \frac{\pi}{R+1}
\end{align*}
\end{thm}

Using this theorem, we can bound the largest eigenvalues of a tree in terms of the maximal degree and the diameter.

\begin{cor}\label{cor: ext-trees}
The largest eigenvalues of a tree satisfy 
\begin{align*}
\alpha_{\max}\leq 2\sqrt{d-1}\: \cos \frac{\pi}{\delta+2},\qquad \lambda_{\max}\leq d+2\sqrt{d-1}\: \cos \frac{\pi}{\delta+1}.
\end{align*}
In particular, $\alpha_{\max}< 2\sqrt{d-1}$ and $\lambda_{\max}< d+2\sqrt{d-1}$.
\end{cor}

\begin{proof}
Let $X$ be a tree of maximal degree $d$ and diameter $\delta$. Then $X$ can be viewed as a subtree of $T_{d, \delta}$: pick a path of length $\delta$ in $X$, and then hang $X$ by one of the endpoints. Therefore $\alpha_{\max}(X)\leq \alpha_{\max}(T_{d, \delta})$ and $\lambda_{\max}(X)\leq \lambda_{\max}(T_{d, \delta})$. The desired bounds follow from Theorem~\ref{ex: fission tree}.
\end{proof}

For the tree $\tilde{T}_{d,R}$, we could carry out a spectral analysis as for $T_{d,R}$, but the computations are more cumbersome. We will content ourselves with estimates on the largest eigenvalues, obtained by exploiting the relationship with $T_{d,R}$. Namely, $T_{d,R}$ is a subtree of $\tilde{T}_{d,R}$, and $\tilde{T}_{d,R}$ is a subtree of $T_{d,2R}$. From Theorem~\ref{ex: fission tree} we also deduce the following.

\begin{cor}\label{cor: tilde T}
The largest eigenvalues of $\tilde{T}_{d,R}$ satisfy the following bounds:
\begin{align*}
2\sqrt{d-1}\: \cos \frac{\pi}{R+2}\leq &\:\alpha_{\max}\leq 2\sqrt{d-1}\: \cos \frac{\pi}{2R+2}\\
d+2\sqrt{d-1}\: \cos \frac{\pi}{R+1}\leq &\:\lambda_{\max}\leq d+2\sqrt{d-1}\: \cos \frac{\pi}{2R+1}
\end{align*}
\end{cor}

The function $x\mapsto \cos (\pi/x)$, for $x\geq 2$, is non-negative and increasing to $1$ as $x\to \infty$. Elementary calculus shows that $1-\cos (\pi/x)$ lies between two numerical multiples of $1/x^{2}$. In asymptotic notation, $1-\cos (\pi/x)=\Theta(x^{-2})$. So both $\tilde{T}_{d,R}$ and $T_{d,R}$ satisfy
\begin{align*}
\alpha_{\max}&= 2\sqrt{d-1} - 2\sqrt{d-1}\:\Theta(R^{-2}),\\ \lambda_{\max}&= d+2\sqrt{d-1}- 2\sqrt{d-1}\: \Theta(R^{-2}).
\end{align*}

\bigskip
\section{More eigenvalues}
In this section, we take a global view on graph spectra. We consider all eigenvalues, not just the extremal ones.

\bigskip
\subsection*{Eigenvalues of symmetric matrices: Courant - Fischer}\label{sec: symmetric}\
\medskip

We start with the main tool for understanding eigenvalues of real symmetric matrices: the Courant - Fischer minimax formulas. The power and the versatility of these formulas will become apparent in subsequent sections.

Let $M$ be a real symmetric $n\times n$ matrix, with eigenvalues
\begin{align*}
\mu_{\min}=\mu_1\leq \ldots\leq \mu_n=\mu_{\max}.
\end{align*} 
We have seen in Theorem~\ref{thm: unconstrained} that the extremal eigenvalues can be described in terms of the Rayleigh ratio \begin{align*}
R(f)=\frac{\la Mf, f\ra}{\la f,f\ra}
\end{align*} as follows:
\begin{align*}
\mu_{\min}=\min_{f\neq 0}\: R(f), \qquad \mu_{\max}=\max_{f\neq 0}\: R(f) 
\end{align*}

Broadly speaking, the minimax formulas describe each eigenvalue $\mu_k$ via constrained optimizations of the Rayleigh ratio.

\begin{thm}[Courant - Fischer]
The following hold:
\begin{align}
\mu_k=\min_{\dim V=k}\;\max_{f\in V}\: R(f)=\max_{\dim W=n-k+1}\;\min_{f\in W}\: R(f) \label{eq: CF global}
\end{align}
Furthermore, let $f_1, f_2, \dots, f_n$ be an orthonormal basis of eigenvectors, let $V_k$ denote the linear span of $f_1,\dots,f_k$, and let $W_{n-k+1}$ denote the linear span of $f_k,\dots,f_n$. Then:
\begin{align}
\mu_k=\max_{f\in V_k}\: R(f)=\min_{f\in W_{n-k+1}}\: R(f) \label{eq: CF local}
\end{align}
\end{thm}

\begin{proof}
We start by establishing the `local' formulas \eqref{eq: CF local}. If $f\in V_k$ then $f=\sum_{i\leq k} c_i f_i$, so
\begin{align*}
R(f)=\frac{\la Mf, f\ra }{\la f,f\ra}=\frac{\sum_{i\leq k} \mu_i |c_i|^2}{\sum_{i\leq k} |c_i|^2}\leq \mu_k.
\end{align*}
If $f\in W_{n-k+1}$ then $f=\sum_{i\geq k} c_i f_i$, so
\begin{align*}
R(f)=\frac{\la Mf, f\ra }{\la f,f\ra}=\frac{\sum_{i\geq k} \mu_i |c_i|^2}{\sum_{i\geq k} |c_i|^2}\geq \mu_k.
\end{align*}
We conclude by noting that for the eigenvector $f_k$, which is both in $V_k$ and in $W_{n-k+1}$, we have $R(f_k)=\mu_k$.

Next, we establish the `global' formulas \eqref{eq: CF global}. We clearly have
\begin{align*}
\mu_k=\max_{f\in V_k}\: R(f)\geq \min_{\dim V=k}\;\max_{f\in V}\: R(f).
\end{align*}
For the reverse inequality, we need to show that $\mu_k\leq \max_{f\in V}\: R(f)$ for each subspace $V$ of dimension $k$. We use the following observation: if two subspaces $V, W\subseteq \C^n$ satisfy $\dim V+\dim W>n$, then $V$ and $W$ share a non-zero vector. Hence a $k$-dimensional space $V$ and the $(n-k+1)$-dimensional space $W_{n-k+1}$ share a non-zero vector $f_0$. We then have
\begin{align*}
\mu_k=\min_{f\in W_{n-k+1}}\: R(f)\leq R(f_0)\leq \max_{f\in V}\: R(f)
\end{align*}
as desired. The second minimax formula is established by a similar argument. \end{proof}

The min-max formulas for the intermediate eigenvalues are rather unwieldy, even for adjacency and laplacian matrices. One refreshing and useful exception is a fairly transparent formula for $\lambda_2$, the smallest non-trivial laplacian eigenvalue. The constant function $\mathbb{1}$ is an eigenfunction for the simple eigenvalue $\lambda_1=0$, and the space $W_{n-1}$ from the Courant - Fischer theorem is then the orthogonal complement of the subspace spanned by $\mathbb{1}$. We get:
\begin{align}
\lambda_2= \min_{0\neq f\perp \mathbb{1}} \frac{\la L f, f\ra }{\la f, f\ra }=\min\Bigg\{\frac{\sum_{\{u,v\}\in E} \big|f(u)-f(v)\big|^2}{\sum_{u\in V} |f(u)|^2}: \sum_{u\in V} f(u)=0, \; f\neq 0\Bigg\}\label{eq: lambda2}
\end{align}

\begin{exer}\label{exer: second laplacian} Show that a graph which is not complete has $\lambda_2\leq d$ and $\alpha_2\geq 0$.
\end{exer}

\bigskip
\subsection*{A bound for the laplacian eigenvalues}\
\medskip

The following result is a bound on the growth of laplacian eigenvalues, formulated in terms of the maximal degree and the diameter. Two ingredients in the proof are the minimax formulas, and our discussion around the largest eigenvalues of trees. 

\begin{thm}\label{thm: laplacian growth}
Let $1\leq k \leq  \delta/2$. Then:
\begin{align*}
\lambda_{k+1}\leq d-2\sqrt{d-1}\: \cos\frac{2\pi k}{\delta}
\end{align*}
\end{thm}

\begin{proof} We break down the proof into three steps. In the first one, we construct a sample function which is supported on a ball, and whose laplacian Rayleigh ratio is controlled in terms of the radius. In the second step, we argue that there are $k+1$ balls which are well-separated. In the last step, we consider global functions on the graph defined by taking scaled sample functions on each of the $k+1$ balls.

\smallskip
\emph{Step 1.} Let $B_R$ be a ball of radius $R\geq 0$ around a vertex $v$. We show that there is a non-zero function $g$ supported on $B_R$ such that
\begin{align*}
\frac{\la Lg,g\ra}{\la g,g \ra}\leq d-2\sqrt{d-1}\: \cos\frac{\pi}{R+2}.\tag{$*$}
\end{align*}

If $R=0$, take $g$ to be $1$ at $v$ and $0$ elsewhere. Then the left-hand side of $(*)$ is $\mathrm{deg}(v)$ while the right-hand side is $d$.

Assume $R\geq 1$. Consider the tree $\tilde{T}_{d,R}$, and let $\alpha_T$ denote its largest adjacency eigenvalue. Let $f$ be a positive eigenfunction on $\tilde{T}_{d,R}$ for $\alpha_T$. Since $f$ is radial, we may view it as a function of the radius $r\mapsto f_r$, for $r=0,\dots,R$. Spelling out the eigenrelation for $f$, we get 
\begin{align*}
\alpha_T  f_0=d f_1, \qquad \alpha_T f_r=f_{r-1}+(d-1)\:f_{r+1}\quad \textrm{ for } r=1,\dots,R
\end{align*} 
with the convention that $f_{R+1}=0$. These relations, and the fact that $\alpha_T<d$, imply that $f$ is decreasing. Indeed, the first relation yields $f_0>f_1$, while the second relation yields that $f_r>f_{r+1}$ provided $f_{r-1}>f_r$.

Next, we define a function $g$ by transplanting $f$ onto the ball $B_R$. Namely, foliate $B_R$ into spheres $S_r$ for $r=0,\dots,R$, and set $g\equiv f_r$ on $S_r$. Set also $g\equiv 0$ outside $B_R$. We claim that the pointwise bound
\begin{align*}
Lg\leq (d-\alpha_T)g
\end{align*}
holds on $B_R$. Indeed, let $u\in S_r$. We argue the generic case $r>0$, leaving the case $r=0$ (that is, $u=v$) to the reader. As $u$ has at least one neighbour on $S_{r-1}$, and $f$ is decreasing in $r$, we have:
\begin{align*}
Ag(u)&\geq f_{r-1}+(\mathrm{deg}(u)-1)f_{r+1}=\alpha_T f_r-(d-\mathrm{deg}(u))f_{r+1}\\
&\geq \alpha_Tf_r-(d-\mathrm{deg}(u))f_{r}=(\alpha_T-d+\mathrm{deg}(u))g(u)
\end{align*}
Hence
\begin{align*}
Lg(u)= \mathrm{deg}(u) g(u)-Ag(u)\leq  (d-\alpha_T) g(u)
\end{align*}
as claimed. As $g$ is positive on $B_R$ and vanishes off $B_R$, the pointwise bound yields $\la Lg,g\ra\leq (d-\alpha_T) \la g,g\ra$. Now $(*)$ follows, by using the lower bound on $\alpha_T$ from Corollary~\ref{cor: tilde T}.

\smallskip
\emph{Step 2.} Let $R= \big\lfloor \frac{1}{2k}\delta\big\rfloor-1$, and note that $R\geq 0$. We show that there are $k+1$ balls of radius $R$ in $X$, $B_R(v_i)$ for $i=1,\dots,k+1$, which are completely separated: they are disjoint, and there are no edges between them.

Indeed, let $v_1$ and $v_{k+1}$ be vertices with $\mathrm{dist}(v_1,v_{k+1})=\delta$, and pick a path of length $\delta$ between them. As $2(R+1)\leq \frac{1}{k}\delta$, we may successively pick vertices $v_2,\dots,v_{k}$ on the path such that the distance between two consecutive vertices is at least $2(R+1)$. In fact, the distance between any two distinct vertices is at least $2(R+1)$. It follows that any two balls $B_R(v_i)$ and $B_R(v_j)$, where $i\neq j$, are completely separated: if $p_i\in B_R(v_i)$ and $p_j\in B_R(v_j)$ then $2(R+1)\leq \mathrm{dist}(v_i,v_j)\leq 2R+\mathrm{dist}(p_i,p_j)$ so $\mathrm{dist}(p_i,p_j)\geq 2$.

\smallskip
\emph{Step 3.} For each ball $B_R(v_i)$, let $g_i$ be a function supported on $B_R(v_i)$ as in Step $1$. Then $g_1,\dots,g_{k+1}$ are mutually orthogonal since the balls $B_R(v_1), \dots, B_R(v_{k+1})$ are completely separated. Furthermore, $Lg_i$ is supported on $B_{R+1}(v_i)$, hence orthogonal to $g_j$ for each $j\neq i$. For a non-zero linear combination $\phi=\sum c_i g_i$ we have
\begin{align*}
\la L\phi,\phi\ra=\sum |c_i|^2 \la Lg_i,g_i\ra, \qquad \la \phi,\phi\ra=\sum |c_i|^2 \la g_i,g_i\ra
\end{align*}
so
\begin{align*}
\frac{\la L\phi,\phi\ra}{\la \phi,\phi \ra}\leq d-2\sqrt{d-1}\: \cos\frac{\pi}{R+2}\leq d-2\sqrt{d-1}\: \cos\frac{2\pi k}{\delta}
\end{align*}
as $R+2= \big\lfloor \frac{1}{2k}\delta\big\rfloor+1>\frac{1}{2k}\delta$. On the other hand, by the minimax formulas we have that
\begin{align*}
\lambda_{k+1}\leq \max_{0\neq \phi\in V} \frac{\la L\phi,\phi\ra}{\la \phi,\phi \ra}
\end{align*}
for the $(k+1)$-dimensional space $V$ spanned by $g_1,\dots,g_{k+1}$. We conclude the proof by combining the last two inequalites. 
\end{proof}

For regular graphs, we immediately deduce a quantitative bound on the adjacency eigenvalues:

\begin{cor}\label{cor: abbs}
Consider a $d$-regular graph, and $1\leq k \leq  \delta/2$. Then:
\begin{align*}
\alpha_{k+1}\geq 2\sqrt{d-1}\: \cos\frac{2\pi k}{\delta}
\end{align*}
\end{cor}

This has, in turn, an interesting qualitative consequence. The diameter of a $d$-regular graph tends to infinity as the size grows to infinity, and then the above lower bound tends to $2\sqrt{d-1}$. We obtain:

\begin{cor} Fix a threshold $t<2\sqrt{d-1}$. Then $\alpha_2> t$ for all but finitely many $d$-regular graphs. More generally, for each $k\geq 2$ we have that $\alpha_k> t$ for all but finitely many $d$-regular graphs.
\end{cor}

\begin{notes}
Naturally, it all happened backwards. First came the last corollary. The part concerning $\alpha_2$ is stated in a paper of Alon (\emph{Eigenvalues and expanders}, Combinatorica 1986) as a result due Alon and Boppana. 
Lubotzky, Phillips, and Sarnak (\emph{Ramanujan graphs}, Combinatorica 1988) proved a weak version of the Alon - Boppana result. The first published proof, and a quantitative one, of the Alon - Boppana result was subsequently given by Alon's alter ego Nilli (\emph{On the second eigenvalue of a graph}, Discrete Math. 1991). Meanwhile, Burger established the general version in unpublished work (\emph{Cheng's inequality for graphs}, preprint 1987). It has been pointed out by Friedman that the Alon - Boppana phenomenon - namely, escape of eigenvalues towards $2\sqrt{d-1}$ and beyond - is implicit in work of McKay (\emph{The expected eigenvalue distribution of a large regular graph}, Linear Algebra Appl. 1981). 

Corollary~\ref{cor: abbs} came after. The earliest reference is Friedman (\emph{Some geometric aspects of graphs and their eigenfunctions}, Duke Math. J. 1993), followed by Quenell (\emph{Eigenvalue comparisons in graph theory}, Pacific J. Math. 1996) with an independent account. See also Nilli (\emph{Tight estimates for eigenvalues of regular graphs}, Electron. J. Combin. 2004) for a short argument. With some hindsight, the corollary can be already glimpsed in a paper by Brooks (\emph{The spectral geometry of $k$-regular graphs}, J. Anal. Math. 1991)

Theorem~\ref{thm: laplacian growth} is essentially due to Urakawa (\emph{Eigenvalue comparison theorems of the discrete Laplacians for a graph}, Geom. Dedicata 1999). Our approach is somewhat different, and somewhat simpler. Theorem~\ref{thm: laplacian growth} is the graph-theoretic analogue of a bound due to Cheng (\emph{Eigenvalue comparison theorems and its geometric applications}, Math. Z. 1975) for the laplacian eigenvalues on a compact Riemannian manifold. 
\end{notes}

\bigskip
\subsection*{Eigenvalues of symmetric matrices: Cauchy and Weyl} \label{sec: variational}\
\medskip

We now turn to the behaviour of eigenvalues with respect to several operations on symmetric matrices. We discuss two general principles, Cauchy's interlacing theorem and Weyl's inequality. The first one addresses the operation of deleting a row and the corresponding column. The second one concerns addition. Both results are applications of the Courant - Fischer minimax formulas.

\begin{thm}[Cauchy]
Let $M'$ be the symmetric $(n-1)\times (n-1)$ matrix obtained by deleting the first row and the first column of $M$. Then the eigenvalues of $M'$ interlace the eigenvalues of $M$:
\begin{align*}
\mu_1\leq \mu_1'\leq \mu_2\leq \ldots\leq\mu_{n-1}\leq\mu_{n-1}'\leq \mu_n
\end{align*}
\end{thm}
\begin{proof}
Let $\phi$ denote the linear embedding $\C^{n-1}\into\C^n$ given by adding $0$ as first coordinate. Then the Rayleigh ratios of $M'$ and $M$ are related by the formula $R'(f')=R(\phi(f'))$.

We have to show that $\mu_k\leq \mu'_k\leq \mu_{k+1}$ for $k=1,\dots,n-1$. Now
\begin{align*}
\mu'_k&=\min_{\dim V'=k}\;\max_{f'\in V'} \: R'(f')\\
&=\min_{\dim V'=k}\;\max_{f\in \phi(V')}\: R(f)\geq \min_{\dim V=k}\;\max_{f\in V}\: R(f)=\mu_k
\end{align*}
and
\begin{align*}
\mu'_k&=\max_{\dim W'=n-k}\;\min_{f'\in W'}\: R'(f')\\
&=\max_{\dim W'=n-k}\;\min_{f\in \phi(W')}\: R(f)\leq \max_{\dim W=n-k}\;\min_{f\in W}\: R(f)=\mu_{k+1}
\end{align*}
as desired.
\end{proof}

\begin{thm}[Weyl]
Let $M$ and $N$ be symmetric $n\times n$ matrices. Then
\begin{align*}
\mu_{k+\ell-1}(M+N)\geq \mu_k(M)+\mu_\ell(N)
\end{align*}
as long as the indices satisfy $1\leq k,\ell, k+\ell-1\leq n$.
\end{thm}

\begin{proof}
If three subspaces $U, V, W\subseteq \C^n$ satisfy $\dim U+\dim V+\dim W>2n$, then $U$, $V$ and $W$ share a non-zero vector. The three subspaces we will use are: the subspace $W_{n-k+1}$ with respect to $M$, the subspace $W_{n-\ell+1}$ with respect to $N$, and the subspace $V_{k+\ell-1}$ with respect to $M+N$. For a common non-zero vector $f_0$, we have
\begin{align*}
\mu_{k+\ell-1}(M+N)\geq R_{M+N}(f_0)=R_M(f_0)+R_N(f_0)\geq\mu_k(M)+\mu_\ell(N)
\end{align*}
as claimed.
\end{proof}

Weyl's inequality provides lower and upper bounds for the eigenvalues of $M+N$, thought of as a perturbation of $M$ by $N$. Let us single out the simplest bounds of this kind.

\begin{cor}
Let $M$ and $N$ be symmetric $n\times n$ matrices. Then:
\begin{align}
\mu_k(M)+\mu_{\min}(N)\leq &\:\mu_k(M+N)\leq \mu_k(M)+\mu_{\max}(N)\label{eq: weyl1}
\end{align}
\end{cor}

\begin{proof} The lower bound $\mu_k(M)+\mu_{\min}(N)\leq \mu_k(M+N)$ is the case $\ell=1$ of Weyl's inequality. Replacing $M$ by $M+N$, and $N$ by $-N$, we get $\mu_k(M+N)+\mu_{\min}(-N)\leq \mu_k(M)$. As $\mu_{\min}(-N)=-\mu_{\max}(N)$, this is precisely the upper bound $\mu_k(M+N)\leq \mu_k(M)+\mu_{\max}(N)$.
\end{proof}

Here is a simple application to graph eigenvalues. Recall that, for a $d$-regular graph, the adjacency and the laplacian eigenvalues are related by $\alpha_k+\lambda_k=d$. In general, $\alpha_k+\lambda_k$ is confined to the interval determined by the minimal and the maximal vertex degrees.

\begin{prop}\label{prop: niki} 
We have $d_{\min}\leq \alpha_k+\lambda_k\leq d$.
\end{prop}

\begin{proof}
We apply the inequality \eqref{eq: weyl1} to $M=-A$, and $N=\mathrm{diag}(\mathrm{deg})$, the diagonal matrix recording the degrees. Thus $M+N=L$. The relevant eigenvalues are: $\mu_k(M)=-\alpha_k$, $\mu_k(M+N)=\lambda_k$, and $\mu_{\min}(N)=d_{\min}$, $\mu_{\max}(N)=d$. Thus $-\alpha_k+d_{\min}\leq \lambda_k\leq -\alpha_k+d$, as desired.
\end{proof}

\begin{exer}\label{exer: aronszajn} Let $M$ be a symmetric matrix of size $n$, partitioned as
\begin{align*}
M=\begin{pmatrix}
M' & N\\
N^t & M''
\end{pmatrix}
\end{align*} 
where $M'$ and $M''$ are square, and hence symmetric, matrices of size $n'$ respectively $n''$. Show that $\mu_1+\mu_{k+\ell}\leq \mu'_k+\mu''_\ell$ for $1\leq k\leq n'$, $1\leq \ell\leq n''$.
\end{exer}

\begin{notes} Proposition~\ref{prop: niki} is an observation of Nikiforov (\emph{Eigenvalues and extremal degrees of graphs}, Linear Algebra Appl. 2006). Exercise~\ref{exer: aronszajn} is a finite-dimensional instance of a result due to Aronszajn (\emph{Rayleigh-Ritz and A. Weinstein methods for approximation of eigenvalues I. Operators in a Hilbert space}, Proc. Nat. Acad. Sci. U. S. A. 1948). 
\end{notes}

\bigskip
\subsection*{Subgraphs} \label{sec: subgraphs}\
\medskip

How do the eigenvalues of a graph change upon passing to a subgraph? For the largest adjacency or laplacian eigenvalues, we have addressed this question in Theorem~\ref{thm: redeem}. We can now extend our scope to the entire spectrum, thanks to the Cauchy and Weyl inequalities. In this section, we dispense with the usual connectivity assumption. We maintain, however, the non-degeneracy assumption that $n\geq 3$.

Starting with an ambient graph, one gets a subgraph by two operations: removing vertices, and removing edges. We will take a close look at each one, proving in particular the following:

\begin{thm}\label{thm: easy subgraph}
If $X'$ is a spanning subgraph of $X$, then the laplacian eigenvalues satisfy $\lambda_k'\leq \lambda_k$. If $X'$ is an induced subgraph of $X$, then the adjacency eigenvalues satisfy $\alpha_k'\leq \alpha_k$.
\end{thm}

A spanning subgraph is a subgraph having all the vertices of the ambient graph. An induced subgraph is a subgraph having all the edges inherited from the ambient graph. The above theorem conforms with the philosophy that laplacian eigenvalues sense the edges, while adjacency eigenvalues sense the vertices.

The next two theorems give bounds for both the adjacency and the laplacian eigenvalues, when either a vertex or an edge is removed from the ambient graph. 

\begin{thm}
Let $X'$ be a graph obtained by removing a vertex from the graph $X$. Then its adjacency and its laplacian eigenvalues can be bounded as follows: 
\begin{align}
\alpha_{k+1}\leq &\:  \alpha_k'\leq \alpha_k\label{eq: ver-ad}
\end{align}
respectively
\begin{align}
\lambda_k-1\leq \lambda_k'\leq \lambda_{k+1}\label{eq: ver-lap}
\end{align}
\end{thm}

\begin{proof}
Label the vertices by $1,2,\dots, n$ so that the removed vertex is $1$. The adjacency matrix $A'$ is obtained by deleting the first row and column of the adjacency matrix $A$. Cauchy's interlacing theorem is directly applicable, yielding \eqref{eq: ver-ad}.

The laplacian matrix $L'$ is given by $L'=M+N$, where $M$ denotes the matrix obtained by deleting the first row and column of the laplacian matrix $L$, and $N$ is the following matrix:
\begin{align*}
\begin{pmatrix}
-I_{\mathrm{deg}(1)} &  \\
   &  0_{n-1-\mathrm{deg}(1)}  
\end{pmatrix}
\end{align*}
Here $\mathrm{deg}(1)$ denotes, as usual, the degree of the vertex $1$, and we furthermore assume that the neighbours of $1$ come right after ($2$, $3$ and so on). The smallest eigenvalue of $N$ is $-1$. The largest is $0$, except when $\mathrm{deg}(1)=n-1$ in which case it is $-1$. By \eqref{eq: weyl1}, we have $\mu_k(M)-1\leq \lambda'_k\leq \mu_k(M)$ for $k=1,\dots,n-1$. On the other hand, Cauchy's interlacing theorem says that $\lambda_k\leq \mu_k(M)\leq \lambda_{k+1}$ for $k=1,\dots,n-1$. The two-sided bounds \eqref{eq: ver-lap} follow. 
\end{proof}

\begin{thm}
Let $X'$ be a graph obtained by removing an edge from the graph $X$. Then its adjacency and its laplacian eigenvalues can be bounded as follows: 
\begin{align}
\alpha_k-1\leq &\:  \alpha_k'\leq \alpha_k+1\label{eq: edge-ad1}
\end{align}
respectively
\begin{align}
\lambda_k-2\leq &\:  \lambda_k'\leq \lambda_k\label{eq: edge-lap1}
\end{align}
\end{thm}

\begin{proof} Label the vertices by $1,2,\dots, n$ so that the removed edge is the one between $1$ and $2$. The adjacency matrix of $X'$ and that of $X$ are related by $A'=A+N$, where $A$ is the adjacency matrix of $X$, and $N$ is the following matrix:
\begin{align*}
\begin{pmatrix}
0 & -1 &  \\
-1 & 0 &     \\
 &  &  0_{n-2}  
\end{pmatrix}
\end{align*}
This matrix has simple eigenvalues $\pm 1$, as well as $0$ with multiplicity $n-2$. The bound \eqref{eq: weyl1}, with $n-k$ in place of $k$, yields \eqref{eq: edge-ad1}.

The laplacian matrix of $X'$ and that of $X$ are related by $L'=L+N$, where $N$ is the following matrix:
\begin{align*}
\begin{pmatrix}
-1 & -1 &  \\
-1 & -1 &     \\
 &  &  0_{n-2}  
\end{pmatrix}
\end{align*}
The eigenvalues of this matrix are $-2$ with multiplicity $1$, and $0$ with multiplicity $n-1$.  Now \eqref{eq: weyl1} gives \eqref{eq: edge-lap1}.
\end{proof}

Revisiting the proof, we see that the adjacency part of the edge-removal theorem can be generalized.

\begin{thm}
Let $Y$ be a subgraph of $X$, and let $X'$ be the graph obtained from $X$ by removing the edges of $Y$. Then the adjacency eigenvalues of $X'$ satisfy the following bounds: 
\begin{align}
\alpha_k-\alpha_{\max}(Y)\leq &\:  \alpha_k'\leq \alpha_k-\alpha_{\min}(Y)\label{eq: edge-adH}
\end{align}
\end{thm}

\begin{ex} A graph is said to be \emph{hamiltonian} if it contains a spanning cycle, meaning a cycle that visits every vertex. Graphically, this means that the graph can be drawn as a cycle with additional chords. 

The Petersen graph is not hamiltonian. This can be argued combinatorially, but a spectral proof fits the discussion. Let us assume that the Petersen graph has a hamiltonian cycle. Such a cycle $C_{10}$ is obtained by removing $5$ edges from the Petersen graph, so we will rely on the edge-removal theorem. Specifically, we could use the left-hand bound in \eqref{eq: edge-lap1}, saying that $\lambda_k(C_{10})\leq \lambda_k(\mathrm{Pet})$ for all $k=1,\dots,10$. Even better, we can use \eqref{eq: edge-adH}. In this case, $Y$ consists of $5$ disjoint copies of $K_2$, so: 
\begin{align*}
\alpha_k(\mathrm{Pet})-1\leq  \alpha_k(C_{10})\leq \alpha_k(\mathrm{Pet})+1, \qquad k=1,\dots,10\tag{$\dagger$}
\end{align*} 
Note that the right-hand bound of ($\dagger$) is the laplacian bound mentioned before.

Listing the adjacency eigenvalues of $C_{10}$ and those of the Petersen graph, we find that ($\dagger$) fails for $k=6$ and $k=7$.

\begin{align*}
\begin{array}{rccccccccccc}
k \quad \quad &  10  &   9   & 8 &  7 & 6 & 5 &  4 & 3 &  2 &  1\\[.4em]\cline{2-11}\\
\alpha_k(C_{10})\quad \quad & -2 & -1.6..  & -1.6.. & -0.6.. & -0.6.. & 0.6.. & 0.6.. & 1.6.. & 1.6.. & 2\\[.4em]
\alpha_k(\mathrm{Pet})\quad  \quad & -2 & -2  & -2 & -2 & 1 & 1 & 1 & 1 & 1 & 3
\end{array}
\end{align*}
\end{ex}

\smallskip

\begin{exer}\label{exer: pendant} Show that a tree has
\begin{align*}
\lambda_2\leq 2-2\cos \frac{\pi}{\delta +1}.
\end{align*}
Deduce that $\lambda_2\leq 1$ with equality if and only if the tree is a star graph.
\end{exer}

\begin{notes}
A spectral proof of the fact that the Petersen graph is not hamiltonian was first given by Mohar (\emph{A domain monotonicity theorem for graphs and Hamiltonicity}, Discrete Appl. Math. 1992). A simpler argument, essentially the one explained above, was noticed by van den Heuvel (\emph{Hamilton cycles and eigenvalues of graphs}, Linear Algebra Appl. 1995).

The result of Exercise~\ref{exer: pendant} is due to Grone, Merris, and Sunder (\emph{The Laplacian spectrum of a graph}, SIAM J. Matrix Anal. Appl. 1990).
\end{notes}

\bigskip
\section{Spectral bounds}

The focus of this section is on bounds relating graph eigenvalues to graph invariants and other combinatorial properties.

\bigskip
\subsection*{Chromatic number and independence number}\label{sec: chromatic}\
\medskip

 We start with one of the earliest results in spectral graph theory. It gives a spectral improvement of the upper bound $\chi\leq 1+d$.
 
\begin{thm}[Wilf] \label{thm: wilf}We have $\chi\leq 1+\alpha_{\max}$.
\end{thm}

\begin{proof} 
We argue as in the proof of Theorem~\ref{prop: chromatic bound}, inducting on the size. Let $X'$ be the (possibly disconnected) graph obtained by removing a vertex $v$ of minimal degree. The induction hypothesis, and the inequality $\alpha_{\max}'\leq \alpha_{\max}$, imply that $X'$ can be coloured using at most $1+ \alpha_{\max}$ colours. Since $\deg(v)\leq d_\mathrm{ave}\leq \alpha_{\max}$, there is at least one more colour available to complete the colouring.
\end{proof}

One may wonder whether the even better bound $\chi\leq 1+d_\mathrm{ave}$ holds. The above proof does not work since $d_\mathrm{ave}$, unlike $d$ and $\alpha_{\max}$, may increase by passing to subgraphs. And a counterexample is not hard to find: the graph obtained from $K_n$ by attaching a pendant edge at each vertex has $\chi=n$ and $d_\mathrm{ave}=\frac{1}{2}(n+1)$.

The next result gives a lower bound for the chromatic number in terms of adjacency eigenvalues. Recall that non-trivial graphs have $\alpha_{\min}<0$.

\begin{thm}[Hoffman]\label{thm: hoffman2}
We have: 
\begin{align*}
\chi\geq 1+\frac{\alpha_{\max}}{-\alpha_{\min}}
\end{align*}
\end{thm}

\begin{proof}
Let $f$ be a positive eigenfunction for $\alpha_{\max}$. Given an independent set of vertices, $S$, we consider the perturbation $f_S=f+af\cdot \ct_S$ where $\ct_S$ is the characteristic function of $S$, and $a\in \R$ will be chosen later. On the one hand, we have
\begin{align*}
\la f_S,f_S\ra&=\la f,f\ra+2a\: \la f,f\cdot \ct_S\ra+a^2\:\la f\cdot \ct_S,f\cdot \ct_S\ra\\
&=\la f,f\ra+(a^2+2a)\: \la f,f\cdot \ct_S\ra.
\end{align*}
On the other hand,
  \begin{align*}
\la Af_S,f_S\ra&=\la Af,f\ra+2a\: \la Af,f\cdot \ct_S\ra+a^2\:\la A(f\cdot \ct_S),f\cdot \ct_S\ra\\
&=\la Af,f\ra+2a\: \la Af,f\cdot \ct_S\ra
\end{align*}
by using the symmetry of $A$, and the independence of $S$.

Now consider a partition $V=S_1\cup\dots\cup S_\chi$ into $\chi$ independent subsets, and let $f_k$ be the corresponding modifications of $f$ on each $S_k$, as above. Adding up the above relations, we get:
\begin{align*}
\sum \la f_k,f_k\ra&= (\chi+a^2+2a)\:\la f,f\ra\\
 \sum \la Af_k,f_k\ra&= (\chi+2a)\:\la Af,f\ra=(\chi+2a)\:\alpha_{\max}\:\la f,f\ra
\end{align*}
As $\la Af_k,f_k\ra\geq \alpha_{\min}\: \la f_k,f_k\ra$ for each $k$, we are led to the following inequality:
\begin{align*}
(\chi+2a)\: \alpha_{\max}\geq (\chi+a^2+2a)\:\alpha_{\min}.
\end{align*}
Taking now $a=-\chi$, we obtain $\alpha_{\max}\leq -(\chi-1)\:\alpha_{\min}$. This is the desired inequality. 
\end{proof}

\begin{exer}\label{exer: hoffman} 
Give an alternate proof of the previous theorem by using the result of Exercise~\ref{exer: aronszajn}.
\end{exer}

We now turn to the independence number, where we get an upper bound in laplacian terms.

\begin{thm}[Hoffman]\label{thm: hoffman1}
We have: 
\begin{align*}
\iota\leq n\:\bigg(1-\frac{d_{\mathrm{min}}}{\lambda_{\max}}\bigg)
\end{align*}
\end{thm}

\begin{proof}
Let $S$ be a proper subset of vertices, and consider the function $f$ defined by $f\equiv |S^c|$ on $S$, respectively $f\equiv -|S|$ on $S^c$. We compute
\begin{align*}
\la f, f\ra &=\sum_{u\in S} f(u)^2+\sum_{u\in S^c} f(u)^2 =|S|\cdot |S^c|^2+|S^c|\cdot |S|^2= n |S||S^c|\\
\la L f, f\ra& =\sum_{\{u,v\}\in E}|f(u)-f(v)|^2 =\sum_{\{u,v\}\in \bd S}|f(u)-f(v)|^2 = n^2 |\bd S|
\end{align*}
and so:
\begin{align}
\frac{\la L f, f\ra}{\la f, f\ra}=\frac{n|\bd S|}{|S||S^c|}\label{eq: useful}
\end{align}
Now, if $S$ is an independent subset, then all edges emanating from vertices of $S$ are in $\bd S$, so $|\bd S|\geq d_{\mathrm{min}}\:|S|$. The Rayleigh ratio on the left-hand side of \eqref{eq: useful} is at most $\lambda_{\max}$. Therefore 
\begin{align*}
\lambda_{\max}\geq \frac{n\:d_{\mathrm{min}}}{n-|S|}
\end{align*} 
that is, $|S|\leq  n\:\big(1-d_{\mathrm{min}}/\lambda_{\max}\big)$.
\end{proof}

In practice, the previous two theorems are usually applied to regular graphs. In this case, they can be stated as follows.  

\begin{cor}
For a regular graph, we have: 
\begin{align*}
\iota\leq \frac{n}{1-d/\alpha_{\min}}, \qquad \chi\geq 1-d/\alpha_{\min}
\end{align*}
\end{cor}

For the independence number, we have rewritten the bound of Theorem~\ref{thm: hoffman1} by using the relation $\alpha_{\min}+\lambda_{\max}=d$. A practical motivation is that computations usually provide the adjacency eigenvalues. For the chromatic number, we have used $\alpha_{\max}=d$ in the bound provided by Theorem~\ref{thm: hoffman2}. Observe that, in the regular case, the chromatic bound can be derived from the independence bound by using $\chi\cdot \iota\geq n$.

In general, the independence bound of Theorem~\ref{thm: hoffman1} and the inequality $\chi\cdot \iota\geq n$ imply the following chromatic bound:
\begin{align*}
\chi\geq \frac{1}{1-d_{\mathrm{min}}/\lambda_{\max}}
\end{align*}
Note, however, that this is no better than the adjacency bound of Theorem~\ref{thm: hoffman2}. As $\alpha_{\max}\geq d_{\mathrm{ave}}\geq d_{\min}$, and $\alpha_{\min}+ \lambda_{\max}\geq d_{\min}$, we get
\begin{align*}
1+\frac{\alpha_{\max}}{-\alpha_{\min}}\geq 1+\frac{d_{\min}}{\lambda_{\max}- d_{\min}}=\frac{1}{1-d_{\mathrm{min}}/\lambda_{\max}}.
\end{align*}

\begin{ex} Consider the Paley graph $P(q)$, where $q\equiv 1$ mod $4$. Plugging in $n=q$, $d=\tfrac{1}{2}(q-1)$, and $\alpha_{\min}=-\tfrac{1}{2}(\sqrt{q}+1)$ in the previous corollary, we find that
\begin{align*}
\iota\leq \sqrt{q}\leq \chi.
\end{align*}
These are fairly good bounds. As we have seen in Theorem~\ref{prop: chromatic paley}, equalities hold when $q$ is a square.
\end{ex}

\begin{notes} Theorem~\ref{thm: wilf} is due to Wilf (\emph{The eigenvalues of a graph and its chromatic number}, J. London Math. Soc. 1967). Theorem~\ref{thm: hoffman2} is due to Hoffman (\emph{On eigenvalues and colorings of graphs}, in `Graph Theory and its Applications', Academic Press 1970) but the proof given above is based on an argument of Nikiforov (\emph{Chromatic number and spectral radius}, Linear Algebra Appl. 2007). Hoffman's original approach is that of Exercise~\ref{exer: hoffman}. Theorem~\ref{thm: hoffman1} is also due to Hoffman (unpublished).
\end{notes}

\bigskip
\subsection*{Isoperimetric constant}\label{sec: isoperimetric}\
\medskip
 
The smallest positive eigenvalue of the laplacian, $\lambda_2$, is an indicator of how well a graph is connected. This principle arises from, and is made quantitative by, relations between $\lambda_2$ and various graph-theoretic measures of connectivity. Here we pursue the relation with the isoperimetric constant $\beta$. As we will see, the spectral perspective on $\beta$ turns out to be particularly successful. It is, in fact, one of the most consistent and conceptual instances of understanding combinatorial aspects by spectral methods.

\begin{thm}[Alon - Milman]\label{thm: AM}
We have $\beta\geq \lambda_2/2$.
\end{thm}

\begin{proof}
Recall \eqref{eq: useful}:
\begin{align*}
\frac{\la L f, f\ra}{\la f, f\ra}=\frac{n|\bd S|}{|S||S^c|}
\end{align*}
where $S$ is a proper subset of vertices, and $f$ is the function given by $f\equiv |S^c|$ on $S$, respectively $f\equiv -|S|$ on $S^c$. As $f\perp \ct$, the left-hand side is at least $\lambda_2$. So, if $S$ is such that $|S|\leq n/2$, then
\begin{align*}
\lambda_2\leq \frac{n|\bd S|}{|S||S^c|} \leq 2 \frac{|\bd S|}{|S|} 
\end{align*}
implying that $\beta\geq \lambda_2/2$.
\end{proof}

There is an upper bound of similar flavour, but this is contingent on the existence of a certain kind of eigenfunction. 

\begin{thm}
If there is a laplacian eigenvalue $\lambda$ having a $\{\pm 1\}$-valued eigenfunction, then $\beta\leq \lambda/2$. 
\end{thm}

\begin{proof}
Let $g$ be a $\{\pm 1\}$-valued eigenfunction for $\lambda$, and let $S=\{v: \: g(v)=1\}$. As $g$ is orthogonal to $\mathbb{1}$, we have $|S|=n/2$. Now $\la Lg,g\ra=4|\bd S|$ and $\la g,g \ra=n=2|S|$. Therefore
\begin{align*}
\lambda=\frac{\la Lg,g\ra}{\la g,g \ra}=2\: \frac{|\bd S|}{|S|}\geq 2\beta
\end{align*}
as claimed.
\end{proof}

Note that, if $\lambda$ is to have a $\{\pm 1\}$-valued eigenfunction, then $\lambda$ has to be an even integer. Indeed, in the eigenrelation $\lambda g(v)=\deg(v)g(v)-\sum_{u: u\sim v} g(u)$, the right-hand side is an even integer. Also, $n$ has to be even.

An immediate consequence of the previous two theorems is the following.

\begin{cor}
If $\lambda_2$ admits a $\{\pm 1\}$-valued eigenfunction, then $\beta=\lambda_2/2$.
\end{cor}

This criterion for computing the isoperimetric constant is applicable in many examples. Even for those where we already knew $\beta$, the spectral approach is more satisfactory than our previous combinatorial arguments. One context where we know not only the eigenvalues, but also corresponding eigenfunctions, is that of Cayley graphs and bi-Cayley graphs of abelian groups. Some of the examples that follow will exploit that perspective.

\begin{ex} The Petersen graph has $\lambda_2=2$. A corresponding $\{\pm 1\}$-valued eigenfunction is most easily seen on the standard drawing of the Petersen graph. Put $+1$ on the vertices of the outer pentagon, and $-1$ on the vertices of the inner pentagram. Thus, $\beta=1$.
\end{ex}

\begin{ex} The cube graph $Q_n$ has $\lambda_2=2$. Putting $+1$ on binary strings starting in $0$, respectively $-1$ on binary strings starting in $1$, defines a $\{\pm 1\}$-valued eigenfunction for $\lambda_2$. Hence $\beta=1$. 

In fact, for \emph{any} Cayley graph of $(\Z_2)^n$ we can find a $\{\pm 1\}$-valued eigenfunction for $\lambda_2$. The reason is that $\lambda_2$ admits some non-trivial character as an eigenfunction, and all non-trivial characters of $(\Z_2)^n$ are $\{\pm 1\}$-valued. Thus $\beta=\lambda_2/2$, an integer, for any Cayley graph of $(\Z_2)^n$. 
\end{ex}

\begin{exer}\label{exer: iso halved} (i) Show that the halved cube graph $\tfrac{1}{2}Q_n$ has isoperimetric constant $\beta=n-1$. (ii) For a fixed $n$, consider the decked cube graphs $DQ_n(a)$ for varying $a\in (\Z_2)^n$. Which choice of $a$ maximizes the isoperimetric constant of $DQ_n(a)$?
\end{exer}

\begin{ex} Consider now the twin graphs, $K_4\times K_4$ and the Shrikhande graph. Both have $\lambda_2=4$, and for both graphs we can see a corresponding $\{\pm 1\}$-valued eigenfunction. In $K_4\times K_4$, we alternatively put $+1$ and $-1$ as we circle around the vertex set. In the Shrikhande graph, we put $+1$ on the vertices of the outer octagon, and $-1$ on all the inner vertices.

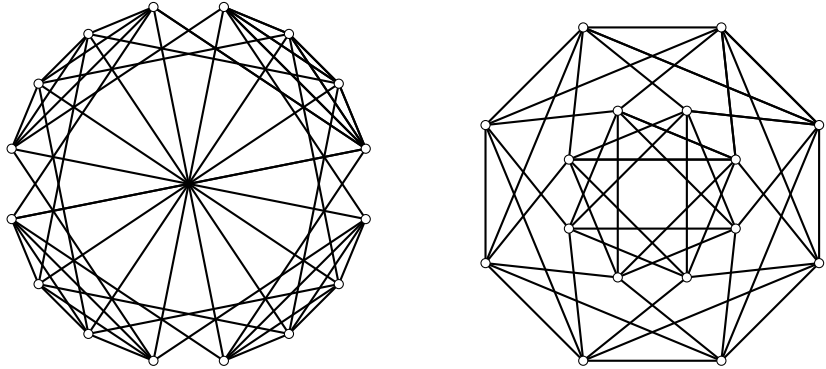
\begin{figure}[!ht]
\bigskip
\GraphInit[vstyle=Simple]
\tikzset{VertexStyle/.style = {
shape = circle,
inner sep = 0pt,
minimum size = .8ex,
draw}}
\begin{minipage}[b]{0.48\linewidth}
\centering
\begin{tikzpicture}[scale=.6]
  \grEmptyCycle[prefix=a, RA=4pt, rotation=11.25]{16}
  \EdgeDoubleMod{a}{16}{0}{1}{a}{16}{4}{1}{16}
  \EdgeDoubleMod{a}{16}{0}{1}{a}{16}{8}{1}{8}
  \EdgeDoubleMod{a}{16}{0}{4}{a}{16}{1}{4}{4}
  \EdgeDoubleMod{a}{16}{0}{4}{a}{16}{2}{4}{4}
  \EdgeDoubleMod{a}{16}{0}{4}{a}{16}{3}{4}{4}
  \EdgeDoubleMod{a}{16}{1}{4}{a}{16}{2}{4}{4}
  \EdgeDoubleMod{a}{16}{1}{4}{a}{16}{3}{4}{4}
  \EdgeDoubleMod{a}{16}{2}{4}{a}{16}{3}{4}{4}
\end{tikzpicture}
\end{minipage}
\begin{minipage}[b]{0.48\linewidth}
\centering
\begin{tikzpicture}[scale=.6]
  \grEmptyCycle[prefix=a, RA=4pt, rotation=22.5]{8}
  \grEmptyCycle[prefix=b, RA=2pt, rotation=22.5]{8}
  \EdgeDoubleMod{a}{8}{0}{1}{a}{8}{1}{1}{8}
  \EdgeDoubleMod{a}{8}{0}{1}{a}{8}{2}{1}{8}
  \EdgeDoubleMod{a}{8}{0}{1}{b}{8}{1}{1}{8}
  \EdgeDoubleMod{a}{8}{1}{1}{b}{8}{0}{1}{8}
  \EdgeDoubleMod{b}{8}{0}{1}{b}{8}{2}{1}{8}
  \EdgeDoubleMod{b}{8}{0}{1}{b}{8}{3}{1}{8}
\end{tikzpicture}
\end{minipage}
\bigskip
\end{figure}

Our seemingly lucky choice of eigenfunctions prompts a legitimate concern: what if the Cheshire Cat takes away our nice picture of, say, the Shrikhande graph? We are then left with a Cayley graph description, inviting us to take an algebraic approach. And we find that, as in the previous example, more is true: \emph{any} Cayley graph of $(\Z_4)^n$ admits a $\{\pm 1\}$-valued eigenfunction for $\lambda_2$, and so $\beta=\lambda_2/2$. Indeed, $\lambda_2$ has a non-trivial character $\chi$ as an eigenfunction. Now $\chi$ takes values in $\{\pm 1, \pm i\}$. The real and imaginary parts of $\chi$, $\mathrm{Re}(\chi)$ and $\mathrm{Im}(\chi)$, are eigenfunctions for $\lambda_2$. So $\mathrm{Re}(\chi)+\mathrm{Im}(\chi)$, which is $\{\pm 1\}$-valued, is an eigenfunction as well. 
\end{ex}

\begin{ex} For the incidence graph $I_n(q)$, we know that 
\begin{align*}
\lambda_2=d-\alpha_2=\frac{q^{n-1}-1}{q-1}-q^{n/2-1}.
\end{align*} Our aim is to show that, for suitable $q$ and $n$, the generic lower bound $\beta\geq \lambda_2/2$ becomes an equality. 

Let us think of $I_n(q)$ as the bi-Cayley graph of $\K^*/\F^*$ with respect to $\{[s]: \Tr(s)=0\}$. Assuming that $q$ is odd, a $\{\pm 1\}$-valued multiplicative character of $\K$ - in fact, the only one - is the quadratic signature $\sigma$. If we also assume that $n$ is even, then $\sigma$ is trivial on $\F^*$ so it defines a $\{\pm 1\}$-valued character, still denoted $\sigma$, of $\K^*/\F^*$. We now refer to the arguments of Theorem~\ref{thm: mock Cayley} and Example~\ref{ex: compute inc}. As $\alpha_\sigma$ is real, in fact integral, we have that the adjacency eigenvalue $|\alpha_\sigma|=\pm \alpha_\sigma$ has an eigenfunction given by $\sigma$ on the black vertices and $\pm \sigma$ on the white vertices. In other words, the second largest eigenvalue $\alpha_2=q^{n/2}-1$ has a $\{\pm 1\}$-valued eigenfunction. We conclude that, for $q$ odd and $n$ even, the incidence graph $I_n(q)$ has $\beta=\lambda_2/2$. As a concrete example, $I_4(q)$ has $\beta=q^2+1$.
\end{ex}

The next result is an upper bound for $\beta$ in terms of $\lambda_2$ and the maximal degree. 

\begin{thm}[Dodziuk]\label{thm: Dodziuk-Mohar}
We have $\beta\leq \sqrt{2d\lambda_2}$.
\end{thm}

\begin{proof} We proceed in three steps. The first step is a functional inequality involving $\beta$. The second step turns the $\ell^1$ inequality of the first step into an $\ell^2$ inequality. In the third step, we apply the inequality established in the second step to the positive part of an eigenfunction for $\lambda_2$.

\smallskip
\emph{Step 1.} Let $f$ be a non-negative function on the graph, supported on no more than half the vertices. Then:
\begin{align*}
\beta\: \sum_{v} f(v)\leq \sum_{\{u,v\}\in E}|f(u)-f(v)|
\end{align*}
Indeed, let $0=f_0<f_1<\ldots<f_s$ be the values taken by $f$. We write
\begin{align*}
\sum_{\{u,v\}\in E}|f(u)-f(v)|=\sum_{k\geq 1} N_k f_k.
\end{align*}
The coefficients are given by
\begin{align*}
N_k=E\big(\{f=f_k\},\{f<f_k\}\big)-E\big(\{f=f_k\}, \{f>f_k\}\big)
\end{align*}
where, say, $E\big(\{f=f_k\},\{f>f_k\} \big)$ denotes the number of edges between the vertex subsets $\{v: f(v)=f_k\}$ and $\{v: f(v)>f_k\}$. We then have
\begin{align*}
N_k &=E\big(\{f\geq f_k\},\{f<f_k\}\big)-E\big(\{f\leq f_k\}, \{f>f_k\}\big)\\
&=|\bd \{f\geq f_k\}|-|\bd \{f\geq f_{k+1}\}|
\end{align*}
so
\begin{align*}
\sum_{k\geq 1} N_k f_k&=\sum_{k\geq 1}\big(|\bd \{f\geq f_k\}|-|\bd \{f\geq f_{k+1}\}|\big)\: f_k\\
& =\sum_{k\geq 1} |\bd \{f\geq f_k\}|\: (f_k-f_{k-1}).
\end{align*}
At this point, we can bring in $\beta$. For each $k\geq 1$, the set $\{f\geq f_k\}$ contains no more than half the vertices, so $|\bd \{f\geq f_k\}|\geq \beta\: |\{f\geq f_k\}|$. Therefore
\begin{align*}
\sum_{k\geq 1} N_k f_k&\geq \beta \sum_{k\geq 1} |\{f\geq f_k\}|\: (f_k-f_{k-1})\\
&=\beta \sum_{k\geq 1} \big(| \{f\geq f_k\}|-| \{f\geq f_{k+1}\}|\big)\: f_k\\
&= \beta \sum_{k\geq 1}  |\{f=f_k\}|\: f_k=\beta \sum f(v)
\end{align*}
as claimed.

\smallskip
\emph{Step 2.} Let $f$ be a real-valued function, supported on no more than half the vertices. Then:
\begin{align*}
\beta^2\: \la f,f\ra\leq 2d\: \la Lf,f\ra
\end{align*}
Indeed, the bound from Step 1, applied to $f^2$, says that
\begin{align*}
 \beta\: \sum_{v} f(v)^2\leq \sum_{\{u,v\}\in E} \big|f^2(u)-f^2(v)\big|.
\end{align*}
Now $|f^2(u)-f^2(v)|=|f(u)-f(v)|\cdot |f(u)+f(v)|$ and so, by Cauchy - Schwarz, we have:
\begin{align*}
\Big(\sum_{\{u,v\}\in E}\big|f^2(u)-f^2(v)\big|\Big)^2\leq \Big(\sum_{\{u,v\}\in E}\big|f(u)-f(v)\big|^2\Big)\: \Big(\sum_{\{u,v\}\in E} \big|f(u)+f(v)\big|^2\Big)
\end{align*}
The first factor on the right-hand side is $\la Lf,f\ra$. We bound the second factor as follows:
\begin{align*}
\sum_{\{u,v\}\in E} \big|f(u)+f(v)\big|^2\leq 2 \sum_{\{u,v\}\in E} \big(f(u)^2+f(v)^2\big)=2\:\sum_{v} \mathrm{deg}(v)\: f(v)^2\leq 2d\: \sum_{v} f(v)^2
\end{align*}
The inequalities combine to give $\beta^2\: \la f,f\ra^2\leq 2d\: \la Lf,f\ra \: \la f,f\ra$. The claim follows after cancelling $\la f,f\ra$.

\smallskip
\emph{Step 3.} Let $g$ be an eigenfunction for $\lambda_2$. By taking the real or the imaginary part, we may assume that $g$ is real-valued. Let us refer to a vertex $v$ as positive (or negative) if $g(v)>0$ (respectively $g(v)<0$). As $g$ is orthogonal to $\ct$, there are both positive and negative vertices. Up to replacing $g$ by $-g$, we may assume that no more than half the vertices are positive. Let $f$ be the positive part of $g$, that is the function defined by $f=g$ on positive vertices, and $f=0$ elsewhere. Then $f(v)\geq g(v)$ for every vertex $v$, and so $Af(v)\geq Ag(v)$ for every vertex $v$. At a \emph{positive} vertex $v$ we have 
\begin{align*}
Lf(v)=\mathrm{deg}(v)\:f(v)-Af(v)\leq \mathrm{deg}(v)\:g(v)-Ag(v)=Lg(v).
\end{align*} 
As $f$ is supported on the positive vertices, where it takes positive values by definition, we have
\begin{align*}
\la Lf,f\ra\leq \la Lg,f\ra=\lambda_2 \:\la g,f\ra=\lambda_2 \:\la f,f\ra.
\end{align*}
Combining this bound with the bound from Step 2, we conclude that $\beta^2\leq 2d\lambda_2$.
\end{proof}

The spectral lower bound for $\beta$ (Theorem~\ref{thm: AM}) had a very easy proof, and turned out to be very effective in practice. So it is quite ironic that the spectral upper bound of Theorem~\ref{thm: Dodziuk-Mohar}, which is significantly harder to prove, is rarely effective in practice. Tested on some families of graphs, it quickly appears to be a fairly bad bound. Both the Paley graphs $P(q)$ and the incidence graphs $I_3(q)$ have $\alpha_2\sim C\sqrt{d}$ as $q\to \infty$, for some numerical constant $C$. Thus $\lambda_2=d-\alpha_2\sim d$, and $\sqrt{2d\lambda_2}\sim \sqrt{2} d$ as $q\to \infty$. This is worse than the trivial upper bound $d$. On the cubes $Q_n$, the performance is also poor: the spectral upper bound is $\sqrt{2d\lambda_2}=2\sqrt{d}$, whereas $\beta=1$. Here, at least, we are beating the combinatorial upper bound of roughly $d/2$.   

On the other hand, the spectral upper bound is effective for the cycles $C_n$: here $\beta\sim 4/n$ and $\sqrt{2d\lambda_2}=4\sin(\pi/n)\sim 4\pi/n$. Unlike the previous families, the cycles have constant degree. 

Recall that an expander family is a $d$-regular family of graphs whose isoperimetric constant is uniformly bounded away from $0$. The conceptual importance of Theorem~\ref{thm: Dodziuk-Mohar} is that, in conjunction with Theorems~\ref{thm: AM}, it  implies the following: a $d$-regular family is an expander family if and only if it has a uniform spectral gap. In laplacian terms, this means that $\lambda_2$ is uniformly bounded away from $0$; in adjacency terms, $\alpha_2$ is uniformly bounded away from $d$. Note, however, that there is a quantitative difference between the combinatorial and the spectral approach to expanders.

\begin{notes}
Fiedler (\emph{Algebraic connectivity of graphs}, Czechoslovak Math. J. 1973) was the first to seize the importance of $\lambda_2$ in relation to graph connectivity. He called $\lambda_2$ the `algebraic connectivity' of a graph. Fiedler's starting point was the observation that, in the context of possibly disconnected graphs, the connectivity of a graph is equivalent to the condition that $\lambda_2>0$.

Theorems~\ref{thm: AM} and ~\ref{thm: Dodziuk-Mohar} are graph-theoretic analogues of results due to Cheeger (\emph{A lower bound for the smallest eigenvalue of the Laplacian}, in `Problems in analysis', Princeton Univ. Press 1970) and Buser (\emph{A note on the isoperimetric constant}, Ann. Sci. \'Ecole Norm. Sup. 1982), relating the isoperimetric constant of a compact Riemannian manifold $M$ to the first positive eigenvalue of the laplacian operator on $L^2(M)$.  Theorem~\ref{thm: AM} is implicitly due to Alon and Milman (\emph{$\lambda_1$, isoperimetric inequalities for graphs, and superconcentrators}, J. Combin. Theory Ser. B 1985). Theorem~\ref{thm: Dodziuk-Mohar}, the analogue of Cheeger's result, is due to Dodziuk (\emph{Difference equations, isoperimetric inequality and transience of certain random walks}, Trans. Amer. Math. Soc. 1984). See also Mohar (\emph{Isoperimetric numbers of graphs}, J. Combin. Theory Ser. B 1989) for an improvement.

Expander families do not come easily. The first explicit construction is due to Margulis (\emph{Explicit constructions of expanders}, Problems Inform. Transmission 1973) and it relies on a rather sophisticated phenomenon, namely Kazhdan's property (T). The spectral perspective on expanders originates here.

Another early explicit construction of an expander family is due to Buser (\emph{On the bipartition of graphs}, Discrete Appl. Math. 1984). Buser's construction also depends on a spectral perspective, but in a rather different way: as he puts it, his proof is ``\emph{very} unorthodox--from the point of view of graph theory--and is in the realm of the spectral geometry of the Laplace operator on Riemann surfaces''.

Remarkably, there are expander families consisting of $d$-regular graphs with the following property: 
\begin{quote}
$(*)$ all eigenvalues different from $\pm d$ lie in the interval $[-2\sqrt{d-1},2\sqrt{d-1}]$
\end{quote}
In light of the Alon - Boppana phenomenon, one cannot pack the eigenvalues of a $d$-regular family in a smaller interval. On the other hand, the existence of expander families satisfying $(*)$ shows that the Alon - Boppana phenomenon is sharp.

Expander families satisfying $(*)$ were first constructed by Lubotzky, Phillips, Sarnak (\emph{Ramanujan graphs}, Combinatorica 1988) and, independently, Margulis (\emph{Explicit group-theoretic constructions of combinatorial schemes and their applications in the construction of expanders and concentrators}, Problems Inform. Transmission 1988). The graphs in question are Cayley graphs of $\mathrm{PSL}_2(\Z_p)$ or $\mathrm{PGL}_2(\Z_p)$, of degree $\ell+1$. Here $p$ and $\ell$ are distinct primes congruent to $1$ mod $4$, and the choice of $\mathrm{PSL}_2$ or $\mathrm{PGL}_2$ depends on the sign of the Legendre symbol $(p/\ell)=(\ell/p)$. The proof of $(*)$ for these specific graphs rests on a deep number-theoretical fact, the solution of the so-called Ramanujan conjecture. For this reason, expander families satisfying $(*)$ have come to be known as Ramanujan expander families. Simpler and more flexible constructions of Ramanujan expander families are currently known. A very recent and most notable result of Marcus, Spielman, and Srivastava (\emph{Interlacing families I: Bipartite Ramanujan graphs of all degrees}, Ann. of Math. 2015) says that Ramanujan expander families can be constructed for all degrees greater than $2$.

\end{notes}


\bigskip
\subsection*{Edge-counting}\label{sec: edgecount}\
\medskip
 
Let $S$ and $T$ be two sets of vertices in a regular graph. Denote by $e(S,T)$ the number of edges connecting vertices in $S$ to vertices in $T$, with the convention that edges having both endpoints in $S\cap T$ are counted twice. We want to estimate the local quantity $e(S,T)$ in terms of the size of $S$, the size of $T$, and global information on the graph. It turns out that the spread of the non-trivial eigenvalues of $X$ can be used to control the deviation of $e(S,T)$ from its expected value, $(\textrm{edge density})\cdot|S||T|$.

We first do this in the bipartite case.

\begin{thm}\label{thm: bipartite mixing lemma}
Let $X$ be a bipartite, $d$-regular graph of half-size $m$, and let $S$ and $T$ be vertex subsets lying in different sides of the bipartition. Then:
\begin{align}\label{eq: alonchung}
\bigg|e(S,T)-\frac{d}{m}|S| |T|\bigg|\leq \frac{\alpha_2}{m} \sqrt{|S| |T|\big(m-|S|\big)\big( m-|T|\big)} 
\end{align}
In particular, we have
\begin{align}\label{eq: alonchung-k}
\big|e(S,T)-\kappa |S| |T|\big|\leq \alpha_2 \sqrt{|S| |T|}\qquad \textrm{ whenever }\quad\frac{d-\alpha_2}{m}\leq \kappa\leq\frac{d+\alpha_2}{m}.
\end{align}
\end{thm}

The constant $\kappa$ in \eqref{eq: alonchung-k} should be thought of as a rounding constant. A generic choice is $\kappa=d/m$, but more convenient choices are usually available in applications.

\begin{proof}
Note first that, for the characteristic functions of $S$ and $T$, we have
\begin{align*}
\la A\: \ct_S, \ct_T\ra= \sum_{u\sim v} \ct_S(u)\:\ct_T(v)= e(S,T).
\end{align*}
Next, we seek a spectral understanding of the left-hand side. Let 
\begin{align*}
d=\alpha_1>\alpha_2\geq\ldots\geq \alpha_{n-1}=-\alpha_2>\alpha_n=-d
\end{align*} 
be the adjacency eigenvalues of $X$, where $n=2m$ denotes, as usual, the number of vertices. Let also $\{f_k\}$ be an orthonormal basis consisting of normalized adjacency eigenfunctions. Thus $f_1$ is constant equal to $1/\sqrt{n}$ while $f_n$ is $1/\sqrt{n}$ on, say, black vertices, and $-1/\sqrt{n}$ on white vertices. Expand $\ct_S$ and $\ct_T$ in the eigenbasis as $\ct_S=\sum s_k f_k$, respectively $\ct_T=\sum t_k f_k$. Then: 
\begin{align*}
\la A\:\ct_S, \ct_T\ra=\sum_{k=1}^n \alpha_k s_k \overline{t}_k
\end{align*}
We separate the terms corresponding to the extremal eigenvalues $\pm d$, and the terms corresponding to the intermediate eigenvalues. If, say, $S$ consists of black vertices and $T$ consists of white vertices, then the extremal coefficients are
\begin{align*}
s_1=\frac{1}{\sqrt{n}}|S|=s_n, \qquad  t_1=\frac{1}{\sqrt{n}}|T|=-t_n
\end{align*}
and so
\begin{align*}
\alpha_1 s_1 \overline{t}_1+\alpha_n s_n \overline{t}_n=\frac{d}{m}|S| |T|. 
\end{align*}
For the intermediate coefficients, we note that
\begin{align*}
\sum_{k= 2}^{n-1} |s_k|^2=\la\ct_S, \ct_S\ra-|s_1|^2-|s_n|^2=|S|-\frac{|S|^2}{m}=\frac{1}{m}|S|(m-|S|)
\end{align*}
and similarly for the coefficients of $T$. Therefore, by Cauchy - Schwarz we have that
\begin{align*}
\sum_{k=2}^{n-1} |s_k| |t_k|\leq \frac{1}{m} \sqrt{|S| |T|(m-|S|)(m- |T|)}.
\end{align*}
Now
\begin{align*}
\bigg|e(S,T)-\frac{d}{m}|S| |T|\bigg|=\bigg|\sum_{k\geq 2}^{n-1} \alpha_k s_k t_k\bigg|\leq \sum_{k\geq 2}^{n-1} |\alpha_k| |s_k| |t_k|\leq \alpha_2 \sum_{k\geq 2}^{n-1} |s_k| |t_k|
\end{align*}
and \eqref{eq: alonchung} follows. Finally, using $\sqrt{(m-|S|)(m-|T|)}\leq m-\sqrt{|S||T|}$ we get
\begin{align*}
 \frac{d+\alpha_2}{m}|S||T|-\alpha_2\sqrt{|S||T|}\leq e(S,T)\leq  \frac{d-\alpha_2}{m}|S||T|+\alpha_2\sqrt{|S||T|}
\end{align*}
which implies \eqref{eq: alonchung-k}.
\end{proof}

The non-bipartite case can be easily derived. 

\begin{cor}\label{cor: mixing lemma}
Let $X$ be a non-bipartite, $d$-regular graph of size $n$, and let $S$ and $T$ be vertex subsets. Then:
\begin{align*}
\bigg|e(S,T)-\frac{d}{n}|S| |T|\bigg|\leq \frac{\alpha}{n} \sqrt{|S| |T||S^c| |T^c|}, \qquad \alpha:=\max\{\alpha_2, -\alpha_n\}
\end{align*}
\end{cor}

\begin{proof}
Let $X'$ be the bipartite double of $X$. Then $\alpha=\max\{\alpha_2, -\alpha_n\}$ is the second largest eigenvalue of $X'$. Viewing $S$ and $T$ as vertex subsets of $X'$ of different colours, we may apply \eqref{eq: alonchung} to get the claimed bound.
\end{proof}

Theorem~\ref{thm: bipartite mixing lemma} provides a versatile method with numerous combinatorial applications. 

\begin{ex}\label{app: sum-product} Let $\F$ be a finite field with $q$ elements, and let $A,B,C,D\subseteq \F$. The question we wish to address is that of estimating the number of solutions $(a,b,c,d)\in A\times B\times C\times D$ to the following two equations: $a+b=cd$, respectively $ab+cd=1$.

The hypersurface $\{x+y=zt\}\subseteq \F^4$ has $q^3=\frac{1}{q}|\F^4|$ points, while the hypersurface $\{xy+zt=1\}\subseteq \F^4$ has $(q^2-1)(q-1)\sim \frac{1}{q}|\F|^4$ points. So, for both equations we would expect about $\frac{1}{q}|W|$ solutions in the `window' $W:=A\times B \times C\times D\subseteq \F^4$. We will now show that the error in making this estimate is on the order of $\sqrt{|W|}$. More precisely, we show that the number of restricted solutions $N_W$ to either one of the two equations satisfies the following bound:
\begin{align*}
\Big|N_W-\frac{1}{q}|W| \Big|\leq \sqrt{q|W|}
\end{align*}

We work in the incidence graph $I_3(q)$, which has $m=q^2+q+1$, $d=q+1$, $\alpha_2=\sqrt{q}$ so we may use $\kappa=1/q$ in \eqref{eq: alonchung-k}. 

For the equation $a+b=cd$, let $S$ be the set of black vertices of the form $[(a,-1,c)]_\bl$ with $a\in A$ and $c\in C$, and let $T$ be the set of white vertices of the form $[(-1,b,d)]_\wh$ with $b\in B$ and $d\in D$. Thus $|S||T|=|A||C||B||D|=|W|$. Furthermore, there is an edge between $[(a,-1,c)]_\bl$ and $[(-1,b,d)]_\wh$ precisely when $a+b=cd$, so $e(S,T)$ counts the number of solutions in $W=A\times B \times C\times D$ to the equation $a+b=cd$. An application of \eqref{eq: alonchung-k} yields the desired estimate.

For the equation $ab+cd=1$, take $S$ to be the set of black vertices of the form $[(1,a,c)]_\bl$ with $a\in A$ and $c\in C$, and $T$ to be the set of white vertices of the form $[(-1,b,d)]_\wh$ with $b\in B$ and $d\in D$.

Instead of the incidence graph $I_3(q)$, we could have used just as well the sum-product $SP(q)$ to handle the equation $a+b=cd$. For the equation $ab+cd=1$, however, no perfectly tailored graph is apparent.
\end{ex}

\begin{ex} Let $\F$ be a field with $q$ elements, where $q$ is odd. Assume that $A,B\subseteq \F$ have the property that the sum-set $A+B=\{a+b: a\in A,b\in B\}$ consists of squares. Then $|A||B|\leq q$.

Consider the bi-Cayley graph of the additive group of $\F$ with respect to $\F^2$, the subset of squares. Recall, this is a bipartite graph whose vertex set consists of two copies of $\F$, in which an edge connects $x_\bl$ to $y_\wh$ whenever $x+y$ is a square. The given subsets define a copy of the complete bipartite graph $K_{|A|,|B|}$, by viewing $A$ and $B$ as vertex subsets of different colours. 

The next step is a rather general application of \eqref{eq: alonchung-k}: if a regular bipartite graph contains the complete bipartite graph $K_{s,t}$, then
\begin{align}\label{eq: bipartite bound}
\sqrt{st}\leq \frac{m}{1+(m-d)/\alpha_2}.
\end{align}
Our graph has $m=q$ and $d=\frac{1}{2}(q+1)$. We claim that $\alpha_2\leq \frac{1}{2}(\sqrt{q}+1)$. Using these values in ~\eqref{eq: bipartite bound}, we obtain $st\leq q$ after a short and pleasant computation.

The bi-Cayley graph we are considering is a close relative of the bi-Paley graph. In order to estimate the adjacency eigenvalues, we adapt the arguments of Example~\ref{ex: paley via cayley}. Let $\psi$ be a non-trivial additive character. Then
\begin{align*}
\alpha_\psi=\sum_{s\in \F^{2}} \psi(s)=1+\sum_{s\in \F^{*2}} \psi(s)=\frac{1}{2}\big(G(\psi,\sigma)+1\big).
\end{align*}
Hence
\begin{align*}
|\alpha_\psi|\leq \frac{1}{2}\big(|G(\psi,\sigma)|+1\big)= \frac{1}{2}(\sqrt{q}+1)
\end{align*}
as claimed.
\end{ex}

\begin{exer}\label{exer: lines meeting planes} Let $\F$ be a field with $q$ elements, and assume that a proportion $c\in (0,1)$ of lines, respectively planes in $\F^3$, is selected. Give a lower bound for $c$, in terms of $q$, guaranteeing that some selected line is contained in some selected plane. 
\end{exer}

The edge-counting estimates from \ref{thm: bipartite mixing lemma} can be easily extended, mutatis mutandis, to path-counting. For two vertex subsets $S$ and $T$ in a graph,  we denote by $p_\ell(S,T)$ the number of paths of length $\ell$ connecting vertices in $S$ to vertices in $T$.

\begin{thm}\label{thm: path-count}
Let $X$ be a bipartite $d$-regular graph of half-size $m$, and let $S$ and $T$ be vertex subsets. Assume that either $\ell$ is odd, and $S$ and $T$ have different colours, or $\ell$ is even, and $S$ and $T$ have the same colour. Then:
\begin{align*}
\bigg|p_\ell(S,T)-\frac{d^\ell}{m}|S| |T|\bigg|\leq \frac{\alpha_2^\ell}{m} \sqrt{|S| |T| \big(m-|S|\big) \big(m-|T|\big)}
\end{align*}
\end{thm}

\begin{cor}
Let $X$ be a non-bipartite $d$-regular graph of size $n$, and let $S$ and $T$ be vertex subsets. Then:
\begin{align*}
\bigg|p_\ell(S,T)-\frac{d^\ell}{n}|S| |T|\bigg|\leq \frac{\alpha^\ell}{n} \sqrt{|S| |T| |S^c| |T^c| }, \qquad \alpha:=\max\{\alpha_2, -\alpha_n\}
\end{align*}
\end{cor}

\begin{exer}\label{exer: chung}
Prove the following diameter bounds. If $X$ is a $d$-regular graph of size $n$, then
\begin{align*}
\delta\leq \frac{\log (n-1)}{\log (d/\alpha)}+1, \qquad \alpha:=\max\{\alpha_2,-\alpha_n\}. 
\end{align*}
If $X$ is a bipartite $d$-regular graph of half-size $m$, then
\begin{align*}
\delta\leq \frac{\log (m-1)}{\log (d/\alpha_2)}+2. 
\end{align*}
\end{exer}

\begin{notes}
Theorem~\ref{thm: bipartite mixing lemma} and Corollary~\ref{cor: mixing lemma} are slight variations on a result due to Alon and Chung (\emph{Explicit construction of linear sized tolerant networks}, Discrete Math. 1988). A conceptual precursor is Thomason's notion of jumbled graph (\emph{Pseudorandom graphs}, in 'Random graphs '85', North-Holland 1987).

The estimates in Example~\ref{app: sum-product} are inspired by, and partially improve, estimates due to Gyarmati and S\'ark\"ozy (\emph{Equations in finite fields with restricted solution sets II (Algebraic equations)}, Acta Math. Hungar. 2008). 

Exercise~\ref{exer: chung} is a result due to Chung (\emph{Diameters and eigenvalues}, J. Amer. Math. Soc. 1989).
\end{notes}

\newpage
\section*{-- Further reading --}
\bigskip

\begin{itemize}[itemsep=12pt, leftmargin=0pt]

\item[] Noga Alon\\
\emph{Tools from higher algebra}, in `Handbook of combinatorics', 1749--1783, Elsevier 1995

\item[] L\'aszl\'o Babai, Peter Frankl\\
\emph{Linear algebra methods in combinatorics}, Preliminary version, University of Chicago 1992

\item[] Simeon Ball\\
\emph{Finite geometry and combinatorial applications}, London Mathematical Society Student Texts no. 82, Cambridge University Press 2015

\item[] B\'ela Bollob\'as\\
\emph{Modern graph theory}, Graduate Texts in Mathematics no. 184, Springer 1998

\item[] Andries E. Brouwer, Willem H. Haemers\\ 
\emph{Spectra of graphs}, Universitext, Springer 2012

\item[] Peter J. Cameron, Jack H. van Lint\\
\emph{Designs, graphs, codes and their links}, London Mathematical Society Student Texts no. 22, Cambridge University Press 1991

\item[] Fan Chung\\
\emph{Spectral graph theory}, CBMS Series in Mathematics no. 92, American Mathematical Society 1997

\item[] Giuliana Davidoff, Peter Sarnak, Alain Valette\\ 
\emph{Elementary number theory, group theory, and Ramanujan graphs}, London Mathematical Society Student Texts no. 55, Cambridge University Press 2003

\item[] Chris Godsil, Gordon Royle\\
\emph{Algebraic graph theory}, Graduate Texts in Mathematics no. 207, Springer 2001

\item[] Shlomo Hoory, Nathan Linial, Avi Wigderson\\
\emph{Expander graphs and their applications}, Bull. Amer. Math. Soc. 43 (2006), no. 4, 439--561

\item[] Ji\v{r}\'i Matou\v{s}ek\\ 
\emph{Thirty-three miniatures. Mathematical and algorithmic applications of linear algebra}, Student Mathematical Library no. 53, American Mathematical Society 2010

\item[] Daniel Spielman\\
\emph{Spectral graph theory}, in `Combinatorial scientific computing', 495--524, CRC Press 2012

\end{itemize}

\end{document}